\documentclass[11pt,reqno]{amsart}
                \usepackage[margin=2.75cm]{geometry}

\usepackage{amssymb,amsfonts}
\usepackage[all,arc]{xy}
\usepackage{enumerate}
\usepackage{mathrsfs}
\usepackage{color}   %May be necessary if you want to color links
\usepackage{hyperref}
\hypersetup{
    colorlinks=true, %set true if you want colored links
    linktoc=all,     %set to all if you want both sections and subsections linked
    linkcolor=blue,  %choose some color if you want links to stand out
}
\usepackage{amsthm}
\usepackage{amsmath}
\usepackage{graphicx}
\usepackage{wasysym}
\usepackage{pgf,tikz}
\usetikzlibrary{positioning}
\usepackage{calc}
\usepackage{xypic}
\usepackage{imakeidx}
\usepackage[toc,page]{appendix}
\usepackage{amscd}
\usepackage{tikz-cd}

\makeatletter
\def\Ddots{\mathinner{\mkern1mu\raise\p@
\vbox{\kern7\p@\hbox{.}}\mkern2mu
\raise4\p@\hbox{.}\mkern2mu\raise7\p@\hbox{.}\mkern1mu}}
\makeatother

\newcommand{\cc}{\mathbb C}

\newcommand{\zz}{\mathbb Z}

\newcommand{\A}{\mathbb A}

\newcommand{\la}{\langle}
\newcommand{\ra}{\rangle}
\newcommand{\lra}{\longrightarrow}
\newcommand{\hra}{\hookrightarrow}
\newcommand{\bs}{\backslash}

\newcommand{\al}{\alpha}
\newcommand{\be}{\beta}
\newcommand{\ga}{\gamma}

\newcommand{\ep}{\epsilon}
\newcommand{\vp}{\varphi}
\newcommand{\lam}{\lambda}
\newcommand{\Lam}{\Lambda}

\DeclareMathOperator{\ind}{ind}
\DeclareMathOperator{\G}{G}
\DeclareMathOperator{\GL}{GL}
\DeclareMathOperator{\Hom}{Hom}
\DeclareMathOperator{\Ind}{Ind}

\DeclareMathOperator{\PGL}{PGL}

\DeclareMathOperator{\OO}{O}
\DeclareMathOperator{\SO}{SO}
\DeclareMathOperator{\SL}{SL}
\DeclareMathOperator{\Mp}{Mp}
\DeclareMathOperator{\Sp}{Sp}
\DeclareMathOperator{\Sym}{Sym}
\DeclareMathOperator{\Lie}{Lie}
\DeclareMathOperator{\sspan}{span}
\DeclareMathOperator{\Mat}{Mat}
\DeclareMathOperator{\diag}{diag}
\DeclareMathOperator{\Stab}{Stab}
\DeclareMathOperator{\Spec}{Spec}
\DeclareMathOperator{\Rep}{Rep}

\newcommand{\fg}{\mathfrak g}

\newcommand{\calo}{\mathcal{O}}

\newcommand{\OF}{\mathfrak{O}_F}
\newcommand{\GF}{\overline{\Sp}_{2r}(F)}
\newcommand{\GA}{\overline{\Sp}_{2r}(\A)}
\newcommand{\Gn}{\overline{\mathrm{G}}^{(n)}}

\newtheorem{Thm}{Theorem}[section]
\newtheorem{Prop}[Thm]{Proposition}
\newtheorem{Lem}[Thm]{Lemma}
\newtheorem{Cor}[Thm]{Corollary}

\newtheorem{Conj}[Thm]{Conjecture}

\theoremstyle{definition}
\newtheorem{Def}[Thm]{Definition}

\theoremstyle{remark}
\newtheorem{Rem}[Thm]{Remark}

\theoremstyle{definition}

\title[Theta liftings, CAP representations, and Arthur Parameters]{A Generalized Theta lifting, CAP representations, and Arthur parameters}

\author{Spencer Leslie}

\date\today

\address{Department of Mathematics, Boston College, Chestnut Hill MA 02467-3806, USA}

\email{lesliew@bc.edu}

\subjclass[2010]{Primary 11F70; Secondary 11F30, 22E50, 22E55}
\keywords{Automorphic Representation, Brylinski-Deligne Covering Group, Theta Correspondence, Unipotent orbit, Arthur Parameter}

\begin{document}

\begin{abstract}
 We study a new lifting of automorphic representations using the theta representation $\Theta$ on the $4$-fold cover of the symplectic group,  $\overline{\Sp}_{2r}(\A)$. This lifting produces the first examples of CAP representations on higher-degree metaplectic covering groups. Central to our analysis is the identification of the maximal nilpotent orbit associated to $\Theta$.
 
 We conjecture a natural extension of Arthur's parameterization of the discrete spectrum to $\overline{\Sp}_{2r}(\A)$.  Assuming this, we compute the effect of our lift on Arthur parameters and show that the parameter of a representation in the image of the lift is non-tempered. We conclude by relating the lifting to the dimension equation of Ginzburg to predict the first non-trivial lift of a generic cuspidal representation of $\overline{\Sp}_{2r}(\A)$.
%[[To remove: made edits in proof of Cor 8.5 per email with Sol]]
\end{abstract}

\maketitle
\setcounter{tocdepth}{1}
\tableofcontents

\section{Introduction}\label{sec: Intro}

Let $K$ be a number field, and let $\G$ be a reductive group over $K$. In \cite{A}, Arthur conjectures a precise decomposition of the discrete spectrum $L^2_{disc}(\G(K)\backslash \G(\A))$ in terms of parameters
\[
\Psi:L_K\times \SL_2(\cc)\to {}^L\G.
\]
Here $L_K$ is the conjectural Langlands group attached to $K$, and ${}^L\G=\G^\vee\rtimes W_K$ is the Langlands $L$-group associated to $\G$. According to Arthur, understanding parameters $\Psi$ which are nontrivial when restricted to $\SL_2(\cc)$ is related to the classification of non-tempered cuspidal representations of $\G(\A)$. These representations, which are called CAP representations, provide counterexamples to the naive generalization of Ramanujan's conjecture to groups other than $\GL_n$. 

Assuming $K$ contain the $n^{th}$ roots of unity, Brylinski and Deligne \cite{BD} have classified degree $n$ topological extensions $\overline{\G}_\A$ of $\G(\A)$ by $\mu_n(K)$ which arise from the algebraic K-theory of the field $K$. This classification of (BD-)covering groups also works for local fields $F$ containing the $n^{th}$ roots of unity.
Given such  a BD-covering group $\overline{\G}$, we may associate to it a natural complex dual group $\overline{\G}^\vee$  which can be distinct from the dual group of $\G$. Recently, Weissman has put forward a promising candidate, ${}^L\overline{G},$ for the $L$-group of such covering groups (see \cite{W2} and \cite{W3}) as an extension
\begin{align}\label{eqn: L-group}
1\lra \overline{\G}^\vee \lra {}^L\overline{G}\lra W_K\lra1.
\end{align}
It is thus important to test if this $L$-group allows us to extend Langlands functoriality and Arthur's conjectures to the case of BD-covering groups. For example, the construction of the metaplectic tensor product has been shown to be functorial by Gan \cite{Gan}. 

Historically, the metaplectic $2$-fold cover of the symplectic group, $\Mp(W)$, where $(W,\la\cdot,\cdot\ra)$ is a symplectic vector space over $K$, was the first example of a BD-covering group to be studied. While much is known about the representation theory of $\Mp(W)$ (in particular, the local Langlands correspondence has been extended to this group via work of Adams-Barbasch \cite{AB1}, \cite{AB2} in the real case, and Gan-Savin \cite{GS} in the nonarchimedean case), few explicit results towards understanding the representation theory of higher degree BD-covering groups are known.

Fix an additive character $\psi:K\backslash \A\to \cc^\times$. A primary tool in the study of the representation theory of $\Mp(W)$ has been the theta correspondence, utilizing the classical oscillator representation $\omega_\psi$. For example, such liftings were recently used by Gan-Ichino to extend the notion of Arthur parameter to $\Mp(W)_\A$ \cite{GI}. Motivated by the work of Bump-Friedberg-Ginzburg \cite{BFG2}, one may attempt to study generalizations of theta liftings on BD-covering groups in the hope of gaining insight into the representation theory of higher degree covers. 

In this direction, assume that the number field $K$ contains the $4^{th}$ roots of unity. In this paper, we consider the $4$-fold cover $\overline{\Sp}(W)_\A$ of $\Sp(W)_\A$ and a generalization $\Theta$ of the oscillator representation $\omega_\psi$. This representation is isomorphic to a residual representation obtained from a Borel Eisenstein series, a process which also gives a construction of $\omega_\psi$. Such \emph{(generalized) theta representations} have been studied on several different covering groups (see \cite{KP}, \cite{BFG1}, \cite{FG},\cite{Gao}, and \cite{C}) and may be defined for a large class of BD-covering groups. 

One of the most important properties of $\omega_\psi$ is that it has Schr\"{o}dinger-type models. This both implies that $\omega_\psi$ is a minimal representation (in the sense of minimal Gelfand-Kirillov dimension) and enables explicit computations. A crucial distinction in our case is the lack of such a model for $\Theta$ or its local components. Moreover, $\Theta$ is not even a minimal representation. Nevertheless, \emph{in the case of the $4$-fold cover of $\Sp(W)_\A$}, $\Theta$ is sufficiently small in a sense to be made precise below. This allows it to effectively function as a kernel to construct a lifting of automorphic representations.

 More specifically, to any decomposition $W = W_1\oplus W_2$ into the sum of two lower dimensional symplectic spaces, we obtain an embedding
\[
\Sp(W_1)\times\Sp(W_2)\hra \Sp(W).
\] 
We study the problem (both locally and globally) of restricting the theta representation $\Theta_W$ to the subgroup $\overline{\Sp}(W_1)\times_{\mu_4}\overline{\Sp}(W_2)$. In the global setting, this problem is best viewed in terms of theta liftings: given an irreducible cuspidal automorphic representation $(\pi,V)$ of $\overline{\Sp}({W_1})_\A$, consider the space $\Theta_W(\pi)$, defined as the Hilbert space closure of the space of functions on $\overline{\Sp}({W_2})_\A$ of the form
\[
F(h)=\displaystyle \int_{[\Sp(W_1)]}\overline{\varphi(g)}\theta(h,g)\,dg,
\]
where $\varphi\in V$ and $\theta\in \Theta_W$. This gives an automorphic representation of $\overline{\Sp}({W_2})_\A$. We study properties of this representation from another perspective.

 Using ideas of Rallis in the theory of the classical theta correspondence \cite{R}, given a symplectic space $W=W_0$, one can form the larger symplectic space
\[
W_n = X_n\oplus W_0\oplus X_n^\ast,
\]
where $X_n$ is totally isotropic and $\dim_K(X_n) =n$, thus giving an embedding $\Sp(W_0)\hra \Sp(W_n)$ for each $n\geq 0$. Setting $V_n = X_n\oplus X_n^\ast$, we may then consider $\Theta_n(\pi)$, the representation of $\overline{\Sp}({V_n})_\A$ formed by the pairing vectors of $\pi$ with vectors in $\Theta_{W_n}$, and study properties of this lifting as $n$ varies. In this spirit, we prove the following towering property of this theta lift:
\begin{Thm}\label{Thm: intro cuspidality}
Suppose that $\pi$ is a cuspidal automorphic representation of $\overline{\Sp}(W_0)_\A$. Then
\begin{enumerate}
\item If $\Theta_k(\pi)=0$, then $\Theta_{k-1}(\pi)=0$;
\item There exists an $n\gg 0$  such that $\Theta_n(\pi)\neq 0$;
\item\label{part 3} If $\Theta_{k-1}(\pi)=0$, then $\Theta_k(\pi)$ is cuspidal.
\end{enumerate}
\end{Thm}
Note that these results are analogous to those of Rallis \cite{R} in the context of liftings between dual reductive pairs of symplectic-orthogonal type, and of those in \cite{BFG2} for the double cover of odd orthogonal groups. This is the first example of such results for higher degree covers. We remark here that the restriction to the $4$-fold cover is not for convenience, but is necessary for the construction to be well-behaved. In particular, part (\ref{part 3}) of this theorem is not true when considering this construction on the $2n$-th degree cover when $n>2$. See the comments after Theorem \ref{Thm: intro Nilpotent orbit} below.

As a corollary of this theorem, we see that for a given cuspidal representation $\pi$ there is a first index  $n(\pi)$, where $\Theta(\pi):=\Theta_{n(\pi)}(\pi)$ is a non-zero cuspidal representation. One of the main results of the paper is that this representation is a CAP representation of $\overline{\Sp}(V_{n(\pi)})_\A$. 
\begin{Thm}\label{Thm: intro CAP}
Suppose $\pi$ is a cuspidal representation of $\overline{\Sp}(W_0)$, and suppose $\Theta(\pi):= \Theta_{n(\pi)}(\pi)$ is the first nontrivial lift. Then each irreducible summand of $\Theta(\pi)$ is CAP with respect to the triple $(Q_{t(\pi),n},\pi, \chi_{\Theta})$.
\end{Thm}
The precise notations and results are stated in Section \ref{Section: CAP}. While CAP representations have previously been studied for $\Mp(W_0)$ (see \cite{Y}, for example), this constitutes the first construction of CAP representations on higher degree covering groups. 

 Motivated by a natural unramified correspondence between $\overline{\Sp}_{2r}$ and $\SO_{2r+1}$, we conjecture in Section \ref{subsection: Shimura} an extension of Arthur's parameterization of the discrete spectrum to $\overline{\Sp}_{2r}$. Such a result has only recently been achieved in the case of $\Mp_{2r}$ \cite{GI}, and relies heavily on the theta lifting associated to $\omega_\psi$. Assuming such a parameterization, in Section \ref{subsection: Arthur} we compute the effect of our lift on Arthur parameters. In particular, we see that the parameter of the lift is necessarily non-tempered. This may be viewed simultaneously as evidence of our conjectural parameterization as well as an extension of Arthur's conjectures to BD-covers using the definition of $L$-group provided by Weissman. 

Of critical technical importance for the proof of this theorem is the computation of the maximal nilpotent orbit $\calo(\Theta)$ associated to $\Theta$. Recall that nilpotent orbits of $\mathfrak{sp}_{2r}(\cc)$ may be parametrized by certain partitions $(p_1^{e_1}\cdots p_k^{e_k})$ of $2r$, and form a partially ordered set under $\calo\leq\calo' $ if $\calo \subset \overline{\calo'}$. We say that a nilpotent orbit $\calo$ supports an automorphic representation $\pi$ if certain Fourier coefficients attached to $\calo$ do not vanish on $\pi$. For details, see Section \ref{Section: Fourier}. The results of Sections \ref{Section: vanishing} and   \ref{Section: Non-vanishing} may be stated as follows:
\begin{Thm}\label{Thm: intro Nilpotent orbit}
Let $\Theta_{2r}$ denote a theta representation of $\overline{\Sp}_{2r}(\A)$. Let $\calo_{\Theta,r}= (2^{r})$.
\begin{enumerate}
\item\label{part a} $\calo_{\Theta,r}$ supports the representation $\Theta_{2r}$.
\item \label{part b} If $\calo'$ is any nilpotent orbit greater than or not comparable to $\calo_{\Theta,r}$, then $\calo'$ does not support the representation $\Theta_{2r}$.
\end{enumerate}
Utilizing the notation from \cite{G3}, we have that
\[
\calo(\Theta_{2r}) = \left\{(2^r)\right\}.
\]
\end{Thm} This result is interesting in its own right, and is crucial to the proofs of Theorems \ref{Thm: intro cuspidality} and \ref{Thm: intro CAP}. We note in Appendix \ref{Appendix A} that the corresponding lift is degenerate in the $2$-fold covering case as $\omega_\psi$ is in some sense ``too small'' for this lift. 
\begin{Rem}
It is natural to hope that this result may be extended to the $2n$-th degree cover. Unfortunately, it is an open question what the nilpotent orbit associated to the theta representation $\Theta_{2r}^{(n)}$ on higher degree covers of the symplectic group ought to be. While Friedberg and Ginzburg have presented a conjecture in the case of odd-degree covers \cite{FG}, this is still open and there are fundamental obstructions to verifying this conjecture. See the comments after Proposition \ref{Prop: odd rank}. In any case, producing an analogous conjecture for even degree covers does not seem straightforward.  We remark that the proof of Theorem \ref{Thm: intro Nilpotent orbit} (\ref{part a}) shows that for $n>2$, we have
\[
\calo\left(\Theta^{(2n)}_{2r}\right)\geq (2^r),
\]
and in general the inequality is strict. In particular, these representations are too large for the integral lift considered here, so one must consider different constructions if one hopes to develop a generalization of this lift to higher degree covers. We expect a variation on the descent constructions of Friedberg-Ginzburg \cite{FG} should replace the theta kernel used in our construction, though this will change the nature of the local analysis and still requires knowing $\calo(\Theta^{(2n)}_{2r})$.
\end{Rem}
Our construction, as well as the value of $\calo(\Theta_{2r})$, is motivated by Ginzburg's dimension equation, which relates first occurrences of liftings defined in terms of integral pairings such as ours a certain equality relating the dimensions of the groups being integrated over to the Gelfand-Kirillov dimension of certain representations. For more details, see \cite{G3} and Section \ref{Section: Whittaker coefficients}. In the case at hand, the dimension equation suggests that in the case that if $\pi$ is cuspidal and generic, then $n(\pi) =\frac{1}{2}\dim(W) +1$. In Section \ref{Section: Whittaker coefficients}, we provide evidence for this conjecture by relating the Whittaker coefficients of the lift to those of $\pi$.

The paper is organized as follows: 

We begin by setting up notation and recalling the definition of BD-covering groups as well as local and global versions of theta representations. Section \ref{Section: Fourier} recalls the theory of Fourier coefficients and Jacquet modules associated to a nilpotent orbit $\calo$, and in Section \ref{Section: local FJ} we define the local Fourier-Jacobi modules associated to specific orbits. The study of these modules is our main tool to prove vanishing of certain global Fourier coefficients. In Section \ref{Section: Root Exchange}, we recall the tool (both locally and globally) of root exchange.

In Section \ref{Section: vanishing}, we prove Theorem \ref{Thm: intro Nilpotent orbit} (\ref{part b}). The method of proof combines in a new way global results of \cite{GRS2} and \cite{JL2} with a detailed study of local Fourier-Jacobi modules to force vanishing of the relevant Fourier coefficients.  We prove Theorem \ref{Thm: intro Nilpotent orbit} (\ref{part a}) in Section \ref{Section: Non-vanishing} in two stages. In the even rank case, we show that the non-vanishing of the certain Fourier coefficients of $\Theta_{4r}$ associated to $\calo_{\Theta,2r}=(2^{2r})$ is equivalent to the non-vanishing of certain semi-Whittaker coefficients of theta functions on the double cover of $\GL_{2r}(\A)$, which are non-zero by results of Bump-Ginzburg \cite{BG}. We then show that this forces non-vanishing in the odd rank case.

In Section \ref{Section: local correspondence}, we study the corresponding local lifting in the case of principal series representations. Despite the lack of a Schr\"{o}dinger model, we establish the existence of and analyze a filtration of certain Jacquet modules of the local theta representation $\Theta$ (see Appendix \ref{Appendix A}). Using the local results obtained in the proof of Theorem \ref{Thm: intro Nilpotent orbit}, we explicate the effect of the theta lift on Satake parameters. See Theorem \ref{Thm: local correspondence}. Happily, these local results may be interpreted in terms of Arthur parameters globally. 
In Sections \ref{Section: Global Lift} and \ref{Section: cuspidality}, we study the global lifting question and prove Theorem \ref{Thm: intro cuspidality}.  In Section \ref{Section: CAP}, we prove Theorem \ref{Thm: intro CAP} establishing that the lift produces CAP representations, and conjecture a natural extension of Arthur's parametrization of the discrete spectrum to the case of $\overline{\Sp}_{2r}(\A)$. We then relate these results  to Arthur's conjectures by computing the effect of the generalized theta lift on Arthur parameters. In particular, we show that when one combines the natural lifting of unramified data from $\overline{\Sp}_{2r}$ to $\SO_{2r+1}$ with Theorem \ref{Thm: local correspondence}, we may determine the (non-tempered) Arthur parameter of the lift.

In Section \ref{Section: Whittaker coefficients} , we review the notion of Gelfand-Kirillov dimension and discuss Ginzburg's philosophy of the dimension equation \cite{G3} and use it to study the case of a generic cuspidal representation $\pi$ of $\overline{\Sp}_{2r}$.  We show that the Whittaker coefficient of the lift of such a representation to $\overline{\Sp}_{2r+2}$ may be expressed in terms of the Whittaker coefficients of $\pi$ as well as a certain integral over the theta function, supporting the conjecture that $n(\pi)=r+1$ in this case.
\subsection*{Acknowledgments}

I want to thank my advisor Sol Friedberg for suggesting this line of research, as well as for providing unrelenting support. I also wish to thank David Ginzburg, whose conjectures directly led to this project; as well as Yuanqing Cai, for many helpful discussions. I also wish to thank the anonymous referee for helpful comments on an earlier version of this paper. Finally, I thank Marty Weissman for several clarifying discussions and lectures on his notion of $L$-group while attending the ``Conference on Automorphic Forms on Metaplectic Groups and Related Topics'' in IISER Pune. This work will make up a part of my doctoral thesis at Boston College.

\section{Notation}\label{Section: Notation}

%%%%%%%%%%%%%%%%%%%%%%%%%%%%%%%%%%%%%%%%%%%%%%%%%%%%%%%%%%%%%%%%%%%%%%%%%%%%%%%%%%%%%%%%%%%%%%%%%%%%%%%%%%%%%%%%%%%%%%%%%%%%%%%%%%%%%%%%%%%%%%%%%%%%%%%%%%%%%%%%%%%%%%%%%%%%%%%%%%%%%%%%%%%%%%%%%%%%%%%%%%%%%%%%%%%%%%%%%%%%%%%%%%%%%%%
%%%%%%%%%%%%%%%%%%%%%%%%%%%%%%%%%%%%%%%%%%%%%%%%%%%%%%%%%%%%%%%%%%%%%%%%%%%%%%%%%%%%%%%%%%%%%%%%%%%%%%%%%%%%%%%%%%%
%%%%%%%%%%%%%%%%%%%%%%%%%%%%%%%%%%%%%%%%%%%%%%%%%%%%%%%%%%%%%%%%%%%%%%%%%%%%%%%%%%%%%%%%%%%%%%%%%%%%%%%%%%%%%%%%%%	
We will have cause to work both locally and globally. Therefore, we let $K$ stand for a number field which contains the full set of $4^{th}$ roots of unity, and we use $F$ to denote a non-archimedean local field. Often $F$ will arise as the completion of our number field $K$ with respect to some finite place, $F=K_\nu$. Fix once and for all an injective character $\ep: \mu_4\to \cc^\times$.  Let $\A$ be the associated ring of adeles for the number field $K$, and fix a non-trivial additive character $\psi:K\backslash \A \to \cc^\times.$ Locally, we focus mainly on the case that $F$ is a local field of residue characteristic $p>2$. In this setting, we fix a nontrivial additive character $\psi: F\to \cc^\times$ of level $0$.

\subsection{Subgroups and characters}
Let $(W, \la\cdot,\cdot\ra)$ be a $2r$-dimensional symplectic space over $F$, and let $\Sp(W)_F$ be the group of isometries. Upon fixing a choice of basis for $W \cong F^{2r}$, we may identify $\Sp(W)_F \cong \Sp_{2r}(F)$. When we do so, we always use  the symplectic form given by
\begin{align}\label{eqn: symplectic form}
\left(\begin{array}{cc}
 				&-{J_r}\\
				J_r&
				\end{array}\right),
\end{align}
where
\begin{align}\label{eqn: J}
J_r = \left(\begin{array}{ccc}
						&&1\\
						&\Ddots&\\
						1&
						\end{array}\right).
\end{align}

Note that $J_r$ corresponds to a split quadratic form $Q_J$ on $F^r$. We denote the corresponding set of symmetric matrices by
\[
\Sym_{J}^{r}(F) = \{l \in \Mat_{r\times r} : J_r l - {}^TlJ_r = 0\}.
\]
With this form, we have the maximal torus $T_r$ of diagonal elements, and Borel subgroup $B_r=T_rU_r$ of upper triangular matrices. Here, $U_r$ is the unipotent radical of $B_r$. 
This choice gives us the (standard) based root datum
\[
(X_r,\Phi_r,\Delta_r,Y_r,\Phi_r^\vee,\Delta_r^\vee),
\]
where $X _r= X^\ast(T_r)$ is the character lattice and $Y_r= X_\ast(T_r)$ is the cocharacter lattice. We label the associated simple roots as in Bourbaki: $\Delta = \{ \al_1,\cdots,\al_r\}$, with $\al_r$ the unique long simple root. Also, let $W(\Sp_{2r})$ denote the Weyl group of $\Sp_{2r}$, with $w_i$ the standard representative of the Weyl reflection for the simple root $\al_i$.

We need to also set notations for various parabolic subgroups and characters: Let $\underline{s}=(p_1,\ldots,p_k,r')$ be a partition of $r$. Set $P_{\underline{s}}= M_{\underline{s}}U_{\underline{s}}$ to be the standard parabolic subgroup of $\Sp_{2r}$ with Levi factor $$M_{\underline{s}} \cong \GL_{p_1}\times\cdots\times \GL_{p_{k}}\times\Sp_{2r'}.$$

To ease the notation, we make a separate notation for certain parabolic subgroups. Namely, set $Q_{k,r-k} = P_{(1,\cdots,1, r-k)}$ to be the parabolic subgroups $Q_{k,r-k}=L_{k,r-k}V_{k,r-k}$ such that $$L_{k,r-k}\cong\GL_1^k\times \Sp_{2(r-k)}.$$

We will have cause to consider two main types of character of $V_{k,r-k}$. We set $\psi_k:V_{k,r-k}\to\cc^\times$ to be the Gelfand-Graev character 
\[
\psi_k(v) = \psi(v_{1,2}+v_{2,3}+\cdots + v_{k,k+1}),
\]
and we set $\psi^0_k:V_{k,r-k}\to \cc^\times$ to be the character
\[
\psi^0_k(v) = \psi(v_{1,2}+v_{3,4}+\cdots v_{k_0,k_0+1}),
\]
where $k_0$ is the largest odd integer $\leq k$.\\

For each $k\leq r$, one has the subgroup $$H_k \hra U_r$$ given by
\begin{equation}\label{heisenberg}
  h(x:y:z):= \left( \begin{array}{cccccc}
									I_{k-1}&&&&&\\
									&1&x&y&z&\\
									&&I_{r-k}&&y^*&\\
									&&&I_{r-k}&x^*&\\
									&&&&1&\\
									&&&&&I_{k-1}\\
\end{array}\right) \qquad x,y \in \mathbb{G}_a^{r-k}, \: z\in \mathbb{G}_a.
\end{equation}

We have an isomorphism $H_k \cong \mathcal{H}_{2(r-k)+1}$, the Heisenberg group on $2(r-k)+1$ variables. Note that the center of $H_k$, $Z(H_k)$, is the one dimensional root subgroup corresponding to the long root $\mu_k$ (see the next subsection).\\

We will often need to integrate over various spaces of the form $H(K)\bs H(\A)$ for some $K$-subgroup $H\subset \Sp_{2r}$. We adopt the shorthand $[H] =H(K)\bs H(\A)$. 

%%%%%%%%%%%%%%%%%%%%%%%%%%%%%%%%%%%%%%%%%%%%%%%%%%%%%%%%%%%%%%%%%%%%%%%%%%%%%%%%%%%%%%%%%%%%%%%%%%%%%%%%%%%%%%%%5%

\subsection{Formalism on roots}

We work extensively with the root system $\Phi = \mathrm{C}_r$,  so we take a moment to introduce a simple notation for the 3 types of root in this system. Recall that we label the simple roots $\Delta = \{ \al_1,\cdots,\al_r\}$, with $\al_r$ the unique long simple root.

 For $1\leq i\leq k \leq r-1$, set $$\ga_{i,k} = \al_{i}+\cdots +\al_{k}.$$

These parametrize our short roots within the sub-root system of type $\mathrm{A}_{r-1}$, corresponding to the Siegel parabolic. For the long roots, we set$$\mu_k= 2\al_{k}\cdots+2\al_{r-1}+\al_{r},\quad \mu_r=\al_{r}.$$

The positive roots not listed above are of the form
$$
\eta_{i,k} =\gamma_{i,k} + \mu_{k+1}.
$$

%----------------------------------------------
\section{Local Theory}\label{Section: local theory}

Let $F$ be a local field equipped with a discrete valuation. In this section, we will develop the theory of local theta representations for $\GF$. 

\subsection{Brylinski-Deligne Extensions}

We develop the notion of local theta (or exceptional) representations to facilitate later comparison with the work of Weissman.

Assume for now that $F$ contains the full set of $n^{th}$ roots of unity, and that $|n|_F=1$ (this is what Gan-Gao refer to as the tame case \cite{GG}). Suppose that $\mathbb{G}$ is a connected reductive group scheme over $F$, which for simplicity we assume is split, semi-simple, and simply-connected. For the more general set up, we refer the reader to \cite{Gao}. \\

We fix a pinning $(\mathbb{G}, \mathbb{T}, \mathbb{B} ,\{x_\al\}_{\al\in \Delta})$ of $\mathbb{G}$, and consider the corresponding to the based root datum $$(X, \Phi, \Delta, Y, \Phi^\vee, \Delta^\vee).$$ Here, $Y = \Hom(\mathbb{G}_m,\mathbb{T})$ is the cocharacter lattice and $X=\Hom(\mathbb{T},\mathbb{G}_m)$ is the character lattice. Recall that $W = N(\mathbb{T})/\mathbb{T}$ is the Weyl group of $\mathbb{G}$.

Let $\mathbb{K}_2$ be Quillen $K_2$-sheaf on the big Zariski site over $\Spec(F)$, and set $\mathrm{CExt}(\mathbb{G},\mathbb{K}_2)$ to be the category of central extensions of $\mathbb{G}$ by $\mathbb{K}_2$ as sheaves over $\Spec(F)$. With our restriction on $\mathbb{G}$, then Brylinski and Deligne give a combinatorial description of $\mathrm{CExt}(\mathbb{G},\mathbb{K}_2)$ (reminiscent of the classification of split reductive groups in terms of root data) in terms of triples $(Q,\mathcal{E}, f)$.
 Here, $Q$ is a $W$-invariant $\zz$-valued quadratic form on $Y$, $\mathcal{E}$ is a central extension of $Y$ by $F^\times$, and $f$ is a certain morphism of abelian groups which is more subtle. See the exposition in \cite{GG} for more details, along with the parallel parameterization in terms of pairs $(D,\eta)$ where $D$ developed by Weissman \cite{W2}. Note that if $\mathbb{G}$ is semi-simple and simply connected, then we may ignore the more subtle data of $f$ (or $\eta$) in the classification.

If we further assume that $\mathbb{G}$ is simple and simply connected, then any $W$-invariant quadratic form $Q:Y\to \zz$ is completely determined by its value on $\al^\vee$, where $\al\in\Delta$ is any long simple root. In other words,
\[
\mathrm{CExt}(\mathbb{G},\mathbb{K}_2)\cong \zz,
\] 
where we view $\zz$ as a discrete category. In this case, we will always choose the quadratic form such that $Q(\al^\vee)=-1$ for any short coroot $\al^\vee$. 

For simplicity, we assume $\mathbb{G}$ is semi-simple and simply connected for the remainder of the section and fix a central extension $\overline{\mathbb{G}}\in \mathrm{CExt}(\mathbb{G},\mathbb{K}_2)$ corresponding to the quadratic form $Q$. By taking $F$-points, we have a short exact sequence of $F$-groups 
\[
1\longrightarrow K_2(F)\longrightarrow \overline{\mathbb{G}}(F)\longrightarrow \G\longrightarrow 1,
\]
where $\G=\mathbb{G}(F)$.
One of the main properties of classical algebraic $K_2$ is that the $n^{th}$ Hilbert symbol $(\cdot, \cdot)_n$ gives a homomorphism 
\[
(\cdot, \cdot)_n : K_2(F) \to \mu_n(F).
\]
Pushing out the above short exact sequence by this morphism gives rise to the exact sequence of topological groups
\[
1\longrightarrow \mu_n(F)\longrightarrow \overline{\G}^{(n)}\longrightarrow \G\longrightarrow 1,
\]
which by the assumptions on $F$ realizes the covering group $\overline{\G}^{(n)}$ as a degree $n$ covering group. This covering group is determined by the data $(Q, n)$, though not uniquely so. For any subset $S\subset \G$, we denote its full inverse image under the covering map by $\overline{S}$.

%%%%%%%%%%%%%%%%%%%%%%%%%%%%%%%%%%%%%%%%%%

To the covering group $\Gn$, one may associate a complex dual group which may be distinct from the dual of $\G$ (see \cite{McN}, \cite{FL}, and \cite{Re}). Let $B_Q : Y\times Y\to \zz$ be the symmetric bilinear form associated to the quadratic form $Q$. To the data $(Q,n)$, we may associate the sublattice 
\[
Y_{Q,n} = \{y\in Y : B_Q(y,y') \in n\zz \mbox{ for all } y'\in Y\} \subset Y.
\]
For each $\al^\vee\in \Psi^\vee$, we define $n_\al = n/\gcd(n,Q(\al^\vee))$. Set
\[
\al_{Q,n}^\vee = n_\al\al^\vee,\;\; \al_{Q,n} =\frac{1}{n_\al}\al.
\]
Set $Y_{Q,n}^{sc} =\sspan_\zz\{\al_{Q,n}^\vee : \al\in \Phi\}$. Then the complex dual group $\overline{\mathrm{G}}^\vee$ of $\Gn$ is the split group associated to the root datum 
\[
\left(Y_{Q,n}, \{\al_{Q,n}^\vee\}, \Hom(Y_{Q,n},\zz), \{\al_{Q,n}\}\right).
\]

As $\GF:=\overline{\Sp}_{2r}^{(4)}(F)$ is the covering group of primary interest for this paper, we remark that in this case $\overline{\G}^\vee=\Sp_{2r}(\cc)$.

Finally, we remark that the assumption $|n|_F=1$ implies that the cover $\overline{K}$ over a fixed hyperspecial maximal compact subgroup $K = \mathbb{G}(\OF)$ is split; that is, there exists a homomorphism $\kappa: K \to \Gn$ splitting the cover. Such a splitting need not be unique (in fact, splittings form a $\Hom(K,\mu_n)$-torsor). That being said, we fix a splitting from here on, and nothing shall depend upon this choice.

\subsection{Principal series and Exceptional representations}\label{Section: Principal series}
 
We turn now to the theory of principal series representations for the covering group $\Gn$. We follow the setup and notation in \cite{Gao} closely.

Fix an embedding $\ep :\mu_n(F)\to \cc^\times$. For any subgroup $H\subset \G$, we say that a representation $(\rho,W)$ of $\overline{H}$ is ($\ep$-)genuine if $\rho|_{\mu_n(F)} \equiv \ep$. 

Let $T=\mathbb{T}(F)$ be the split maximal torus in $\G$ contained in $B=\mathbb{B}(F)$, and let $\overline{T}$ be its full inverse image in $\Gn$. In general, $\overline{T}$ is no longer abelian, but rather a $2$-step nilpotent group.  One can show that the center $Z(\overline{T})$ is given by the full inverse image of the isogeny $$i: T_{Q,n}=Y_{Q,n}\otimes F^\times \to T.$$ 

Let $\chi : Z(\overline{T})\to \cc^\times$ be a genuine character of $Z(\overline{T})$. For any maximal abelian subgroup $A\supset Z(\overline{T})$, we set 
\[
i(\chi)=\ind_A^{\overline{T}}(\chi'),
\]
where $\chi'$ is any extension of $\chi$ to $A$. A version of the Stone-Von Neumann theorem tells us that $i(\chi)$ is independent of $A$ or the extension $\chi'$, and that the map
\[
\chi \mapsto i(\chi)
\]
induces a bijection between the irreducible genuine representations of $Z(\overline{T})$ and $\overline{T}$. In particular, we see that any irreducible representation of $\overline{T}$ is determined up to isomorphism by its central character. Despite the independence of $i(\chi)$ from $A$, it is useful to point out that in our setting, we may take $A = Z(\overline{T})(\overline{T}\cap K)$.

% where $\delta_B$ is the modular character of the Borel $B = \mathbb{B}(F)$, which we view as a non-genuine character of $Z(\overline{T})$
Define $\Ind(\chi) := \Ind_{\overline{B}}^{\Gn}(i(\chi))$ be the normalized induced representation. These are the principal series representations of $\Gn$. As in the reductive setting, $\Ind(\chi)$ is unramified ($\Ind(\chi)^K \neq\{0\}$) if and only if $\chi$ is unramified (that is $\chi|_{Z(\overline{T})\cap K} \equiv 1$). Note that unramified characters are the same as characters of $\overline{Y_{Q,n}}$ where
\[
1\lra \mu_n \lra \overline{Y_{Q,n}}\lra Y_{Q,n}\lra 1
\]
is the induced (abelian) extension of the lattice.

We now turn to the definition of an exceptional character. Following the notation in \cite{Gao}, we have for each root $\al$ the natural map $\overline{h}_\al: \mathbb{G}_m \to \Gn$, covering the embedding into $\G$ induced by our choice of pinning.
\begin{Def}
An unramified genuine character $\overline{\chi}:\overline{T}\to \cc^\times$ is called \emph{exceptional} if $\overline{\chi}(\overline{h}_\al(\varpi^{n_\al}))=q^{-1}$ for each simple root $\al\in \Delta$.
\end{Def}

The set of unramified exceptional characters of $\Gn$ is a torsor over $$Z(\overline{\G}^\vee) = \Hom(Y_{Q,n}/Y^{sc}_{Q,n},\cc^\times).$$
To each such character, we may associate a theta representation:
\begin{Def}
For an exceptional character $\overline{\chi}$, the theta representation $\Theta(\Gn,\overline{\chi})$ associated to $\overline{\chi}$ is the unique Langlands quotient of $\Ind(\overline{\chi})$. 
\end{Def}
For properties of theta representations, see \cite{Gao} and \cite{BFG1}.
\subsection{The Case at hand: $G=\Sp_{2r}$ $n=4$}\label{Section: Case at hand}

In this section, we inspect further the case of primary interest in this paper: the 4-fold cover $\GF$ of $\Sp_{2r}(F)$. Thus, from this point forward we assume that $F$ is a local field containing the $4^{th}$ roots of unity.

As before, we choose our quadratic form such that $Q(\al_r^\vee)=-1$, where $\al_r$ is the unique long simple root. In this case, if we identify $Y =Y^{sc}=\sspan_\zz\{\al_i^\vee\} \to \zz^r$ via
\[
\sum_ic_i\al_i^\vee \mapsto (c_1,c_2-c_1, \ldots, c_r-c_{r-1}), 
\]
 we obtain (setting $n=4$)
\[
Y_{4} :=Y_{Q,4}= \{ (x_1,\ldots, x_r) : 2| x_i\},
\:\;
Y^{sc}_{4} = \{ (x_1,\ldots, x_r)\in Y_{4} : 4| \sum_i x_i\}.
\]

This shows that $|Z(\overline{G}^\vee)|=2$, so there are two distinct unramified exceptional characters.

For the sake of explicit computations, we find it useful to work with an explicit cocycle. To this end, we will utilize the bisector $D: Y\otimes Y \to \zz$ given by 
\[
D(\al_i^\vee,\al_j^\vee) =\begin{cases}
			B_Q(\al_i^\vee,\al_j^\vee) &: i<j\\ 
			Q(\al_i^\vee) &: i=j\\ 
				0&: i>j\end{cases}.
\]
With this choice of bisector, we may compute explicitly the cocycle on the maximal torus. Since the map $H^2(\Sp_{2r}(F),\mu_4)\to H^2(T_r, \mu_4)$ is injective, we see that this does in fact determine the cover completely. 

Let us identify $\GF = \Sp_{2r}(F)\times\mu_4$, and parametrize $T_r$ via
\[
M(t_1,\ldots, t_r) = \diag(t_1,\ldots,t_r,t_r^{-1},\ldots,t_1^{-1}).
\]
Then, our bisector induces the cocycle $\sigma_D$ given by
\[
\sigma_D(M(t_1,\ldots,t_r) ,M(s_1,\ldots,s_r)) = \prod_{k=1}^r(t_k,s_k)^{-1}_4,
\]
where $(\cdot,\cdot)_n$ is the $n^{th}$ power Hilbert symbol. 

\begin{Rem}
Note that this choice of bisector (since $\Sp_{2r}$ is simply connected, we may set $\eta =1$; see \cite{GG}), determines our $4$-fold covering group up to isomorphism and identifies it as the pullback along the standard embedding $\Sp_{2r}\hra \SL_{2r}$ of the $4$-fold cover of $\SL_{2r}(F)$ with the BLS block compatible cocycle (see \cite{BLS}).
\end{Rem}
Recall the isogeny of tori 
\[
i:T_{Q,4}\to T_r.
\]
An observation made clear by our choice of bisector is that if $t,t'\in i\left(T_{Q,4}\right)$, then $\sigma_D(t,t') =1$. This implies that there is a bijection between characters of $i(T_{Q,4})$ and genuine characters of $Z(\overline{T})$. In particular, the genuine unramified characters of $Z(\overline{T})$ are in bijection with characters of the lattice $Y_{4}$. The bijection is not, however, canonical. 

In \cite{GG}, Gan-Gao introduce the idea of a genuine distinguished character $\overline{\chi}_0$ of $Z(\overline{T})$. We do not recall the full definition, but wish to clarify the case of $\GF$ as the choice of distinguished character affects the Satake parameter attached to an unramified principal series representation. While there are $4$ distinct distinguished characters for $\GF$, only two of them, $\overline{\chi}^0_+$ and $\overline{\chi}^{0}_-$, are unramified. Here, setting $y_i = 2\al_i^\vee$, these characters are determined by
\[
\overline{\chi}^0_{\pm}(y_i(a),\zeta) =\begin{cases}\zeta\qquad\:\:\;: \:\:\;\; i<r\\ \pm\zeta(a,a)_2: \; i=r. \end{cases}
\]
Here $(\cdot,\cdot)_2$ is the quadratic Hilbert symbol. and $\psi:F\to \cc^\times$ is our fixed additive character. Note that the distinguished character $\overline{\chi}^0_+$ equals the unitary distinguished character $\overline{\chi}^0_\psi$ constructed in \cite{GG}. One can show that in the $4k$-fold covering case, the dependence on the choice of character $\psi:F\to \cc^\times$ disappears. We note that if we were to further assume $\mu_8\subset F^\times$, then $(a,a)_2=1$, and $\overline{\chi}^0_+$ becomes essentially trivial. In any case, precisely one of these characters agrees with the choice of distinguished character implicit in \cite{FG}. 

We fix a choice $\overline{\chi}^0$ of unramified distinguished character to obtain our bijection, $$\chi\mapsto \overline{\chi}^0\chi,$$ for any unramified character $\chi: i(T_{Q,4})\to\cc^\times$, viewed on the right hand side as a non-genuine character of $Z(\overline{T})$. This choice also gives us our exceptional character as follows. Consider the character $$\rho_4=\sum_{i\in\Delta}\omega^\vee_i\in \Hom(Y_{Q,n},\zz) = X_\ast(\overline{T}^\vee),$$
the sum of fundamental weights of the dual group $\overline{G}^\vee$. Setting $\chi_{\rho_4}:Y_4\otimes F^\times\to \cc^\times$ as the corresponding unramified character, we have our chosen genuine unramified exceptional character
\[
\overline{\chi} = \overline{\chi}^0\chi_{\rho_4}.
\] 
This choice will become important in Section \ref{Section: local correspondence} as it fixes a notion of Satake parameters relative to a distinguished splitting of the $L$-group ${}^L\overline{\Sp}_{2r}$ (see \cite{W2}, \cite{W3}, and the work of Weissman more generally). Set $\Theta_{2r}:=\Theta(\overline{\Sp}_{2r},\overline{\chi})$ to be the corresponding theta representation.
\subsection{Jacquet modules of local theta representations}\label{Section: local Levi}

Consider the maximal parabolic subgroup $P_{(k,r-k)}(F) =M_{(k,r-k)}(F)U_{(k,r-k)}(F),$ and note that, by a simple computation using the block compatibility of the cocycle, the induced cover
\[
\overline{M}_{(k,r-k)}(F) = \overline{\GL}^{(2)}_k(F)\times_{\mu_4}\overline{\Sp}_{2(r-k)}(F)
\]
is given by the direct product amalgamated at $\mu_4$ of a metaplectic cover $\overline{\GL}^{(2)}_k(F)$ corresponding to the BD-covering group $\widetilde{\mathbb{GL}}_k[-1,0]$ (see the notation in \cite[Section 2.1]{Gao2}) and the $4$-fold metaplectic cover $\overline{\Sp}_{2(r-k)}(F)$. 

\begin{Rem}We denote the corresponding cover of the general linear group by $ \overline{\GL}^{(2)}_k(F)$ despite the fact that it is a four-fold cover since it is the determinantal twist of the four-fold cover associated to the BD-covering group $\widetilde{\mathbb{GL}}_k[0,2]$. The associated four-fold covering group $\overline{\GL}^{(4)}_k(F)_{S}$\footnote{The $S$ stands for Savin, per \cite{Gao2}} is the pushout of the double cover $\overline{\GL}^{(2)}_k(F)_{KP}$ of the Kazhdan-Patterson cover $\widetilde{\mathbb{GL}}_k[0,1]$; that is, we have a commutative diagram
$$
\begin{tikzcd}
1\arrow[r] & \mu_2 \arrow[r] \arrow[d ] & \overline{\GL}^{(2)}_k(F)_{KP} \arrow[d] \arrow[r] & \GL_k(F) \arrow[d, "="] \arrow[r]& 1 \\
1\arrow[r] & \mu_{4} \arrow[r]& \overline{\GL}^{(4)}_k(F)_{S} \arrow[r] & \GL_k(F) \arrow[r]& 1. 
\end{tikzcd}
$$

It follows that the category of genuine smooth admissable representations of  $\overline{\GL}^{(2)}_k(F)_{KP}$ is equivalent to the category of $\ep$-genuine smooth admissable representations of $ \overline{\GL}^{(4)}_k(F)_{S}$, and for our purposes the determinantal twist does not complicate things (This is essentially contained in the cocycle computation of Friedberg-Ginzburg \cite[Section 2]{FG}.) This is useful for translating results about theta representations on the double cover to this covering group.
\end{Rem}

Under the inclusion of covering tori
\[
Z(\overline{T}_{\GL_k})\times_{\mu_4}Z(\overline{T}_{r-k})\hra Z(\overline{T}_r),
\] 
our choice of distinguished character $\overline{\chi}^0$ induces a distinguished character on both factors, and either choice of character gives the trivial distinguished character for $Z(\overline{T_{\GL_k}})$. Let $\Theta_{2(r-k)}$ be the corresponding theta representation of $\overline{\Sp}_{2(r-k)}(F)$. Additionally, define the unramified exceptional character of $Z(\overline{T}_{\GL_k})$: 
\begin{align}\label{eqn: GL-character}
\chi_{\GL,k}(\diag(t_1,\ldots,t_k),\zeta) = \zeta\prod_{i=1}^k|t_i|^{(2(r-i)+3)/4}
\end{align}
Of central importance to the study of theta representations on symplectic groups is the following version of the ``periodicity theorem'' of \cite{KP} and its corollary (see also \cite{BFG1} for a version in the odd orthogonal case): 
\begin{Thm}\label{Thm: local constant term} 
Let $F$ be a non-archimedean local field, and let $\Theta_{2r}$ be the theta representation on $\overline{\Sp}_{2r}(F)$. Let $\Theta_{\GL_k}^{(2)}$ denote the local theta representation on $\overline{\GL}_{a}^{(2)}(F)$ associated to the exceptional character $\chi_{\GL,k}$. Then for any maximal parabolic subgroup $P_{(k,r-k)}$, we have an isomorphism of $\GL_{k}^{(2)}(F)\times_{\mu_4}\overline{\Sp}_{2(r-k)}(F)$ representations 
\[
J_{U_k}(\Theta_{2r}) \cong \Theta_{GL_k}^{(2)}\otimes\Theta_{2(r-k)}.
\]
\end{Thm} 
\begin{proof}
This follows as in \cite{BFG1}, Theorem 2.3. Note, however, that the exceptional character for the $\GL$-component of the Levi subgroup differs here from the character in the odd orthogonal case by multiplication by the character
\[
g\mapsto |\det(g)|^{3/4}.
\]
This is allowed as the definition of exceptional character depends only on its values on the image of the root subgroups.
\end{proof}

\begin{Cor}\label{Cor: local constant term}
There exists a $\overline{\Sp}_{2r}(F)$-equivariant embedding  $$\Theta_{2r}\hra \Ind\left( \Theta_{GL_k}^{(2)}\otimes\Theta_{2(r-k)}\right).$$ Additionally, there exists a value $s_0\in \cc$ such that there is a $\overline{\Sp}_{2r}(F)$-equivariant surjection $$\Ind\left( \Theta_{GL_k}^{(2)}\otimes\Theta_{2(r-k)}\otimes\delta_{P_k}^{s_0}\right)\twoheadrightarrow \Theta_{2r}.$$
\end{Cor}

\section{Global Theory}\label{Section: global theta}
Let $K$ be a number field, and let $\overline{\Sp}_{2r}(\A)$ be the restricted direct product of the $4$-fold covering groups $\overline{\Sp}_{2r}(K_\nu)$ amalgamated at the central $\mu_4$. We fix a choice of global distinguished character $\overline{\chi}^0 :Z(\overline{T}_r)_K\backslash Z(\overline{T}_r)_\A\to \cc^\times$. This gives us a corresponding global exceptional character
\[
\overline{\chi}_{\Theta,2r} := \overline{\chi}^0 \chi_{\Theta},
\]
where 
\[
\chi_{\Theta}(M(t_1,\ldots,t_r)) = \prod_{i=1}^r|t_i|^{(2(r-i)+1)/4}
\]
Let $\Theta_{2r}$ denote the global theta representation with respect to this choice of character, as described in \cite[Section 2]{FG}. In a manner completely analogous to the local construction as the Langlands quotient of a specific principal series representation, this representation is constructed by taking the residue of the Borel-Eisenstein series at the pole $s_\theta=((2(r-1)+1)/4,
\ldots, 3/4,1/4)$. In particular, $\Theta_{2r}$ is an irreducible genuine automorphic representation and if $\overline{\chi}^0 =\prod_\nu\overline{\chi}^0_\nu$,\footnote{For each finite place, $\overline{\chi}^0_\nu$ is automatically distinguished.} we have $$\Theta_{2r} = \otimes'_\nu\Theta_{2r,\nu}.$$
Note that in \cite{FG}, Friedberg and Ginzburg work with the case $\overline{\chi}^0_\nu$ is the trivial distinguished character for each finite place $\nu$. However, the set-up is the same for a general $\overline{\chi}^0$.

We take note of the following proposition  of \cite[Prop. 1]{FG}: Let $\overline{\chi}_{\GL,k}$ denote the global analogue of the exceptional character (\ref{eqn: GL-character}) of $\overline{\GL}^{(2)}_{k}$ from the previous section, and let $\Theta_{\GL_k}^{(2)}$ be the corresponding global theta representation.
\begin{Thm}\label{Thm: global constant term}
Let $\theta_{2r}\in \Theta_{2r}$. Then there exists $\theta_{\GL_k}\in\Theta_{\GL_k}^{(2)}$ and $\theta_{2(r-k)}\in \Theta_{2(r-k)}$ such that for any diagonal $g\in \GL_{k}^{(2)}(\A)$ which lies in the center of the Levi of the parabolic subgroup $\overline{P}_{(k,r-k)}(\A)$ and for all unipotent $h\in \overline{\Sp}_{2(r-k)}$ and $v\in\GL_k^{(2)}(\A)$, we have
\[
\int_{[U_{(k,r-k)}]}\theta_{2r}(u(gv,h))\,du = \overline{\chi}_{\Theta, k}(g)\theta_{\GL_k}(v)\theta_{2(r-k)}(h).
\]
\end{Thm}
The global analogue of Corollary \ref{Cor: local constant term} also holds. We refer the reader to \cite{FG} for the details.

\section{Fourier Coefficients Associated to a Nilpotent Orbit}\label{Section: Fourier}
In this section, we outline the process of associating  harmonic-analytic data to a nilpotent orbit $\calo$ in a complex Lie algebra $\fg=\Lie(G)$ (or equivalently, a unipotent orbit in any complex connected Lie group $G$ such that $\Lie(G)=\fg$) and a representation $\pi$ of $\G$. To simplify the exposition, we will discuss only $\G=\Sp_{2r}$.  However, as these coefficients are given globally by integration over certain unipotent subgroups and locally by applying certain Jacquet functors, all the definitions and statements in this section apply without change to the case of a BD-covering group. These concepts are discussed in more detail in \cite{G1}, \cite{M}, and \cite{MW1}. 
\subsection{Algebraic Setup}
Let $F$ be any field, and fix an algebraic closure $\overline{F}$. Also, let $\calo$ be a nilpotent orbit in the $\overline{F}$-Lie algebra $\mathfrak{sp}_{2m}$. These orbits are finite in number and are parametrized by symplectic partitions of $2r$; that is, partitions  where each odd number occurs with even multiplicity. If the orbit $\calo$ corresponds to the partition $(p_1^{e_1}p_2^{e_2}\cdots p_r^{e_r})$ where $p_i\leq p_{i+1}$, we simply write
\[
\calo=(p_1^{e_1}p_2^{e_2}\cdots p_r^{e_r}).
\]

There is a natural partial order on both nilpotent orbits (via closure relations) and partitions (via dominance). The above correspondence identifies these partial orders.\\

Let $\calo=(p_1^{e_1}p_2^{e_2}\cdots p_r^{e_r})$ be such an orbit. To each entry $p_i$ which appears in the partition associated to $\calo$ (including mutiplicities) we associate the torus element
\[
h_{p_i}(t)=\diag(t^{p_i-1},t^{p_i-3},\ldots, t^{3-p_i},t^{1-p_i}).
\]
From, these we may produce a semisimple element $h_\calo(t)\in \Sp_{2r}(\overline{F}(t))$ by concatenating all the $h_{p_i}(t)$ and then conjugating by a Weyl group element until the entries are such that the powers of $t$ are decreasing. It is a consequence of the fact that $\calo$ is a symplectic partition of $2r$ that the resulting semisimple element lies in $\Sp_{2r}$. See \cite{CM} for more background. This data induces a filtration on the unipotent radical of $B$ 
\[
I\subset\cdots \subset V_2(\calo)\subset V_1(\calo) \subset V_0(\calo)= U,
\]
where the terms in the filtration are defined in terms of root subgroups
\[
V_l(\calo) = \la x_\al(s) \in U : h_\calo(t) x_\al(s) h_\calo(t)^{-1} = x_\al(t^ks)\:\mbox{ for some } k\geq l\ra.
\]
Also, we define 
\[
M(\calo) = T\cdot \la x_{\pm\al}(s) :  h_\calo(t) x_\al(s) h_\calo(t)^{-1} = x_\al(s)\ra.
\]
Note that all of these subgroups are defined over $\zz$, and that $P_\calo =M(\calo)V_1(\calo)$ is a parabolic subgroup of $\Sp_{2r}(F)$.

Since we are working with an algebraically closed field, $M(\calo)_{\overline{F}}$ acts via conjugation on the quotient $$Z_2(\calo):=V_2(\calo)/[V_2(\calo),V_2(\calo)]$$ with a dense orbit. Let $u_\calo\in V_2(\calo)$ be a representative of this orbit. Set $L(\calo)^{alg}\subset M(\calo)_{\overline{F}}$ to be the stabilizer of $u_\calo$. While the exact group $L(\calo)^{alg}$ depends on the choice of representative, its Cartan type does not.\\
\subsection{Global Setting and Fourier coefficients}
Let $K$ be a number field with adele ring $\A$, and let $(\pi,V)$ be an automorphic representation of $\Sp_{2r}(\A)$.
Since $M(\calo)_{\A}$ acts on $V_2(\calo)_\A$ by conjugation, one obtains an action of $M(\calo)_K$ on the character group  $$\left(Z_2(\calo)_K\backslash Z_2(\calo)_\A\right)^\vee\cong Z_2(\calo)_K.$$ 
We consider only characters $\psi_\calo : V_2(\calo)_K\backslash V_2(\calo)_\A\to \cc^\times$ such that the connected component of the stabilizer $L(\calo)_K := \Stab_{M(\calo),K}(\psi_\calo)$ has the same Cartan type as $L(\calo)^{alg}$ after base changing to $\overline{K}$. Such characters are called \emph{generic characters} associated to $\calo$. There can exist infinitely many $M(\calo)_K$-orbits of generic characters.

\begin{Def}
 Let $\psi_\calo:~V_2(\calo)_K\backslash V_2(\calo)_\A\to \cc^\times$ be a generic character associated to $\calo$.
For a vector $\varphi\in\pi$, we consider the coefficient
\[
F_{\psi_\calo}(\varphi)(g)=\displaystyle \int_{[V_2(\calo)]}\vp(vg)\psi_\calo(v)\,dv.
\]

We say that the orbit $\calo$ supports $\pi$ if there exists a choice of data ($\varphi,\: \psi_{\calo}$) such that the above integral is not identically zero. Otherwise, we say that $\calo$ does not support $\pi$.\qed
\end{Def}

Note that $F_{\psi_\calo}(\varphi)(g)$ is an automorphic function of $L(\calo)_\A$.

\subsection{Local Setting}
 Now suppose that $F$ is a non-archimedean local field. Let $U$ be a unipotent subgroup of an $l$-group $G$ (in the terminology of \cite{BZ}), let $\chi :U\to \cc^\times$ be a character of $U$. Suppose that $M\subset G$ is a subgroup of $G$ normalizing $U$ and leaving $\chi$ invariant. 

Then we have the exact functor 
\[
J_{U,\chi}: Rep(G) \to Rep(M),
\]
the details of which are discussed in \cite{BZ}, and we refer there for properties of Jacquet modules. 
%%%%%%%%%%%%%%%%%%%%%%%%%%%%%%%%%%%%%%%%%%%%%%%%%%%%%%%%%%%%%%%%%%%%%%%%%%%%%%%%%%%%%%%%%%%%%%%%%%%%%%%%%%%%%%%%%%%%%%%%%%%%%%%%%%%%%%%%%%%%%%%%%%%%%%%%%%%%%%%%%%%%%%%%%%%%%%%%%%%%%%%%%%%%%%%%%%%%%%%%%%%%%%%%%%%%%%%%%%%%%%%%%%%%%%%
%%%%%%%%%%%%%%%%%%%%%%%%%%%%%%%%%%%%%%%%%%%%%%%%%%%%%%%%%%%%%%%%%%%%%%%%%%%%%%%%%%%%%%%%%%%%%%%%%%%%%%%%%%%%%%%%%%%%%%%%%%%%%%%%%%%%%%%%%%%%%%%%%%%%%%%%%%%%%%%%%%%%%%%%%%%%%%%%%%%%%%%%%%%%%%%%%%%%%%%%%%%%%%%%%%%%%%%%%%%%%%%%%%%%%%%
We now define the local analogue of the Fourier coefficients associated to a nilpotent orbit $\calo$. We have the action of $M(\calo)_F$ on $Z_2(\calo)_F$ by conjugation, and hence an action on the Pontryagin dual $Z_2(\calo)_F^\vee \cong Z_2(\calo)_F$. As in the global setup, we consider only characters $\psi_\calo$ such that the stabilizer $L(\calo)_F\subset M(\calo)_F$ is of the same absolute Cartan type as $L(\calo)^{alg}$. We call such characters \emph{generic} characters associated to $\calo$.
\begin{Def}
Let $(\pi,V)$ be a smooth admissible representation of finite length for $\overline{\Sp}_{2r}(F)$. We define the \textbf{twisted Jacquet modules} associated to the nilpotent orbit $\calo$ and the generic character $\psi_\calo$
\[
J_{\psi_\calo}(\pi) := J_{V_2(\calo),\psi_\calo}(\pi).
\]
We say that a nilpotent orbit $\calo$ {supports} $\pi$ if there exists a generic character such that $J_{\psi_\calo}(\pi) \neq 0$.\qed
\end{Def}

The following proposition encodes the well-known relationship between Fourier coefficients and Jacquet modules.
\begin{Prop}\label{Prop: local-global}
Suppose that $\pi \cong \otimes'_\nu \pi_\nu$ is an automorphic representation of $\overline{\Sp}_{2r}(\A)$ and suppose that the nilpotent orbit $\calo$ supports $\pi$.
Then for any finite place $\lam$, $\calo$ also supports $\pi_v$.
\end{Prop}
\section{Local Fourier-Jacobi Coefficients}\label{Section: local FJ}
Our main technique for proving the vanishing of Fourier coefficients associated to certain small nilpotent orbits is to study associated local Fourier-Jacobi models. These are local analogues of the Fourier-Jacobi coefficients studied in  \cite{GRS2}, and the functor we study is intimately related to the Fourier-Jacobi map studied by Weissman \cite{W1} as well as the generalized Whittaker models of \cite{GGS}, though this is the first occurrence of this technique in the context of higher-degree covering groups.

Letting $F$ be a non-archimedean local field, we note the natural compatibility between the various covers of $\Sp_{2r}(F)$ arising from a fixed BD-cover induced by the compatibility between the various Hilbert symbols:
\[
(x,y)_{nm}^n = (x,y)_m \qquad \mbox{ for any } x,y \in F^\times.
\]
Thus, there is a canonical covering group morphsim
\begin{align}\label{eqn: cover}
\overline{\Sp}_{2r}(F) \twoheadrightarrow \Mp_{2r}(F),
\end{align}
where $\Mp_{2r}(F)$ is the classical metaplectic group, viewed as a $2$-fold BD-cover of $\Sp_{2r}(F)$.

 Let $(\pi, V)$ be a smooth admissible representation of $\GF$, and let $\calo$ be a nontrivial nilpotent orbit. To $\calo$ we have attached the unipotent subgroups $V_1(\calo)$ and $V_2(\calo)$. It is shown in \cite{GRS2} that the quotient $$V_1(\calo)/\ker(\psi_\calo)= (X\oplus Y )\oplus Z,$$  is a generalized Heisenberg group $\mathcal{H}_{2m+1}\oplus F^k$, for some $m,k\in \zz_{\geq 0}$. Here $Z = V_2(\calo)/\ker(\psi_\calo)$ forms the center and $$V_1(\calo)/V_2(\calo)=X\oplus Y$$ gives the symplectic space of dimension $2m$ over $F$. One sees that there always exists a choice of polarization $X\oplus Y$ such that both isotropic subspaces embed as subgroups of $U$.

Consider the nilpotent orbits $\calo_m = ((2m)1^{2(r-m)})$ for some $1\leq m\leq r.$ We remark that the following construction works for any orbit, but we need to only consider orbits of this form.
We fix the following choice of generic character $\psi_{m,\al}:=\psi_{\calo_m}$: for $v\in V_2(\calo_m)_F$ and $\al\in F^\times$,
\[
\psi_{m,\al}(v) = \psi(v_{1,2}+\cdots+v_{m-1,m}+\al v_{m,2r-m+1}).
\]
We also define the characters $\psi_{m}(v) = \psi(v_{1,2}+\cdots+v_{m-1,m})$ and $\psi_\al(v)=\psi(\al v_{m,2r-m+1})$.

In terms of matrices, the Heisenberg group $H_m=V_1(\calo_m)/\ker(\psi_\al) \cong \mathcal{H}_{2(r-m)+1}$ associated to the orbit $\calo_m$ may be represented by the natural Heisenberg group occurring as a subgroup of $U$ by matrices of the form $ h(x:y:z)$ as in equation (\ref{heisenberg})

In this way, we may identify $\psi_\al$ as a character of the center $$Z(H_m) =\{h(0:0:z) : z\in F\},$$ and form the corresponding oscillator representation
\[
\omega_{\al} =\ind_{YZ(H_m)}^{H_m}(\psi^{-1}_\al),
\]
where $Y=\{h(0:y:0): y \in F^{r-m}\}$ is a Lagrangian subspace. We may identify this as a representation of $V_1(\calo_m)$, and note that twisting by the character $\psi_m^{-1}$ preserves the induction:
\begin{align*}
\omega_{m,\al} = \ind_{YV_2(\calo_m)}^{V_1(\calo_m)}(\psi_\al^{-1})\otimes\psi_m^{-1} \cong \ind_{YV_2(\calo_m)}^{V_1(\calo_m)}(\psi_{m,\al}^{-1})
\end{align*}
Note that the subgroup $\Mp_{2(r-m)}(F)\subset \Mp_{2r}(F)$, corresponding to the full inverse image of $L(\calo_m)_F \cong \Sp_{2(r-m)}(F)$ acts on this oscillator representation.
Precomposing with the projection (\ref{eqn: cover}), we may view $\omega_{m,\al}$ as a representation of $\overline{L(\calo_m)}_F \cong \overline{\Sp}_{2(r-m)}(F)$.

Consider the diagonal action of $V_1(\calo)\rtimes \overline{L(\calo_m)}_F$ on the tensor product $\pi\otimes \omega_{m,\al}$.

\begin{Def}
The {\bf Fourier-Jacobi module} of $\pi$ with respect to orbit $\calo_m$ and character $\psi_{m,\al}$ is the Jacquet module
\[
FJ_{m,\al}(\pi):=J_{V_1(\calo)}(\pi\otimes\omega_{m,\al}).
\]
It is naturally a $\overline{\ep}$-genuine smooth representation of $\overline{L(\calo)}_F \cong\overline{\Sp}_{2(r-m)}(F)$.
\end{Def}

The purpose for introducing Fourier-Jacobi modules rather than working with the Jacquet modules directly stems largely from the following properties.

\begin{Prop}
Set $H=\Stab_{N_{L(\calo_m)}(YV_2(\calo_m))}(\psi_{m,\al})$. As $YV_2(\calo_m)\rtimes H$-representations, we have an isomorphism
\[
FJ_{m,\al}(\pi) \cong J_{YV_2(\calo_m),\psi_{m,\al}}(\pi).
\]
\end{Prop}
\begin{proof}  Restricting $FJ_{m,\al}(\pi)$ to $YV_2(\calo_m)\rtimes H$, the action is preserved under Frobenius reciprocity. Therefore, applying Frobenius reciprocity \cite[Lemma 2.3.6]{GGS}, we obtain
\begin{align*}
FJ_{m,\al}(\pi) 		&= J_{V_1(\calo_m)}\left(\pi\otimes\ind_{YV_2(\calo_m)}^{V_1(\calo_m)}(\psi^{-1}_{m,\al})\right)\\
				&\cong J_{YV_2(\calo_m)}\left(\pi\otimes \psi^{-1}_{m,\al}\right).
\end{align*}
We are done, since as $H$-representations,
\[
J_{YV_2(\calo_m)}\left(\pi\otimes \psi^{-1}_{m,\al}\right)\cong J_{YV_2(\calo_m),\psi_{m,\al}}(\pi).\qedhere
\]
\end{proof}
In particular, we have that the standard Borel subgroup $B_{\calo_m}\subset L(\calo_m)_F$ is contained in $H=\Stab_{N_{L(\calo_m)}(YV_2(\calo_m))}(\psi_{m,\al})$, allowing us to conclude the following.
\begin{Prop}\label{Prop: FJ property}
Let $V\subset B_{\calo_m}$ be a unipotent subgroup of $\overline{L(\calo_m)}_F$, and suppose that $\psi_V$ is a character of $V$. Then we have 
\[
J_{V, \psi_V}(FJ_{m,\al}(\pi)) \cong J_{V,\psi_V}(J_{YV_2(\calo_m),\psi_{m,\al}}(\pi)).
\]
\end{Prop}

Note that the case that $V = \{I\}$ along with local root exchange implies the following local version of \cite[Lemma 1.1]{GRS2}:

\begin{Cor}\label{Cor: local equiv}
As representations of $V_2(\calo_m)_F$, we have isomorphisms
\[
FJ_{m,\al}(\pi)\cong J_{YV_2(\calo_m),\psi_{m,\al}}(\pi) \cong J_{V_2(\calo_m),\psi_{\calo_m}}(\pi)
\]
\end{Cor}
%%%%%%%%%%%%%%%%%%%%%%%%%%%%%%%%%%%%%%%%%%%%%%%%%%%%%%%%%%%%%%%%%%%%%%%%%%%%%%%%%%%%%%%%%%%%%%%%%%%%%%%%%%%%%%%%%%%%%%%%%%%%%%%%%%%%%%%%%%%%%%%%%%%%%%
%%%%%%%%%%%%%%%%%%%%%%%%%%%%%%%%%%%%%%%%%%%%%%%%%%%%%%%%%%%%%%%%%%%%%%%%%%%%%%%%%%%%%%%%%%%%%%%%%%%%%%%%%%%%%%%%%%%%%%%%%%%%%%%%%%%%%%%%%%%%%%%%%%%%%%%%%%%%%%%%%%%%%%%%%%%%%%%%%%%%
Thus, Fourier-Jacobi modules are isomorphic as $V_2(\calo)$-modules to the Jacquet module associated to $\calo_m$, implying that we may study one to ascertain vanishing of the other. The benefit of the Fourier-Jacobi modules is the smooth admissible $\overline{L(\calo_m)}_F$-action. This action allows us to apply the following observation:
\begin{Prop}\label{Prop: super unram}
Suppose that $\mathbb{G}(F) =G$ is a reductive group over a non-archimedean local field $F$ containing a full set of $n^{th}$ roots of unity $\mu_n(F)$, and suppose that $\Gn$ is a central extension of $G$ by $\mu_n(F)$ arising from $\overline{\mathbb{G}}\in CExt(\mathbb{G},\mathbb{K}_2)$. Let $(\pi,V)$  be a non-zero smooth admissible representation of $\overline{G}$. If $\pi$ is unramified, then $\pi$ is not supercuspidal.
\end{Prop}

\begin{proof}
The Satake isomorphism holds in this context and, as in \cite{McN} and \cite{GG}, one still has that if $(\tau, W)$ is an irreducible unramified representation, then there exists an unramified character $\chi: Z(\tilde{T})\to \cc^\times$ such that 
\[
\Hom_{\overline{G}}(\tau, \Ind^{\overline{G}}_{\overline{B}}(\rho(\chi)))\neq 0.
\]
In particular, $\tau$ cannot be cuspidal.
By \cite[Theorem 2.4.a]{BZ}, which is proven for reductive groups, but the proof goes through without change for metaplectic covers once we replace parabolic subgroups with their covers, we have that our representation decomposes 
\[
\pi = \pi_c\oplus \pi_c^\perp,
\] 
where $\pi_c$ is a cuspidal representation, and $\pi_c^\perp$ has no nontrivial cuspidal sub-quotients.  Since$(\pi,V)$ is unramified, it has at least one spherical composition factor $\tau_\pi$. This forces $\pi_c^{\perp}\neq 0$, and hence $\pi$ is not itself cuspidal. 
\end{proof}

\section{Root Exchange}\label{Section: Root Exchange}
\subsection{Global Root Exchange}
In this section we will introduce the statement of the general lemma on (global) root exchange. The main reference for this is \cite{GRS1}. Let $K$ be a number field with adele ring $\A$. Fix a unipotent subgroup $C\subset \G$, and let $\psi_C$ be a non-trivial character of $C(K)\bs C(\A)$.

Suppose there exist two unipotent $F$-subgroups $X,Y\subset \G$ satisfying the following six properties:
\begin{enumerate}
\item $X$ and $Y$ normalize $C$,
\item $X\cap C$ and $Y\cap C$ are normal in $X$ and $Y$, respectively, and $X\cap C\bs X$, $Y\cap C\bs Y$ are both abelian,
\item $X(\A)$ and $Y(\A)$, acting via conjugation, preserve $\psi_C$,
\item $\psi_C$ is trivial on $(X\cap C)(\A)$ and $(Y\cap C)(\A)$,
\item $[X,Y]\subset C$,
\item The pairing $(X\cap C)(\A)\bs X(\A)\times (Y\cap C)(\A)\bs Y(\A)\to \cc^\times$, given by
\[
(x,y) \mapsto \psi_C([x,y]),
\]
is multiplicative in each coordinate , non-degenerate, and identifies $(Y\cap C)(K)\bs Y(F)$ with the dual of $X(K)(X\cap C)(\A)\bs X(\A)$, and vice versa.
\end{enumerate} 
With this set up, set $B=YC$, $D=XC$ and extend $\psi_C$ trivially to both $B(K)\bs B(\A)$ and $D(K)\bs D(\A)$, referring to each separately as $\psi_B$ and $\psi_D$. Finally, set $A=BX=DY$.
\begin{Lem}\label{Lem: root exchange}
Suppose that the quadruple $(C,\psi_C, X, Y)$ satisfy the above criteria. If $f$ is an automorphic form on $\G$ of uniform moderate growth, then
\begin{equation}\label{eqn: int C}
\displaystyle \int_{[C]}f(cg)\psi_C^{-1}(c)dc \equiv 0, \quad \forall g\in A(\A)
\end{equation}
if and only if 
\begin{equation}\label{eqn: int B}
\displaystyle \int_{[B]}f(vg)\psi_B^{-1}(v)\,dv \equiv 0, \quad \forall g\in A(\A) 
\end{equation}
if and only if 
\begin{equation}\label{eqn: int D}
\displaystyle \int_{[D]}f(ug)\psi_D^{-1}(u)\,du \equiv 0, \quad \forall g\in A(\A),
\end{equation}
\end{Lem}
\begin{proof}
See \cite[Section7.1]{GRS1}.
%%%%%%%%%%%%%%%%%%%%%%%%%%%%%%%%%%%%%%%%%%%%%%%%%%%%%%%%%%%%%%%%%%%%%%%%%%%%%%%%%%%%%%%%%%%%%%%%%%%%%%%%%%%%%%%%%%%%%%%%%%%%%%%%%%%%%%%%%%%%%%%%%%%%%%%%%%%%%%%%%%%%%%%%%%%%%%%%%%%%%%%%%%%%%%%%%%%%%%%%%%%%%%%%%%%%%%%%%%%%%%%%%%%%%%%%%%
\end{proof}
\subsection{Local Root Exchange}
Suppose now that $F$ is a non-archimedean local field, and let $\G$ be the $F$-rational points of a split algebraic group over $F$ or finite cover thereof. As in the global setting,  we consider unipotent subgroups $C, X, Y\subset \G$, and a nontrivial character
\[
\psi_C:C\to\cc^\times.
\] 
As before, set $B=YC$, $D=XC$,  and $A=BX=DY$. We also assume that these subgroups satisfy the local analogue of the six properties preceding Lemma \ref{Lem: root exchange} (For a precise statement and proof, see \cite[Section 6.1]{C}).
\begin{Lem}\label{Lem: local root exchange}
Suppose that $\pi$ is a smooth representation of $A$, and extend the character $\psi_C$ trivially to $\psi_B$ on $B$ and $\psi_D$ on $D$. Then we have an isomorphism of $C$-modules
\[
J_{B,\psi_B}(\pi)\cong J_{D,\psi_D}(\pi).
\]
Moreover,
\[
J_{C,\psi_C}(\pi)=0\iff J_{B,\psi_B}(\pi)=0\iff  J_{D,\psi_D}(\pi)=0.
\]
\end{Lem}
\subsection{How one uses root exchange}

There are many contexts in which the above lemmas are useful, but we pause here to introduce a formalism for a particularly common application. We illustrate this application globally. Let $C$ and $\psi_C$ be as above, only now we assume that $C\subset U_r$. For a positive root $\be$, let $U_\be\subset U_r$ be the corresponding root subgroup. \\

Assume that $x_\be(t)\in C(\A)$ for all $t$ such that the induced character
\[
\mathbb{G}_a(K)\bs\mathbb{G}_a(\A)\xrightarrow{x_\be}C(K)\bs C(\A) \xrightarrow{\psi_C} \cc^\times
\]
is nontrivial. Let $\al, \ga \in \Phi$ such that the following properties hold:
\begin{enumerate}
\item Set $\Phi_C = \{\ga : x_\ga(t) \subset C\}$. Then for all $i,j\in \zz_{>0}$ such that $i\al+j\ga \in \Phi$, 
we have $$i\al+j\ga \in \Phi_C,$$
\item For all $\mu \in \Phi_C$ and all $i,j\in \zz_{>0}$ such that $i\al+j\mu \in \Phi$, we have $$i\al+j\mu \in \Phi_C,$$ and likewise for $\ga$ and $\mu$.
\item We have $$\be =\al+\ga,$$ 
\item Suppose $\be'\in \Phi_C$, with $\be'\neq \be$, such that the induced character on $x_{\be'}(t)$ is also non-trivial. Then there do not exist $i,j$ such that $$\be' = i\al+j\ga.$$
\end{enumerate}
Set  $B = U_\al\cdot C$, $D=U_\ga\cdot C$.\\

When the group $C$ and character $\psi_C$ are clear from the context, we shall refer to the ordered triple of roots $$(\al,\ga,\be)$$ as above as an {\bf exchange triple}. 

 Suppose then that we have the integral (\ref{eqn: int B}), and want to apply Lemma~\ref{Lem: root exchange} to obtain the integral (\ref{eqn: int D}). In this setting, we say that we apply root exchange to the exchange triple $(\al,\ga,\be)$, meaning that we apply Lemma~\ref{Lem: root exchange} to the quadruple $(C, \psi_C, \{x_\al(t)\},\{x_\ga(t)\})$.

Thus, the equivalence
\begin{align*}
\displaystyle \int_{[U_\al\cdot C]}f(vg)\psi_C^{-1}(v)\,dv \equiv 0, \quad \forall g\in \G 
\end{align*}
if and only if 
\begin{align*}
\displaystyle \int_{[U_\ga\cdot C]}f(ug)\psi_D^{-1}(u)\,du \equiv 0, \quad \forall g\in \G,
\end{align*}
from Lemma \ref{Lem: root exchange} tells us that, holding all else constant, we can exchange the integration over the root group of $\al$ with integration over the root group of $\ga$. When we apply root exchange to the triple $(\al, \be, \ga)$, we say that we exchange the root $\al$ for the root $\ga$.

\section{Vanishing Statements}\label{Section: vanishing}

 In this section we seek to prove part (\ref{part b}) of Theorem \ref{Thm: intro Nilpotent orbit}:

\begin{Thm}\label{Thm: vanishing}
Suppose that $\calo'$ is a nilpotent orbit such that either $\calo'> \calo_\Theta = (2^r)$ or $\calo'$ and $\calo_\Theta$ are not comparable. Then $\calo'$ does not support $\Theta_{2r}$.
\end{Thm}

The proof relies on global results of Ginzburg-Rallis-Soudry \cite{GRS2} and Jiang-Liu \cite{JL2}. If we can show that nilpotent orbits of the form 
$$\calo'=((2k)1^{2r-2k})$$ 
for $2\leq k\leq n$ do not support $\Theta_{2r}$, then by \cite[Lemma 2.6]{GRS2} , it follows that all orbits of the form 
$$\calo'=((2k)p_1^{e_1}\cdots p_l^{e_l})$$
 with $2\leq k\leq r$ also fail to support $\Theta_{2r}$. Combining this vanishing with Propositions 3.2 and 3.3 of Jiang-Liu \cite{JL2}, we now see that nilpotent orbits of type 
 $$\calo' = ((2k+1)^2p_1^{e_1}\cdots p_l^{e_l})$$
  for $2\leq k \leq r$ also fail to support $\Theta_{2r}$. This exhausts those orbits $\calo'$ such that either $\calo'> \calo_\Theta$ or $\calo'$ and $\calo_\Theta$ are not comparable. We are thus reduced to showing that the orbits of the form $\calo_k=((2k)1^{2r-2k})$, as well as $\calo'=(3^21^{2r-6})$, do not support $\Theta_{2r}$. In fact, it suffices to see that the single orbit $\calo_2=(41^{2r-4})$ fails to support $\Theta_{2r}$.
  
  We study the orbit $\calo_2$ by considering the corresponding local Fourier-Jacobi module $FJ_{2,\al}(\Theta_{2r,\nu})$. In particular, subject to mild restrictions on the local field $F=K_\nu$, we show that if this representation is non-zero, it must be both unramified and supercuspidal. This contradicts Proposition \ref{Prop: super unram}. Therefore, by Corollary \ref{Cor: local equiv}, we see that $\calo_2$ does not support $\Theta_{2r}$. The sufficiency of this result to prove Theorem \ref{Thm: vanishing} will become clear as the argument progresses.

\subsection{Local Vanishing results}\label{subsection: local vanishing}

Consider the nilpotent orbit $\calo_2 = (41^{2r-4})$.
The unipotent subgroup $V_1(\calo)$ is the unipotent radical of the standard parabolic subgroup of $\Sp_{2r}$ with Levi subgroup $L_{2,2r-4} \cong \GL_1^2\times \Sp_{2r-4}$, and the subgroup $V_1(\calo)$ is as follows:
\[
\left\{ \left( \begin{array}{cccc}
									u&X&Y\\
									&I_{2r-4}&X^*\\
									&&u^{-1}
\end{array}\right): u\in U_{\GL_2}; Y\in \Sym_J^{2}(F); X\in \Mat_{2\times (2r-4)}, X_{2,i}=0 \right\}.
\]
Here $U_{\GL_2}$ denotes the subgroup of upper triangular unipotent elements of $\GL_2$. The character $\psi_\calo$ may be chosen to be of the form
\[
\psi_{\calo,\al}(v) = \psi(v_{1,2}+\al v_{2,2r-1}),
\]
where $\al\in F^\times/(F^\times)^2$. 

 Let $\nu$ be a finite place of $K$ with with odd residue characteristic. In this case, both $\Theta_{2m}$ and $\omega_{2,\al}$ are unramified representations of $\overline{\Sp}_{2m-4}(K_\nu)$, so that the local Fourier-Jacobi module $FJ_{2,\al}(\Theta_{2m})$ is unramified as a representation of $\overline{\Sp}_{2m-4}(F)$ for all such places. 
We will work only with $F=K_\nu$ satisfying this property.

\begin{Prop}\label{Prop: base case}
Let $\Theta_{6}$ be the theta representation of $\overline{\Sp}_{6}(F)$. 
Then the Fourier-Jacobi module 
\[
FJ_{2,\al}(\Theta_6)=J_{V_1(\calo_2)}(\Theta_6\otimes\omega_{2,\al})
\]
is trivial.
\end{Prop}
\begin{proof}

In this case, one easily checks that $L(\calo)(F)$ is the embedded copy of $\SL_2(F)$ corresponding to the long simple root of $\Sp_6$. 
Thus, $FJ_{2,\al}(\Theta_6)$ is a smooth $\overline{\ep}$-genuine representation of $\overline{\SL}_2(F)$.  Let $U_{\al_3}$ be the root subgroup of $\Sp_6(F)$ corresponding to the unipotent radical in $\overline{\SL}_2(F)$.

By our assumption on $\nu$, it suffices to show that the representation is supercuspidal. This will follow if we can show that
\[
J_{U_{\al_3}}(J_{V_1(\calo_2)}(\Theta_6\otimes\omega_{2,\al})) = 0.
\]
By Proposition \ref{Prop: FJ property}, we see that this is equivalent to showing that 
\[
J_{U_{\al_3}}(J_{YV_2(\calo_2),\psi_{2,\al}}(\Theta_6))\cong J_{V, \psi_V}(\Theta_6) = 0.
\]
Here, we have $V = \{ v\in U_6 : v_{2,3}=v_{4,5}=0\}$, and $\psi_V(v) =\psi(v_{1,2}+\al v_{2,5})$.

Conjugating by the Weyl group element $w= w_1w_2$ and applying Lemma \ref{Lem: local root exchange} to the exchange triple $$ (-\ga_{1,1}, \ga_{1,2}, \ga_{2,2})=(-\al_1, \al_1+\al_2, \al_2),$$ we see that this module is isomorphic to $J_{U^0,\psi^0}(\Theta_6)$, where $U^0$ is the unipotent radical of the parabolic subgroup associated to the simple coroot $\al^\vee_1$ and $\psi^0(v) = \psi(v_{2,3}+\al v_{3,4})$.

Since $\Theta_6$ is not generic (see \cite[Section 5]{Gao}), we see that this Jacquet module is isomorphic to $J_{U_6,\psi'}(\Theta_6)$, where $\psi'$ is the trivial extension of $\psi^0$ to the full unipotent radical. This follows from noting that the $\SL_2$-factor of the Levi subgroup corresponding to the simple coroot $\al^\vee_1$ preserves the character $\psi^0$, and thus we may view this Jacquet module as an $\overline{\SL}_2(F)$-representation. As $\Theta_6$ is not generic, this $\overline{\SL}_2(F)$-representation has no non-trivial twisted Jacquet module. Therefore, upon restricting to the unipotent radical, this representation is isomorphic to its Jacquet module.

Note that this Jacquet module factors through the constant term of type $\GL_1\times \Sp_4$, and by Theorem \ref{Thm: local constant term}, we see that
\[
J_{U_6,\psi'}(\Theta_6)\cong J_{U_4, \psi_{(4),\al}}(\Theta_4)= 0,
\]
as $\Theta_4$ is also not generic.\qedhere
\end{proof}

It follows from Corollary \ref{Cor: local equiv} that the Jacquet module associated to $\calo_2$ vanishes.

\begin{Thm}\label{Thm: supercuspidal}
Let $r\geq3$ and let $\Theta_{2r}$ be the theta representation of $\GF$. Then the Fourier-Jacobi module
\[
FJ_{2,\al}(\Theta_{2r})=J_{V_1(\calo_2)}(\Theta_{2r}\otimes\omega_{2,\al})
\]
vanishes.
\end{Thm}

We shall prove this by induction, noting that Proposition \ref{Prop: base case} forms the base case. Thus, we assume the theorem holds for $r$. In light of this assumption, we are able to say much about certain Jacquet modules on $\Theta_{2r+2}$:

\begin{Prop}\label{Prop: internal induction}
Suppose that Theorem \ref{Thm: supercuspidal} holds for $m\leq r$.
Then for all $3\leq k\leq m\leq r+1$, the nilpotent orbits $\calo_{m,k}= ((2k)1^{2(m-k)})$ do not support $\Theta_{2m}$.
\end{Prop}
\begin{proof}
The case $k=m=3$ is proven so by induction we may assume that it is proven up to $m$. 

By Corollary \ref{Cor: local constant term}, we know that there is a value $s_0\in \cc$ such that there is a surjection
\[
\Ind(\delta^{s_0}\otimes\Theta_{2m}) \twoheadrightarrow \Theta_{2m+2}.
\]
Therefore, it suffices to show that 
\[
\Hom_{V_2(\calo_{m+1,k})}(\Ind(\delta^{s_0}\otimes\Theta_{2m}),\psi_{\calo_{m+1,k}}) =0.
\]

For this task, we apply \cite[Theorem 5.2]{BZ}, and study the double cosets  
\[
P_{(1,m)}(F)\backslash \Sp_{2m+2}(F)/\Sp_{2(m-k-2)}(F)V_2(\calo_{m+1,k})
\]
One can check, using the Bruhat decomposition, that the coset representatives may be chosen to be Weyl group elements or Weyl group elements times certain unipotent elements.

As in the proof of \cite[Lemma 7]{BFG1}, most cases do not contribute as the Weyl group element conjugates a root of $V_2(\calo_{m+1,k})$ where the character is nontrivial into the unipotent radical of $P_{(1,m)}$. There are only two cases that need to be considered:
\begin{enumerate}
\item
The representatives of the form $w=w_1w_2\cdots w_{m+1} u$ give rise to characters corresponding to orbits $\calo_{m',k'}$ where $m'<m+1$. By induction, these also vanish.
\item
The relative long element $w=w_1w_2\cdots w_mw_{m+1}w_m\cdots w_1$ gives rise to the character corresponding to $(41^{2m-4})$ on $\Theta_{2m}$. This vanishes by the assumption that Theorem \ref{Thm: supercuspidal} holds for $m$, and we are done.\qedhere
\end{enumerate}
\end{proof}

\begin{Cor}\label{Cor: 3^2}
The orbit $(3^21^{2r-4})$ does not support $\Theta_{2r+2}$.
\end{Cor}
\begin{proof}
Set $\calo=(3^21^{2r-4})$. Then we have $V_2(\calo) = U_{(2,r-1)}$ and we may choose the generic character $$\psi_\calo(v) = \psi(v_{1,3}+ v_{2,2r}).$$ Supposing that  $J_{\psi_\calo}(\Theta_{2r+2})$ does not vanish, it follows that it gives to an unramified representation of  $\overline{L(\calo)}_F=\overline{\SL}_2(F)^\Delta \times\overline{\Sp}_{2r-4}(F)$. Restricting this representation to the diagonally embedded $\SL_2$, we see there exists a character $\chi$ of $Z(\widetilde{T}_1)\subset \overline{\SL}_2(F)$ such that by Frobenius reciprocity,
\[
0\neq \Hom_{\overline{\SL}_2(F)}(J_{V_2(\calo),\psi_\calo}(\Theta_{2r+2}), \Ind(\chi))\cong  \Hom_{\widetilde{T}_1}(J_{V_2(\calo)N,\psi_\calo}(\Theta_{2r+2}), \iota(\chi)).
\]
We begin studying the Jacquet module $J_{V_2(\calo)N,\psi_\calo}(\Theta_{2r+2})$ by conjugating the Weyl group element $w_2$ and applying a root exchange with triple $(\mu_3,-\al_2,\eta_{2,2})$. This removes the diagonal embedding along the two root groups associated to $\al_1$ and $\mu_3$ arising from $N$.

%switched the upper bound on the first set of root exchanges from r+1 (which is undefined) to r.
Applying the root exchanges associated to the exchange triples $$(\ga_{3,i}, \eta_{2,i},\eta_{2,2})$$ for $3\leq i\leq r$ in increasing values of $i$ followed by the exchange triples $$(\eta_{3,j},\ga_{2,j},\eta_{2,2})$$ for $3\leq j\leq r$ in decreasing values for $j$, one finds that this constant term is isomorphic to $J_{U',\psi'}(\Theta_{2r+2})$ where 
\[
U' =\{ u\in V_2(421^{2r-4}): u_{2,3} = 0\},%changed u_{1,2} to ${2,3} and vice versa in character
\]
and $\psi'(u) = \psi(u_{1,2} + u_{3,2r})$. Conjugating by the standard lift of the simple reflection associated to the long root $\mu_3$,
and applying root exchange to the triple $(-\mu_3,\eta_{2,2},\ga_{2,2})$, we obtain an isomorphism with the Jacquet module $J_{L_{2,r-1},\psi_2}(\Theta_{2r+2})$.

 However, the vanishing of of the Jacquet modules associated to $((2k)1^{2(r+1-k)})$ for $k>2$ from the proposition forces this final Jacquet module to be trivial. To see this, one applies an inductive argument identical to the global proof of Lemma \ref{Lem: lemma 1}. This is not circular as the argument in the proof of Lemma \ref{Lem: lemma 1} requires only the vanishing of the coefficients associated to the aforementioned orbits. By way of a contradiction, we find that
\[
J_{V_2(\calo),\psi_\calo}(\Theta_{2r+2})=0.\qedhere
\]
\end{proof}

Having established that the inductive hypothesis forces vanishing of the Jacquet modules associated to $\calo_k$ for $k>2$, as well as for the orbit $\calo'=(3^21^{2r-4})$, we may now consider $\calo_2$.
\begin{Prop}
As a representation of $\overline{\Sp}_{2r-2}(F)$, the Fourier-Jacobi module $FJ_{2,\al}(\Theta_{2r+2})$ is supercuspidal.
\end{Prop}
\begin{proof}
The argument mirrors the one in the base case. Let $\al_0 = \mu_3 = 2\al_3+\cdots +2\al_n+\al_{r+1}$ be the long root corresponding to the highest root in the embedded copy of $\Sp_{2r-2}$. As the root group $U_{\al_0}$ corresponding to this root is contained in the unipotent radical of each maximal parabolic of $\Sp_{2r-2}$, to show that the Jacquet module is supercuspidal, it suffices to see that 
\[
J_{U_{\al_0}}(FJ_{2,\al}(\Theta_{2r+2}))=0.
\]
To demonstrate this, we note that by Proposition \ref{Prop: FJ property} it suffices to show that
\[
J_{U_{\al_0}}(J_{YV_1(\calo_2),\psi_{2,\al}}(\Theta_{2r+2}))\cong J_{V,\psi_V}(\Theta_{2r+2})= 0.
\]
In the above isomorphism, we have used root exchange to remove the roots $\eta_{2,k}$ for $3\leq k\leq n$, which fit into the exchange triples
$
(\ga_{2,k},\eta_{2,k}, \mu_2).\\
$

Conjugating the subgroup $V =V_2(421^{2r-4})$ by the Weyl group element $w=w_1w_2$, and then applying root exchange to the exchange triple $(-\al_1,\al_1+\al_2, \al_2)$, we see that this Jacquet module is isomorphic to $J_{V^w,\psi_{V^w}}(\Theta_{2r+2})$, where
\[
V^w = \left\{ 
				v=	\left(\begin{array}{ccc}
									Z&X&Y\\
									&I_{2r-4}&X^*\\
									&&Z^*

\end{array}\right)\right\}, 
\] where $Y\in \Sym_J^{3}(F),$ $X\in\Mat_{3\times 2r-4}(F)$ such that  $X_{ij}\neq0$ implies $i=2$, and $Z\in U(\GL_3)$ is of the form
\[
Z = \left(\begin{array}{ccc}
									1&&\ast\\
									&1&\ast\\
									&&1

\end{array}\right).
\]
Here, $\psi_{V^w}(v) =  \psi(v_{2,3}+ \al v_{3,2r}).$ It should be noted that $J_{V^w,\psi_{V^w}}(\Theta_{2r+2})$ is a representation of $\overline{\Sp}_{2r-4}(F)$.
 
Consider now the abelian subgroup of $U_{2r+2}$
\[
Z = \left\{ z=	\left(\begin{array}{ccc}
									1&X&0\\
									&I_{2r}&X^*\\
									&&1

\end{array}\right) : X \in F^{2r} \mbox{  such that  } x_1=x_2=x_{2r-1}=x_{2r}=0\right\}.
\]
The subgroup $\overline{\Sp}_{2r-4}(F)$ acts by conjugation on the characters of $Z$ with two orbits. Expanding $J_{V^w,\psi_{V^w}}(\Theta_{2r+2})$ along $Z$, we see that all representatives of the nontrivial orbit vanish as they are isomorphic to a quotient of the Jacquet module associated to the nilpotent orbit $(3^21^{2r-4})$, which were shown to vanish on $\Theta_{2r+2}$ in Corollary \ref{Cor: 3^2}. 

Therefore, the only term which remains is the trivial orbit, so that our Jacquet module vanishes if and only if $J_{ZV^w,\psi^0}(\Theta_{2r+2})$ vanishes, where
\[
\psi^0(v) =  \psi(v_{2,3}+ \al v_{3,2r})
\]
is the trivial extension of $\psi_{V^w}$ to this larger unipotent subgroup.

Finally, the vanishing of $(61^{2r-4})$ forces this Jacquet module to be isomorphic to $J_{R,\psi^0}(\Theta_{2r+2})$, where $R=U_{\al_1}ZV^w = V_2(61^{2r-4})$. Factoring this Jacquet module through the unipotent radical of the maximal parabolic of type $\GL_1\times \Sp_{2r}$,  and applying Theorem \ref{Thm: local constant term}, we see that
\[
J_{R,\psi^0}(\Theta_{2r+2})\cong J_{V_2(41^{2r-4}),\psi_{(41^{2r-4})}}(\Theta_{2r}),
\]
which vanishes by induction.

Thus, we have shown that $$FJ_{2,\al}(\Theta_{2r+2})$$ is supercuspidal as a $\overline{\Sp}_{2r-2}(F)$ representation.
\end{proof}
 
As remarked above, $FJ_{2,\al}(\Theta_{2r})=0$ by Proposition \ref{Prop: super unram}. By Corollary \ref{Cor: local equiv}, we have proven the Theorem \ref{Thm: supercuspidal}. Combining this with Proposition \ref{Prop: internal induction}, Corollary \ref{Cor: 3^2}, and the argument above, we have proven Theorem \ref{Thm: vanishing}.\qed

\section{Non-Vanishing Statement}\label{Section: Non-vanishing}

 We turn now to part (\ref{part a}) of Theorem \ref{Thm: intro Nilpotent orbit}. For the orbit $\calo_\Theta :=\calo_{\Theta,r}=(2^r)$ on $\Sp_{2r}$, we have $$V_2(\calo_\Theta) = \left\{ \left( \begin{array}{cc}
									I_{r}&Y\\
									&I_{r}\\\end{array}\right): Y\in \Sym_{J}^r(\A)\right\}.$$
For $v\in V_2(\calo_\Theta)$, we define the character $\psi_{\calo_\Theta, \underline{\ep}} :V\to \cc^\times$ by
\[
\psi_{\calo_\Theta,\underline{\ep}}(v) = \psi(\ep_1 v_{1,2r} +\cdots + \ep_r v_{r,r+1}),
\]
 where $\ep_i\in K^\times$. With this notation, the Fourier coefficients of $\Theta_{2r}$ associated to the nilpotent orbit $\calo$ are given by
\begin{eqnarray}\label{eqn: theta int}
\displaystyle \int_{[V_2(\calo_\Theta)]} \theta(vg) \psi_{\calo_\Theta,\underline{\ep}}(v)\,dv.
\end{eqnarray}

By the automorphy of $\theta$, we have the freedom to conjugate the above integral by various elements $t\in T_{r}(K)$ which, after a change of variables, implies that our coefficient only depends on $\ep_i \pmod{(K^\times)^2}$. Additionally, $M(\calo_\Theta) \cong \GL_{r}$ is the Siegel Levi subgroup of $\Sp_{2r}$ and one can show that for each choice $\underline{\ep}$,  $L(\calo_\Theta)_K = \Stab_{M(\calo_\Theta),K}(\psi_{\calo_\Theta,\underline{\ep}})$ is isomorphic to  $\OO_{\underline{\ep}}(K)$, the orthogonal group on $K^{r}$ with respect to the quadratic form
\[
q_{\underline{\ep}}(v) =\ep_1v_1^2+\cdots  +\ep_{r}v_{r}^2.
\]

\begin{Thm}\label{Thm: non-vanishing}
The nilpotent orbit $\calo_{\Theta}=(2^r)$ supports $\Theta_{2r}$.
\end{Thm}

 This will be proven in two stages: first proving non-vanishing for $\overline{\Sp}_{4r}$, then using this to obtain the result for $\overline{\Sp}_{4r+2}$.

%%%%%%%%%%%%%%%%%%%%%%%%%%%%%%%%%%%%%%%%%%%%%%%%%%%%%%%%%%%%%%%%%%%%%%%%%%%%%%%%%%%%%%%%%%%%%%%%%%%%%%%%%%%%%%%%%%%%%%%%%%%%%%%%%%%%%%%%%%%%%%%%%%%%%%%%%%%%%%%%%%%%%%%%%%%%%%%%%%%%%%%%%%%%%%%%%%%%%%%%%%%%%%%%%%%%%%%%%%%%%%%%%%%%%%%%%%%%%%%%%%%%%%%%%%%%%%%%%%%%%%%%%%%%%%%%%%%%%%%%%%%%%%%%%%%%%%%%%%%%%%%%%%%%%%%%%%%%%%%%%%%%%%%%%%%%%%%%%%%%%%%%%%%%%%%%%%%%%%%%%%%%%%%%%%%%%%%%%%%%%%%%%%%%%%%%%%%%%%%%%%%%%%%%%%%%%%%%%%%%%%%%%%%%%%%%%%%%%%%%%%%%%%%%%
%%%%%%%%%%%%%%%%%%%%%%%%%%%%%%%%%%%%%%%%%%%%%%%%%%%%%%%%%%%%%%%%%%%%%%%%%%%%%%%%%%%%%%%%%%%%%%%%%%%%%%%%%%%%%%%%%%%%%%%%%%%%%%%%%%%%%%%%%%%%%%%%%%%%%%%%%%%%%%%%%%%%%%%%%%%%%%%%%%%%%%%%%%%%%%%%%%%%%%%%%%%%%%%%%%%%%%%%%%%%%%%%%%%%%%%%%%%%%%%%%%%%%%%%%%%%%%%%%%%%%%%%%%%%%%%%%%%%%%%%%%%%%%%%%%%%%%%%%%%%%%%%%%%%%%%%%%%%%%%%%%%%%%%%%%%%%%%%%%%%%%%%%%%%%%%%%%%%%%%%%%%%%%%%%%%%%%%%%%%%%%%%%%%%%%%%%%%%%%%%%%%%%%%%%%%%%%%%%%%%%%%%%%%%%%%%%%%%%%%%%%%%%%%%%%%%%%%%%%%%%%%%%%%%%%%%%%%%%%%%%%%%%%%%%%%%%%%%%%%%%%%%%%%%%%%%%%%%%%%%%%%%%%%%%%%%%%%%%%%%%%%%%%%%%%%%%%%%%%%%%%%%%%%%%%%%%%%%%%%%
\begin{Prop}\label{Prop: nonvan even}
The nilpotent orbit $\calo_{\Theta}=(2^{2r})$ supports $\Theta_{4r}$.
\end{Prop}
\begin{proof}

Set $\calo=\calo_\Theta$. To prove the proposition, it is sufficient to show that there is at least one $M(\calo)_K$-orbit of characters such that (\ref{eqn: theta int}) is non-vanishing. To this end, suppose $\underline{\ep}$ is such that $\OO_{\underline{\ep}}(K)$ is split. This is equivalent to the assumption that, up to conjugation by a Weyl group element, for each $1\leq i \leq 2r-1$ odd, $$\ep_i\ep_{i+1} =-\al_i^2$$ for some $\al_i\in K^\times$. With this choice, we refer to $\psi_{\calo,\underline{\ep}}$ as the split character. With this choice, there exists an $l\in M(\calo)_K$ such that, conjugating the integral (\ref{eqn: theta int}) by $l$, we see that non-vanishing of our Fourier coefficient for the choice of split character is equivalent to the non-vanishing of
\begin{align}\label{int base?}
\displaystyle \int_{[V_2(\calo)]} \theta(vg) \hat{\psi}_\calo(v)\,dv,
\end{align}
where 
\[
\hat{\psi}_\calo(v) = \psi( v_{2r-1,2r+1} + v_{2r-3,2r+3}+\cdots + v_{1,4r-1}).
\]
Note that the roots with a nontrivial character on their root subgroups are

 $$\be_k = \eta_{2r-2k+1,2r-2k+1}=\al_{2r-2k+1}+2\al_{2r-2k+2}+\cdots +\al_{2r},$$ 
for $1\leq k \leq 2r$.

Thus, to get to an integral we can work with, for each $1\leq i \leq2r$ set $r_i$ to be the standard lift of the simple reflection for the long root $\mu_i$; namely
\[
r_i = \left( \begin{array}{ccccc}
									I_{i-1}&&&&\\
									&&&-1&\\
									&&I_{4r-2i}&&\\
									&1&&&\\
									&&&&I_{i-1}\\\end{array}\right).
\]

The idea is to iteratively move the character to the appropriate simple roots. We begin by conjugating (\ref{int base?}) by $r_{2r}$. 
This moves the character from the root $\be_1$ to the simple root $\al_{2r-1}$, but also moves $\al_{2r}$ to $-\al_{2r}$. We then need to exchange the integration over this negative root with an integration over $\be_1$, which is permitted by root exchange. That is, we apply Lemma \ref{Lem: root exchange} to the exchange triple $$(-\al_{2r}, \be_1, \al_{2r-1}).$$ After this first step, we obtain the integral
\begin{align*}
\displaystyle \int_{[R_{1}]} \theta(vg) \hat{\psi}_{1}(v)\,dv,
\end{align*}
where 
\[
R_{1} = \left\{ \left( \begin{array}{cccccc}
									I_{2r-2}&&X_1&&X_2&Y\\
									&1&t&y&z&X_2^*\\
									&&1&0&y&\\
									&&&1&-t&X_1^*\\
									&&&&1&\\
									&&&&&I_{2r-2}
\end{array}\right): Y\in \Mat_{2r-2}^0; X_i\in \A^{2r-2};\:t,y,z\in \A\right\},
\]
and 
\[
 \hat{\psi}_{1}(v)= \psi( v_{2r-1,2r} + v_{2r-3,2r+3}+\cdots + v_{1,4r-1}).
\]

We proceed inductively and assume that we have shown that, for $1\leq k\leq r-1$, the non-vanishing of (\ref{eqn: theta int}) is equivalent to the non-vanishing of 
\begin{align}\label{eqn: induction vanish}
\displaystyle \int_{[R_{k}]} \theta(vg) \hat{\psi}_{k}(v)\,dv,
\end{align}
where 
\[
R_{k} = \left\{ \left( \begin{array}{ccc}
									I_{2r-2k}&X&Y\\
									&M_k&X^*\\
									&&I_{2r-2k}\\
									
\end{array}\right): Y\in \Sym_J^{2r-2k}; X_i\in \Mat^E_{(2r-2k)\times(4k)} \: M_k \in N_k\right\},
\]
 
\[
 \Mat^E_{l\times r}= \{ X \in \Mat_{l\times r}: x_{i,j}\neq0 \Rightarrow j \mbox{ even}\},
\]
and
\[
N_k = \{ M\in U(\Sp_{4k}) : \mbox{for } i<j\leq 4k-i+1, m_{i,j}\neq 0 \Rightarrow i\mbox{ odd}\}.
\]
Here we are considering the natural embedding $\Sp_{4k}\hra \Sp_{4r}$ into the Levi subgroup of $P_{(2r-2k,2k)}$, and the character is given by
\begin{align*}
 \hat{\psi}_{k}(v) &= \psi(v_{2r-1,2r} + v_{2r-3,2r-2}+\cdots +v_{2r-2k+1,2r-2k+2})\\
				 &\quad\times \psi(v_{2r-2k-1, 2r+2k+1}+\cdots + v_{1,4r-1}).
\end{align*}
We need to show that the non-vanishing of (\ref{eqn: induction vanish}) is equivalent to the non-vanishing of 
\[
\displaystyle \int_{[R_{k+1}]} \theta(vg) \hat{\psi}_{k+1}(v)\,dv.
\]
To this end,  the next root we need to move is 
$$\be_{k+1} = \al_{2r-2(k+1)+1}+2\al_{2r-2k}+\cdots +\al_{2r}.$$ 

Before conjugating by $r_{2r-2k}$, we apply root exchange to all the root groups in $V_k$ which would be conjugated to negative roots. Thus, we apply root exchange to the triples of type
 $$(\ga_{2(r-k),2(r-i)},\;\eta_{2(r-k)-1,2(r-i)},\;\be_{k+1})$$
 in the order of descending values of $i$, as $i$ ranges over $1\leq i\leq r-1$. These correspond to the roots contained in the bottom row of the ``even columns'' of $X$ which are contained in the Siegel Levi subgroup. 

 Next, we apply root exchange on the triples of type
 $$(\eta_{2(r-k),2(r-i)},\;\ga_{2(r-k)-1,2(r-i)},\;\be_{k+1})$$
 in order of descending values of $j$, where here $1\leq j \leq r$. These roots correspond to those even columns not contained in the Siegel Levi.

 Finally, we apply root exchange to the triple
 $$(\mu_{2(r-k)},\; \al_{2r-2k-1},\;\be_{k+1})$$ 
to obtain the integral
\[
\displaystyle \int_{[R_{k}]} \theta(vg) \hat{\psi}_{k}(v)\,dv,
\]
where
\[
R_{k} = \left\{ \left( \begin{array}{ccc}
									I_{2r-2(k+1)}&X&Y\\
									&M_{k+1}&X^*\\
									&&I_{2r-2(k+1)}\\
									
\end{array}\right): Y\in \Sym_J^{2r-2k};\: M_{k+1} \in N_{k+1}\right\},
\]
where
\[
X\in \Mat^*_{l\times m}= \{ X \in \Mat_{l\times m}: x_{i,j}\neq0 \mbox { and } j<m-1 \Rightarrow j \mbox{ even}\},
\]
and where the character has the same formula as before. Conjugating this integral by $r_{2r-2k}$ we see that the resulting integral is
\[
\displaystyle \int_{[R_{k+1}]} \theta(vg) \hat{\psi}_{k+1}(v)\,dv,
\]
thus finishing our induction. Thus, we have shown that the non-vanishing of (\ref{eqn: theta int}) is equivalent to the non-vanishing of 
\begin{align}\label{eqn: final int nonvan}
\displaystyle \int_{[R_{r}]} \theta(vg) \hat{\psi}_{r}(v)\,dv,
\end{align}
where, in this final form, we recall that
\[
R_r = N_r= \{ m\in U_{2r} : \mbox{for } i<j\leq 4r-i+1, m_{i,j}\neq 0 \Rightarrow i\mbox{ odd}\},
\]
and 
\[
 \hat{\psi}_{r}(v)=\psi(v_{1,2} + v_{3,4}+\cdots +v_{2r-1,2r}).
\]

We claim that this integral is not trivial for some choice of data. If not, setting $U_{(2^r)}$ to be the unipotent radical of the standard parabolic $P_{(2^r)}$ which has Levi subgroup $L\cong \GL_2^r$, one may show (using the results in Appendix \ref{Appendix B}) that (\ref{eqn: final int nonvan}) is equal to the constant term
\begin{align}\label{eqn: final even rank}
\int_{[R_{r}]} \theta(vg) \hat{\psi}_{r}(v)\,dv= \int_{\left[U_{(2^r)}\right]}\int_{[R_{r}]} \theta(uvg) \hat{\psi}_{r}(v)\,dvdu= \int_{[U_{2r}]}\theta(vg)\psi_{2r}^0(v)\,dv,
\end{align}
which is also zero for all choices of data. Here $\psi_{2r}^0$ is as in Section \ref{Section: Notation}. In particular, the integral vanishes when $g=1$. Applying Theorem \ref{Thm: global constant term}, we see that this integral equals, using the notation in \cite{C},
\[
\displaystyle\int_{[U_{\GL_{2r}}]} \theta_{\GL_{2r}}(u)\psi_\lam(u) du,
\]
where $\theta_{\GL_{2r}}\in \Theta^{(2)}_{\GL_{2r}}$ is a theta function on the double cover of $\GL_{2r}$, and $U_{\GL_{2r}}$ is the unipotent radical of the Borel subgroup of $\GL_{2r}$. As noted in Section \ref{Section: local Levi}, we are working with a twist of a pushout of this double cover. The analysis of Bump and Ginzburg in \cite{BG} on the global theta representation applies to this cover without essential change.

 This integral is the $\lam$-semi-Whittaker  coefficient of $\theta_{\GL}$ corresponding to the partition $\lam=(2^r)$. 
This coefficient is non-vanishing on $ \Theta^{(2)}_{\GL_{2r}}$ as seen in \cite{BG}, so that our assumption gives a contradiction. 
 Thus, there exists a choice of data such that the integral (\ref{eqn: final int nonvan}) is non-vanishing, and the proposition is proved.
\end{proof}

To finish the proof of Theorem \ref{Thm: non-vanishing}, we consider the odd rank case:
\begin{Prop}\label{Prop: odd rank}
The orbit $\calo_{\Theta}= (2^{2r+1})$ supports $\Theta_{4r+2}$.
\end{Prop}
\begin{Rem}
If one sought to prove analogous results for higher degree covers, a critical obstruction occurs at this stage. The non-vanishing of the Fourier coefficient (\ref{eqn: SL2 int}) below follows from standard results for automorphic forms on $\SL(2)$. In the case of studying Fourier coefficients of $\Theta^{(2n)}_{2r}$ with $n>2$, one must to prove genericity results for residues of certain Eisenstein series to obtain the analogous non-vanishing result. Accomplishing this is a fundamental obstacle to considering generalizations to higher degree covers.  
\end{Rem}\begin{proof}
Now set $\calo=(2^{2r+1}),$ and consider the integral 
\begin{align}\label{eqn:int 1 odd}
\displaystyle\int_{[V_2(\calo)^0]}\theta(vg)\psi_{0}(v)\,dv,
\end{align}
where 
\[
V_2(\calo)^0= \left\{ \left( \begin{array}{cc}
									I_{2r+1}&Y\\
									&I_{2r+1}\\\end{array}\right): Y\in \Sym_J^{2r+1}(\A), \mbox{ and } y_{2r+1,1} = 0\right\},
\]
and 
\[
\psi_{0}(v) =  \psi(\ep_1 v_{2,4r+1} +\cdots + \ep_m v_{2r+1,2r+2}).
\]
Here we make the same assumption on $\ep_i$ as before, so that the corresponding orthogonal group is split. We claim that this integral is non-vanishing for some choice of data. If not, then the integral
\[
\displaystyle\int_{[V_{1,2r}]}\int_{[V_2(\calo)^0]}\theta(uvg)\psi_{0}(v)\,dvdu,
\]
where $V_{1,2r}$ is the unipotent radical of the maximal parabolic in $\Sp_{4r+2}$ with Levi $L_{1,2r}\cong \GL_1\times\Sp_{4r}$, is zero for all choices of data.  However,  Theorem \ref{Thm: global constant term} tells us that this integral is equal to 
\[
\int_{[V(2^{2r})]}\theta_{4r}(v)\psi_{\underline{\ep}}(v)\,dv,
\]
which is non-vanishing for some choice of data by Proposition \ref{Prop: nonvan even}. Thus, the integral (\ref{eqn:int 1 odd}) is non-vanishing for some choice of data.

As before, our assumption that the stabilizer $\OO_{\underline{\ep}}(K)$ is a split group tells us that there exists an $h\in M(\calo)_K$ such that if we conjugate (\ref{eqn:int 1 odd}) by $h$ we obtain

\begin{align*}
\displaystyle\int_{[V_2(\calo)^0]}\theta(vg')\hat{\psi}_{0}(v)\,dv,
\end{align*}
where 
\[
\hat{\psi}_{0}(v)=\psi( v_{2r,2r+1} + v_{2r-2,2r+3}+\cdots + v_{3,4m-3}).
\]
Note that the roots with a nontrivial character on the corresponding root group are 
\[
\eta_{2k,2k} = \al_{2k}+2\al_{2k+1}+\cdots+ \al_{2r},
\]
for $1\leq k\leq r$. 

Consider the following  $r$ root triples
\[
(\eta_{1,2k},-\ga_{1,2k-1}, \eta_{2k,2k}).
\]
One can check that these are in fact exchange triples, to which we apply Lemma \ref{Lem: root exchange}. Thus, we have that the function
\begin{align*}
f(g) =\displaystyle\int_{[V']}\theta(vg)\hat{\psi}_{0}(v)\,dv,
\end{align*}
is a non-zero function in $g$.
Here, we have
\[
V'= \left\{ \left( \begin{array}{cccc}
									1&&&\\
									u&I_{2r}&&\\
									&&I_{2r}&\\
									&&u^*&1\\
\end{array}\right)
						\left(\begin{array}{cc}
									I_{2r+1}&Y\\
									&I_{2r+1}\\\end{array}\right)\right\}, 
\]
with
\begin{align*}
Y\in \Sym_J^{2r+1}(\A), \:&\mbox{ and }\qquad y_{i,1} = 0 \mbox{ for $i$ odd},\\
u\in \A^{2r}\qquad &\mbox{ with }\qquad u_j = 0 \mbox{ for $j$ even.}
\end{align*}

The point of this last root exchange is that when we restrict $g\in \overline{\Sp}_{4r+2}(\A)$ to the copy of $\overline{\SL}_2(\A)$ corresponding to the highest root, $f$ is a nonzero automorphic function on this $\overline{\SL}_2$. That is, $f$ is invariant under the $K$-points of 
\[
\iota\left(\begin{array}{cc}
				a&b\\
				c&d\\
				\end{array}\right) = 
	\left(\begin{array}{ccc}
				a&&b\\
				&I_{4r}&\\
				c&&d\\
				\end{array}\right).
\]
As a nonzero automorphic function of $\overline{\SL}_2(\A)$, $f\circ \iota$ has a non-trivial Fourier coefficient; thus, there exists an $\xi\in K^\times$ such that the integral
\begin{align}\label{eqn: SL2 int}
\displaystyle \int_{[\A]} f\left(\iota\left(\begin{array}{cc}
				1&t\\
				&1\\
				\end{array}\right)g\right)\psi(\xi t) \,dt 
\end{align}
is not zero.

Reversing the previous root exchanges, and conjugating by an appropriate element in $M(\calo)_K$, we see that the non-vanishing of (\ref{eqn: SL2 int}) is equivalent to the non-vanishing of 
\[
\displaystyle \int_{[V_2(\calo)]} \theta(vg) \psi_{\calo,\underline{\ep}}(v)\,dv.\qedhere
\]

\end{proof}

As a final note, we remark that the equality analogous to (\ref{eqn: final even rank}) holds in the odd rank case. We record this as follows.
\begin{Prop}\label{Prop: Split coefficient}
Set $\calo=(2^{r})$. Then the integral
\[
\int_{[V_2(\calo)]} \theta_{r}(vg)\psi_{\calo,\ep}(u)\,du \neq 0
\]
if and only if
\[
\theta^{U_{r},\psi_{r}^0}(g) = \int_{[U_{r}]}\theta(vg)\psi_{r}^0(v)\,dv\neq 0.
\]
\end{Prop}

We refer to the function $\theta^{U_{r},\psi_{r}^0}(g)$ as the \emph{split coefficient} of $\theta$. 
\section{Local Correspondence}\label{Section: local correspondence}
In this section, we study the local restriction problem in the case that $\pi$ and $\tau$ are both principal series representations. More precisely, let $W_0$ be a symplectic space over our local field $F$. We require the same restrictions on $F$ imposed  in Section \ref{Section: vanishing}. For example, $F$ may be the completion of a number field $K$ at a finite place of odd residue characteristic. This ensures that the results of that section hold for theta representations over $F$.

Let $W_n = X_n\oplus W_0\oplus X_n^\ast$ be the $\dim(W) +2n$ dimensional symplectic space as in the introduction. Setting $V_n = X_n\oplus X_n^\ast$, we have the embedding of metaplectic covers
\[
\overline{\Sp}(W_0)_F\times_{\mu_4}\overline{\Sp}(V_n)_F \to \overline{\Sp}(W_n)_F.
\]
Setting $\dim(W_0) = 2m$, this embedding restricts to an embedding 
\[
Z(\overline{T}_m)\times_{\mu_4}Z(\overline{T}_n)\to Z(\overline{T}_{m+n}).
\]
 As noted in Section \ref{Section: local theory}, there are two unramified theta representations, $\Theta_{2(m+n)}$, for $\overline{\Sp}(W_n)_F$ corresponding to the two choices of genuine unramified exceptional characters, corresponding to the two genuine unramified distinguished characters:
setting $y_i = 2\al_i^\vee$, these characters are determined by
\[
\overline{\chi}^0_{\pm}(y_i(a),\zeta) =\begin{cases}\zeta\qquad\:\:\;: \:\:\;\; i<r\\ \pm\zeta(a,a)_2: \; i=r \end{cases}.
\]
Here $(\cdot,\cdot)_2$ is the quadratic Hilbert symbol. 

 The importance of this choice is discussed in \cite{GG}, where it is shown that the choice of such a character is equivalent to choosing a distinguished splitting of the $L$-group short exact sequence
\[
1\lra \overline{\Sp}(W_n)^\vee\lra {}^L\overline{\Sp}(W_n)\lra WD_F\lra 1,
\]
where $WD_F=W_F\times \SL_2(\cc)$ denotes the Weil-Deligne group of $F$. See \cite{W2} and \cite{W3}.
We fix a such choice, $\overline{\chi}^0$, and note that the restriction to either subtorus $Z(\overline{T}_m)$ or $Z(\overline{T}_n)$ is also a distinguished character for the smaller rank group. In particular, if we choose $\overline{\chi}^0_\pm$ on $Z(\overline{T}_{m+n})$, then the restriction to $Z(\overline{T}_{\ast})$ for $\ast\in\{m,n\}$ is also the unitary distinguished character $\overline{\chi}^0_\pm$. 
Therefore, our choice of unramified genuine distinguished character in the construction of $\Theta_{2(m+n)}$ uniquely determines an unramified distinguished character for both smaller symplectic groups.  Changing the choice multiplies the Satake parameters by the nontrivial element of $Z(\overline{\G}^\vee)$.

 Without loss, assume $m\leq n$. Fix a polarization $W_n = X\oplus Y$ and choose a basis of $X$, $\{e_1\}_{i=1}^{m+n}$. Let $\{f_i\}_{i=1}^{m+n}$ be the dual basis:
\[
\la e_i,f_j\ra = \delta_{i, m+n-j+1}.
\]
We do this so that by restriction, we have induced  polarizations $$W_0= \sspan_F\{e_i : n< i\}\oplus \sspan_F\{f_j : j\leq n\}$$ and $$V_n = \sspan_F\{e_i : i\leq n\} \oplus\sspan_F\{ f_j: n< j\}.$$
Acting by the appropriate Weyl group element $\hat{w}$ (which respects the polarization of $W_n$), we have
\[
\hat{w}(W_0) =  \sspan_F\{e_i : 2\leq i \leq 2n \mbox{  even}\}\oplus Y(\hat{w}(W))
\]
and
\[
\hat{w}(V_n) = \sspan_F\{e_i : 1\leq i < 2n+1 \mbox{  odd or } i> 2n\}\oplus Y(\hat{w}(V_n)),
\]
where $Y(V)\subset Y$ indicates the appropriate dual to the subspace $V\cap X$. In this section, we use the induced embedding of $\overline{\Sp}(W_0)_F\times_{\mu_4}\overline{\Sp}(V_n)_F$ into $\overline{\Sp}(W_n)_F$. 

The main reason for this choice is to facilitate easier comparison with our results on the Jacquet modules of $\Theta_{2(m+n)}$.

Having fixed a basis of $W_0$, we now identify $\Sp(W_0)_F = \Sp_{2m}(F)$ and recall the parametrization of our maximal torus $T_m$
\[
M(t_1,\ldots , t_m) = \diag(t_1,\ldots,t_m,t_m^{-1},\ldots, t_1^{-1}).
\]
With this parametrization, it is easy to parametrize genuine characters of $Z(\overline{T_m})$: $\overline{\chi}=\overline{\chi}^0\chi$, with 
\[
\chi(M(t_1,\ldots , t_m) ,\ep)=\prod_i\chi_i(t_i),
\]
where $\chi_i:F^\times\to\cc^\times$ are multiplicative characters. Since we are discussing characters of  $Z(\overline{T_m})$, it is implicit in the above formulas that $t_i\in (F^\times)^2$.

\begin{Thm}\label{Thm: local correspondence}
 Suppose $\chi : Z(\overline{T_m})\to \cc^\times$ and $\mu: Z(\overline{T_n})\to \cc^\times$ are characters such that
\[
\Hom_{\overline{\Sp}_{2m}(F)\times_{\mu_4}\overline{\Sp}_{2n}(F)}(\Theta_{2(m+n)}, \Ind(\chi)\otimes\Ind(\mu))\neq 0.
\]
Set $k=n-m$. Then there exist $k$ indices $\{j_1,j_2,\ldots,j_k\}\subset \{1,\ldots, n\}$ such that
\[
\mu_{j_i} = |\cdot|^{-(2(k-i)+1)/4}.
\]
Furthermore, we have 
\[
\left\{ \chi_i^{\pm1}\right\}_{1\leq i\leq m}=\left\{ \mu_j^{\pm1}\right\}_{j\notin \{j_1,j_2,\ldots,j_k\}}\\
\]

\end{Thm}
\begin{Rem}
In the above notation, we have made explicit use of our choice of distinguished character $\overline{\chi}^0$. Therefore, the characters $\left\{ \chi_i^{\pm1}\right\}$ encode the Satake parameters relative to this choice in the case $\chi$ is unramified, and likewise with $\{\mu_j^{\pm1}\}$.
\end{Rem}
\begin{proof}
We argue by induction on $m$. The case of $\SL_2\times\SL_2\hra \Sp_4$ is trivial to work out and we leave the details to the reader. The argument is of the same form as the general case.

By transitivity of induction, we identify
\[
\Ind(\mu) \cong \Ind_{\overline{P}_{(1,n-1)}}^{\overline{\Sp}_{2n}(F)}\left(\mu_1\delta_{(1,n-1)}^{1/2}\otimes\Ind(\mu')\right),
\]
where $P_{(1,n-1)} = M_{(1,n-1)}U_{(1,n-1)}$ is the maximal parabolic of $\Sp_{2n}(F)$ with Levi subgroup $M_{(1,n-1)}\cong \GL_1\times \Sp_{2n-2}$, and $\mu':Z(\overline{T_{n-1}})\to \cc^\times$ is the character induced by deleting $\mu_1$. Additionally, $\delta_{(1,n-1)}$ is the modular character of $P_{(1,n-1)}$. 

Therefore by Frobenius reciprocity,
\begin{align}\label{eqn: first frobenius}
 \Hom_{\overline{\Sp}_{2m}(F)\times_{\mu_4}\overline{M}_{(1,n-1)}}\left(J_{U_{1,n-1}}(\Theta_{2(m+n)}), \mu_1\delta_{(1,n-1)}^{1/2}\otimes\Ind(\chi)\otimes\Ind(\mu')\right)\neq0
\end{align}
At this point, we note the following proposition:

\begin{Prop}\label{Prop: filtration}
There exists a ${\overline{\Sp}_{2m}(F)\times_{\mu_4}\overline{M}_{(1,n-1)}}$-equivariant filtration 
\[
0\subset V(\Theta,\psi_1)\subset J_{U_{1,n-1}}(\Theta_{2(m+n)})
\]
such that 
\begin{align}\label{eqn: Mackey isom}
V(\Theta,\psi_1)\cong J_{U_{1,n-1}}\left(\ind_{\overline{Q}_{2,n+m-2}'}^{\overline{Q}_{1,n+m-1}}\left(J_{U_{1,n+m-1},\psi_1}(\Theta_{2(m+n)})\right)\right),
\end{align}
and 
\[
 J_{U_{1,n-1}}(\Theta_{2(m+n)})/V(\Theta,\psi_1) \cong J_{U_{1,n+m-1}}(\Theta_{2(m+n)}).\qedhere
\]
\end{Prop} 

For sake of continuity, the proof of this proposition is postponed to Appendix \ref{Appendix A}. Thus, we have that
\begin{align*}
0&< \dim_\cc\left(\Hom\left(J_{U_{1,n-1}}(\Theta_{2(m+n)}),  \mu_1\delta_{(1,n-1)}^{1/2}\otimes\Ind(\chi)\otimes\Ind(\mu')\right)\right)\\
&\leq \dim_\cc\left(\Hom\left(V(\Theta,\psi_1),  \mu_1\delta_{(1,n-1)}^{1/2}\otimes\Ind(\chi)\otimes\Ind(\mu')\right)\right)\\
&\qquad\quad+\dim_\cc\left(\Hom\left(J_{U_{1,n+m-1}}(\Theta_{2(m+n)}),  \mu_1\delta_{(1,n-1)}^{1/2}\otimes\Ind(\chi)\otimes\Ind(\mu')\right)\right).
\end{align*}
Therefore, we have two cases to consider:\\

{\bf Case 1:}
Suppose that
\begin{align}\label{eqn: Hom space}
\Hom\left(V(\Theta,\psi_1), \mu_1\delta_{(1,n-1)}^{1/2}\otimes\Ind(\chi)\otimes\Ind(\mu')\right)\neq 0.
\end{align}
According to the isomorphism (\ref{eqn: Mackey isom}), we may apply \cite[Theorem 5.2]{BZ} and, setting $$H_\gamma = \left(\Sp_{2m}(F)\times P_{(1,n-1)}\right) \cap \gamma^{-1}Q_{2,n+m-2}'\gamma,$$ we see that this $\Hom$-space is glued together from the vector spaces
\[
\Hom_{H_\gamma}\left( {}^\gamma\left(J_{U_{1,n+m-1},\psi_1}(\Theta_{2(m+n)})\right), \delta^{-1}_{(1,n-1)}\delta_{H_\gamma}\otimes \mu_1\delta_{(1,n-1)}^{1/2}\otimes\Ind(\chi)\otimes\Ind(\mu')\right),
\]
as $\gamma$ ranges over double coset representatives 
\[
\gamma\in Q_{2,n+m-2}'\backslash \Sp_{2(n+m)}(F)/\Sp_{2m}(F)\times P_{(1,n-1)},
\]
where we are viewing the group in the right quotient as a subgroup of $\Sp_{2(n+m)}(F)$  It is an enjoyable exercise (of a similar spirit to the proof of \cite[Prop. 3.4]{K}) to see that there are only three double cosets. By applying \cite[Theorem 5.2]{BZ}, it is immediate that the only coset which contributes to (\ref{eqn: Hom space}) corresponds to $\gamma =1$.
 We thus obtain a non-zero $H = H_1 = (\Sp_{2m}(F)\times P_{(1,n-1)})\cap Q_{1,n+m-1}'$- equivariant morphism
\[
J_{U_{1,n+m-1},\psi_1}(\Theta_{2(m+n)})\longrightarrow  \delta^{-1}_{(1,n-1)}\delta_{H}\otimes \mu_1\delta_{(1,n-1)}^{1/2}\otimes\Ind(\chi)\otimes\Ind(\mu').
\]
We have the following proposition:
\begin{Prop}\label{Prop: local 411}
Suppose $r>1$ and consider the local theta representation $\Theta_{2r}$ on $\overline{\Sp}_{2r}(F)$. Consider the unipotent radical $U_{(2,r-2)}$ of the maximal parabolic subgroup with Levi subgroup $M_{(2,r-2)}\cong\GL_2\times\Sp_{2r-4}$. Then there is a surjection of $\overline{L}'_{2,r-2}(F)$-modules 
\[
J_{U_{(2,r-2)}}(\Theta_{2r}) \twoheadrightarrow J_{U_{1,r-1},\psi_1}(\Theta_{2r}),
\]
where $L'_{k,r-k} =\GL_1^\Delta\times\Sp_{2(r-k)}\subset L_{k,r-k}$ for all $1\leq k\leq r$.
\end{Prop}
\begin{proof}
Our assumptions on the field $F$ ensure that Theorem \ref{Thm: supercuspidal} holds. If we let $V$ denote the subgroup  $V = U_{1,r-1}U_{\mu_2}$, where $U_{\mu_2}$ is the one dimensional root subgroup associated to the long positive root
\[
\mu_2=2\al_2+\cdots +2\al_{r-1}+\al_r,
\]
then there is a $\overline{M}'_{(2,r-2)}(F)$-isomorphism
\[
J_{U_{1,r-1},\psi_1}(\Theta_{2r})\cong J_{V,\psi_1}(\Theta_{2r}),
\]
where by a slight abuse of notation, we denote by $\psi_1$ the character of $V$ which trivially extends $\psi_1$. This is equivalent to the statement that the nilpotent orbit $\calo=(41^{2r-4})$ does not support $\Theta_{2r}$.

The proof can proceed in similar fashion to the global argument in Lemma \ref{Lem: lemma 1} and Lemma \ref{Lem: lemma 2}. It is clear that the result follows if we show that the unipotent subgroup $U_{1,r-2}\subset \Sp_{2(r-1)}(F)\subset \Sp_{2r}(F)$ acts trivially on $ J_{V,\psi_1}(\Theta_{2r})$.  If not, then there is some nontrivial character of $\psi:  U_{1,r-2}/U_{\mu_2}\to \cc^\times$ such that 
\[
J_{U_{1,r-2}/U_{\mu_2}, \psi'}(J_{V,\psi_1}(\Theta_{2r}))\cong J_{V_{2,r-2},\psi_2}(\Theta_{2r}) \neq0.
\]
We have made use of the fact that $\Sp_{2(r-1)}(F)\subset \Sp_{2r}(F)$ acts transitively on such characters, so that we obtain the twisted Jacquet module with respect to the Gelfand-Graev character 
\[
\psi_2: V_{2,r-2}\to \cc^\times, \qquad\qquad v\mapsto \psi(v_{1,2}+v_{2,3}).
\]
Again applying Theorem \ref{Thm: supercuspidal}, we have an isomorphism of $\overline{L}'_{3,r-3}(F)$-modules
\[
J_{V_{2,r-2},\psi_2}(\Theta_{2r})\cong J_{V_{2,r-2}U_{\mu_3},\psi_2}(\Theta_{2r}),
\]
where $\mu_3=2\al_3\cdots +2\al_{r-1}+\al_r$ is the next long root. This is equivalent to the statement that the nilpotent orbit $\calo= (61^{2r-6})$ does not support $\Theta_{2r}$.

We know, as in the proof of Lemma \ref{Lem: lemma 1}, that $U_{1,r-3}/U_{\mu_3}$ cannot act trivially on this Jacquet module, for if it did then we could apply Theorem \ref{Thm: local constant term} to obtain a non-zero Whittaker module on $\Theta_{GL_3}^{(2)}$, a contradiction as this representation is non-generic. Hence, as before we conclude the non-vanishing of $J_{V_{3,r-3},\psi_3}(\Theta_{2r}).$

Continuing in this way, we finally arrive at the non-vanishing of $J_{U_r, \psi_r}(\Theta_{2r})$, where we recall that $U_r$ denotes the full unipotent radical of the Borel subgroup of $\Sp_{2r}(F)$ and $\psi_r$ is a generic Gelfand-Graev character of $U_r$. This contradicts the fact that $\Theta_{2r}$ is non-generic, completing the proof.
\end{proof}
It follows from the Proposition along with Theorem \ref{Thm: local constant term} that we in fact have a non-zero $H\cap\overline{L}'_{2,n+m-2}(F)U_{1,n+m-2}$- equivariant morphism
\[
J_{N_1,\psi'}(\Theta_{\GL_2}^{(2)})\otimes\Theta_{2(m+n-2)}\longrightarrow\delta^{-1}_{(1,n-1)}\delta_{H}\otimes \mu_1\delta_{(1,n-1)}^{1/2}\otimes\Ind(\chi)\otimes\Ind(\mu').
\]

Here $N_1\subset \GL_2(F)$ is the subgroup of upper triangular unipotent matrices, and 
\[
\psi'\left(\begin{array}{cc}	
					1&x\\
					&1
					\end{array}\right) = \psi(x).
\]
To simplify the notation, we collect all the characters on the right hand side as 
\[
\lam=\delta^{-1}_{(1,n-1)}\delta_{H}\otimes \mu_1\delta_{(1,n-1)}^{1/2}
\]
Note that 
\begin{align*}
H\cap\overline{L}'_{2,n+m-2}(F)U_{1,n+m-2} &= \Sp_{2(n-1)}\times(\GL_1^\Delta\times\Sp_{2(m-1)})U_{1,m-1}\\
											&\cong \Sp_{2(n-1)}\times Q_{1,m-1},
\end{align*} 
By the properties of induction, this implies that 
\begin{align}\label{eqn: Hom space 2}
\Hom\left(\Ind_{\overline{\Sp}_{2(n-1)}\times \overline{Q}_{1,m-1}}^{\overline{\Sp}_{2(n-1)}\times \overline{\Sp}_{2m}}\left(J_{N_1,\psi'}(\Theta_{\GL_2}^{(2)})\otimes\Theta_{2(m+n-2)}\right),\lam\otimes\Ind(\chi)\otimes\Ind(\mu')\right)\neq0.
\end{align}

%%%%%%%%%%%%%%%%%%%%%%%%%%%%%%%%%%%%%%%%%%%%%%%%%%%%%%%%%%%%%%%%%%
We make use of a result of Bernstein (see \cite{B} for the proof): 
\begin{Prop}\label{Prop: Bernstein}
If $P=MU$ is a parabolic subgroup of $G$. Let $P_-=MU_-$ denote the opposite parabolic of $P$. If $W$ is a smooth representation of the Levi subgroup $M$ and $V$ is a smooth representation of $G$, there there is an isomorphism
\[
\Hom_G(\Ind_P^G(W),V) \cong \Hom_M(W,J_{U_-}(V)).
\]
\end{Prop}\qed

Applying this proposition to (\ref{eqn: Hom space 2}), we find that
\begin{align*}%\label{eqn: Hom space 3}
\Hom\left(J_{N_1,\psi'}(\Theta_{\GL_2}^{(2)})\otimes\Theta_{2(m+n-2)},\lam\otimes J_{\overline{U}_{1,m-1}}\left(\Ind(\chi)\right)\otimes\Ind(\mu')\right)\neq0
\end{align*}

Applying the geometric lemma of \cite{BZ} to this Jacquet module, we have that there is a $\overline{M}_{(1,m-1)}$-equivariant filtration of $J_{\overline{U}_{1,m-1}}\left(\Ind(\chi)\right)$ such that the corresponding factors are representations of the form
\[
\chi_i^{\pm1}\delta_{(1,m-1)}^{-1/2}\otimes\Ind(\chi'),
\]
with $k\in\{2,3,\ldots, m\},$ and where $\chi':Z(\overline{T_{m-1}})\to \cc^\times$ is a character such that
\[
\chi'_j\in \{ \chi_k^{\pm1} : k\in \{2,3,\ldots, m\}, k\neq i\}.
\]
Therefore, there exists some $k$ such that our morphism induces a non-zero morphism (without loss we can assume $\chi_k^{-1}$ appears):
\[
J_{N_1,\psi'}(\Theta_{\GL_2}^{(2)})\otimes\Theta_{2(m+n-2)}\longrightarrow\lam\otimes \chi_k^{-1}\delta_{(1,m-1)}^{-1/2}\otimes\Ind(\chi')\otimes\Ind(\mu').
\]
We now conclude: the diagonal $\GL_1$ acts on the right hand side by the accumulated character
\[
\lam \chi_k^{-1}\delta_{(1,m-1)}^{-1/2}=\mu_1\chi_k^{-1}\delta^{-1}_{(1,n-1)}\delta_{H}\delta_{(1,n-1)}^{1/2}\delta_{(1,m-1)}^{-1/2}=\mu_1\chi_k^{-1}|\cdot|^{m+n}
\]
where the modular characters are easy to compute, and we note $\delta_H = \delta_{(1,n-1)}\delta_{(1,m-1)}$. On the other hand, it acts on the left hand side by the appropriate exceptional character on $\GL_2$ as indicated in Theorem \ref{Thm: local constant term}:
\[
\chi_{\GL,2}\left(\begin{array}{cc}a&\\&a\end{array}\right) = |a|^{m+n}.
\]
Therefore $\mu_1 = \chi_k$, and we are left with 
\[
\Hom_{\overline{\Sp}_{2(m-1)}(F)\times_{\mu_4}\overline{\Sp}_{2(n-1)}(F)}\left(\Theta_{2(m+n-2)},\Ind(\chi')\otimes\Ind(\mu')\right)\neq 0.
\]
By induction, this case is proved. We note that in this case the difference in the rank of the smaller symplectic groups $k$ remains unchanged.\\

{\bf Case 2:}
Suppose instead that 
\[
\Hom\left(J_{U_{1,n+m-1}}(\Theta_{2(m+n)}),  \mu_1\delta_{(1,n-1)}^{1/2}\otimes\Ind(\chi)\otimes\Ind(\mu')\right)\neq0.
\]
 Applying Theorem \ref{Thm: local constant term}, we have a non-zero ${\overline{\Sp}_{2m}(F)\times_{\mu_4}\overline{M}_{(1,n-1)}}$-equivariant map
\[
\chi_{\GL,1}\otimes\Theta_{2(m+n-1)} \longrightarrow \mu_1\delta_{(1,n-1)}^{1/2}\otimes\Ind(\chi)\otimes\Ind(\mu').
\]
Here, $\chi_{\GL,1}$ is the character on $\overline{\GL_1}(F)$ arising in Theorem \ref{Thm: local constant term}. Comparing the action of this $\GL_1$-component of $ \overline{M}_{(1,n-1)}$, we find that
\[
\chi_{\GL,1} = \mu_1\delta^{1/2}_{(1,n-1)}
\]
so that $\mu_1 = |\cdot|^{-(2k-1)/4}$, where recall that $k=n-m$. We therefore obtain 
\[
\Hom_{\overline{\Sp}_{2m}(F)\times_{\mu_4}\overline{\Sp}_{2(n-1)}(F)}(\Theta_{2(m+n-1)}, \Ind(\chi)\otimes\Ind(\mu'))\neq 0.
\]
Since our analysis has utilized the choice of embedding $$\Sp_{2m}\times\Sp_{2(n-1)}\hra \Sp_{2(m+n-1)},$$ continuing the same argument would lead us to now apply transitivity of induction to $\Ind(\chi)$ and study the Jacquet module $J_{V_{1,m-1}}(\Theta_{2(m+n-1)})$. As above, this Jacquet module has a filtration and there are two cases to consider. If Case 1 occurs, then the analysis is the same and we reduce to 
\[
\Hom_{\overline{\Sp}_{2(m-1)}(F)\times_{\mu_4}\overline{\Sp}_{2(n-2)}(F)}(\Theta_{2(m+n-3)}, \Ind(\chi'')\otimes\Ind(\mu''))\neq 0.
\]
Repeat this ($i$ times, say) until Case 2 occurs. Thus, we have a non-zero map
\[
\overline{\chi}_{1}\otimes\Theta_{2(m+n-2i-1)} \longrightarrow \chi_1^{(i)}\delta_{(1,m-i-1)}^{1/2}\otimes\Ind(\chi^{(i)})\otimes\Ind(\mu^{(i+1)})),
\]
where $\chi^{(i)}:Z(\overline{T}_{m-i})\to\cc^\times$ and $\mu^{(i+1)}:Z(\overline{T}_{n-i-1})\to\cc^\times$ are the characters arising from the argument in Case 1 applied $i$ times. An identical computation to the one above now shows that 
\[
 \chi_{1}^{(i)}=|\cdot|^{-(2(1-k)-1)/4} = \mu_1^{-1}.
\]
This shows that if Case 2 occurs when analyzing $\Ind(\chi)$, we still obtain $\chi_1\in \{\mu_j^{\pm 1}\}$. This is the only case not immediate from induction. Applying the case-by-case argument above, the theorem is proven.\qedhere

\end{proof}
In Section \ref{Section: CAP}, we will interpret this result in terms of Arthur's conjectures, subject to a natural generalization of Shimura's correspondence.

\section{Global Lifting}\label{Section: Global Lift}
Suppose now that $K$ is a number field such that $|\mu_4(K)|=4$. As in the introduction, let $W_0$ be a $2m$-dimensional symplectic space over $K$, and let
\[
W_n=X_n\oplus W_0\oplus X_n^\ast,\:\:\: V_n=X_n\oplus X_n^\ast 
\]
 be symplectic spaces of dimension $2(m+n)$ (resp. $2n$). Then there is an inclusion of symplectic groups $$\Sp(W_0)_\A\times\Sp(V_n)_\A\hra \Sp(W_n)_\A.$$ Here $\A =\A_K$ is the adele ring associated to $K$. 
Then the above embedding is covered by an embedding of metaplectic groups 
\[
\overline{\Sp}(W_0)_\A\times_{\mu_4}\overline{\Sp}(V_n)_\A\hra \overline{\Sp}(W_n)_\A
\]

Let $\pi$ be an $\ep$-genuine cuspidal automorphic representation of $\overline{\Sp}(W_0)_\A$, and let $\Theta_{W_n} := \Theta_{2(m+n)}$ be the theta representation of $\overline{\Sp}(W_n)_\A$. Consider the integral pairing
\begin{align}\label{eqn: pairing}
\displaystyle \int_{[\Sp(W_0)]}\overline{\varphi(g)}\theta_{n}((h,g))\,dg,
\end{align}
where $\varphi\in \pi$ and $\theta_{n}\in \Theta_{W_n}$. The vector space generated by such integrals as we vary over all vectors in $\pi$ and $\Theta_{W_n}$ forms an $\ep$-genuine automorphic representation of $\overline{\Sp}(V_n)_\A$, which we denote by $\Theta_n(\pi)$. Note that it is immediate from the spectral decomposition that if $\pi$ is cuspidal then, $\Theta_0(\pi)=0$, and for any $\pi$ we set $\Theta_{-1}(\pi)=0$.

%%%%%%%%%%%%%%%%%%%%%%%%%%%%%%%%%%%%%%%%%%%%%%%%%%%%%%%%%%%%%%%%%%%%%%%%%%%%%%%%%%%%%%%%%%%%%%%%%%%%%%%%%%%%%%%%%%%%%%%%%%%%%%%%%%%%%%%%%%%%%%%%%%%%%%%%%%%%%%%%%%%%%%%%%%%%%%%%%%%%
%%%%%%%%%%%%%%%%%%%%%%%%%%%%%%%%%%%%%%%%%%%%%%%%%%%%%%%%%%%%%%%%%%%%%%%%%%%%%%%%%%%%%%%%%%%%%%%%%%%%%%%%%%%%%%%%%%%%%%%%%%%%%%%%%%%%%%%%%%%%%%%%%%%%%%%%%%%%%%%%%%%%%%%%%%%%%%%%%%%%%%%%%%%%%%%%%%%%%%%%%%%%%%%%%%
We begin our global study by noting the following theorem:
\begin{Thm}\label{Thm: Non-vanishing}
Let $\pi$ be as above. Then $\Theta_{4m}(\pi)\neq 0$, so that the $\pi$ lifts nontrivially to $\overline{\Sp}(V_{4m})_\A \cong \overline{\Sp}_{8m}(\A)$.
\end{Thm}
We sketch the proof as this result is analogous to Theorem 2 in \cite{BFG2}, and is proved in a similar manner. Supposing toward a contradiction that the integral
\begin{equation*}
 \int_{[\Sp(W_0)]}\overline{\varphi(g)}\theta_{4m}((h,g))\,dg 
\end{equation*}
vanishes for all data, it is clear that the integral
\begin{equation}\label{eqn: assumption}
 \int_{[\Sp(W_0)]}\int_{[V_2(\calo)]}\overline{\varphi(g)}\theta_{4m}((v,g))\psi_\calo(v)\,dv\,dg, 
\end{equation}
where $\calo = (2^{4m}1^{2m})$ and  $$\psi_\calo(v) =\psi(v_{4m+1,6m+1} + \cdots v_{1,10m-1})$$ must also vanish for all data. Note that $V_1(\calo)$ has the structure of a generalized Heisenberg group with center equal to $V_2(\calo)$, so that there is a homomorphism $\lam : V_1(\calo)_\A\to \mathcal{H}_{8m^2+1}(\A)$ to the Heisenberg group on $8m^2+1$ variables. 

Denote by $\omega_\psi$ the global oscillator representation of the dual pair $(\SO_{4m},\Mp_{2m})\subset\Mp_{8m^2}$ and let $\tilde{\theta}_{\phi}$ be the theta function associated to $\phi\in \omega_\psi$. Results of Ikeda on Fourier-Jacobi coefficients \cite{I} tell us that functions of the form
\begin{equation}\label{eqn: FJ coeff}
\tilde{\theta}_{\phi_1}(1,g)\int_{[V_1(\calo)]}\theta_{4m}\left(v(1,g)\right)\tilde{\theta}_{\phi_2}\left(\lam(v)(1,g)\right)\,dv
\end{equation}
are dense in the space of functions of the form
\begin{equation}\label{eqn: coeff}
\int_{[V_2(\calo)]}\theta_{4m}(v,g)\psi_{\calo}(v)\,dv.
\end{equation}
Thus, the vanishing (\ref{eqn: assumption}) tells us that if we set 
\[
H(g) = \overline{\varphi}(g)\int_{[V_1(\calo)]}\theta_{4m}\left(v(1,g)\right)\tilde{\theta}_{\phi_2}\left(\lam(v)(1,g)\right)\,dv, 
\]
then we have 
\begin{equation*}
\int_{[\Sp(W_0)]}H(g)\tilde{\theta}_{\phi}(1,g)\,dg = 0
\end{equation*}
for all $\phi\in \omega_\psi$. As in \cite{BFG2}, since we are considering a high enough lift, the arguments used in the proof of \cite[Theorem I.2.2]{R} allow us to conclude that $H(g) \equiv 0$. This forces all functions of the form (\ref{eqn: FJ coeff}) to vanish, and hence the density results force all functions of the form (\ref{eqn: coeff}) to also vanish. However, it is a straightforward check to show that the non-vanishing of Fourier coefficients associated to $\calo_{\Theta, 5m} = (2^{5m})$ in Theorem \ref{Thm: non-vanishing} forces these Fourier coefficients associated to the orbit $\calo= (2^{4m}1^m)$ to be non-vanishing for some choice of data. We therefore obtain a contradiction.

\section{Towering property of the theta lift}\label{Section: cuspidality}
In this section we prove the rest of Theorem \ref{Thm: intro cuspidality}.

\begin{Thm}\label{Thm: cuspidality}
Suppose that $\pi$ is a cuspidal automorphic representation of $\overline{\Sp}(W_0)_\A$. Then
\begin{enumerate}
\item\label{1} If $\Theta_k(\pi)=0$, then $\Theta_{k-1}(\pi)=0$;
\item\label{2} If $\Theta_{k-1}(\pi)=0$, then $\Theta_k(\pi)$ is cuspidal.
\end{enumerate}
\end{Thm}

\begin{proof}
We first prove the persistence property (\ref{1}). By assumption we have that for all choices of data, the integral (\ref{eqn: pairing}) vanishes. Thus, if $V_{1,n-1}\subset \Sp(V_n)$ is the unipotent radical of the maximal parabolic preserving an isotropic line, then
\begin{align}\label{eqn: preim}%referenced
\int_{[\Sp(W)]}\int_{[V_{1,n-1}]}\overline{\varphi(g)}\theta_{n}((u,g))du\,dg = 0.
\end{align}
We have that $V_{1,n-1}\subset V_{1,n+m-1}$ with an abelian complimentary subgroup $\left(V_{1,n+m-1}/V_{1,n-1}\right) \cong \mathbb{G}_a^{2m}$. Acting via conjugation, we identify this vector space with the standard representation of $\Sp(W)$. Expanding the inner integral of (\ref{eqn: preim}) against this subgroup, we have that $\Sp(W)_K\cong \Sp_{2m}(K)$ acts on the character group with two orbits. Applying Theorem \ref{Thm: global constant term} with $a=1$, the trivial orbit gives the integral
\[
\int_{[\Sp_{2n}]}\overline{\varphi(g)}\theta_{n-1}((1,g))\,dg,
\]
which, as we range over all data, gives an arbitrary element in $\Theta_{n-1}(\pi)$. Thus, the proposition follows if we show that the contribution of the non-trivial orbit in the above expansion vanishes.

Let $Q_{1,m-1}$ be the maximal parabolic subgroup of $\Sp_{2m}$ preserving a line. We may choose a representative of the nontrivial $\Sp_{2m}(F)$-orbit to be the character $\tilde{\psi}(u) = \psi(u_{1,m+1})$, the stabilizer of which is $Q'_{1,m-1}\subset Q_{1,m-1}$, the subgroup obtained by omitting the $\GL_1(F)$ piece of the Levi. Unfolding, we find that the non-trivial orbit contributes
\[
	\int_{Q'_{1,m-1}(K)\backslash\Sp_{2m}(\A)}\int_{[V_{1,n+m-1}]}\overline{\varphi(g)}\theta_{n}(u(1,g))\tilde{\psi}(u)\,du\,dg.
\]

Consider the Weyl group element $w\in W$ preserving $U_1$ such that $w\cdot \tilde{\psi} = \psi_1^0$. 
Using automorphy of $\theta_{n}$,  we may conjugate the inner integration by $w$ and change variables so that the above integral equals
\begin{align}\label{eqn: P^0}%referenced
\int_{Q'_{1,m-1}(K)\backslash\Sp_{2m}(\A)}\overline{\varphi(g)}\theta^{V_{1,n+m-1},\psi^0_{1}}(w(1,g))\,dg,
\end{align}
 where 
\[
\theta^{V_{k,m-k},\psi^0_{k}}(g) = \int_{[V_{k,m-k}]}\theta(vg)\psi^0_{k}(v)\,dv.
\]

Recall that $V_{1,m-1}$ is the unipotent radical of $Q_{1,m-1}$. We may decompose the outer integral so that we obtain
\[
\int_{[V_{1,m-1}]}\overline{\varphi(vg)}\theta^{V_{1,m+n-1},\psi^0_{1}}(w(1,vg))\,dv
\]
as an inner integration of (\ref{eqn: P^0}). Identifying $V_{1,m-1}$ with its image in $\Sp_{2(m+n)}$, we see that $w V_{1,m-1}w^{-1} \subset V_{2,m+n-2}$, so that by Lemma \ref{Lem: lemma 2} we may change variables in $\theta^{V_{1},\psi^0_{1}}$ to remove the dependence on $v$, implying that the integral equals
\[
c\int_{[V_{1,m-1}]}\overline{\varphi(vg)}\,dv,
\]
for some constant $c\in \cc$. By the cuspidality of $\pi$, this is trivial.
Therefore, the non-trivial orbit does not contribute. This completes the proof of part (\ref{1}).

We now turn to proving part (\ref{2}). For the sake of this proof, let $R_l $ denote the image of the unipotent radical $U_{(l,n-l)}$ of the maximal parabolic subgroups of $\Sp_{2n}$. We need to show that the vanishing of $\Theta_{n-1}(\tau)$ implies the vanishing of the integral
\begin{align}\label{eqn: constant term}%referenced
\int_{[\Sp_{2m}]}\int_{[R_l]}\overline{\varphi(g)}\theta_{n}((v,g))\,dv\,dg. 
\end{align}

We begin by noting that the root subgroup associated to the highest root $\al_0=\mu_1$ is a subgroup of the images of $R_l$ for all $l$. Thus, by identifying this root subgroup with $Z(H_1)$, we may expand the inner integration against $[H_1/Z(H_1)]$. 

As before, $\Sp_{2(n+m-1)}(K)$ acts on the characters of this group with two orbits. By Theorem \ref{Thm: global constant term}, the trivial orbit gives an element of the lifting $\Theta_{m-1}(\tau)$, and so the corresponding contribution in the Fourier expansion vanishes. Left with the non-trivial orbit, we unfold to see that (\ref{eqn: constant term}) is equal to
\begin{align}\label{eqn: step 1}%referenced
\sum_{\ga\in Q_{1,n+m-2}(K)\backslash\Sp_{2(n+m-1)}(K)} \int_{[\Sp_{2m}]}\int_{[R_l]}\overline{\varphi(g)}\vartheta^{V_1,\psi_1^0}_{n}(\ga(v,g))\,dv\,dg,
\end{align}
where $Q_{1,n+m-2}$ is the maximal parabolic subgroup of $\Sp_{2(n+m-1)}(K)$ preserving a line, and $V_i= V_{i,n+m-l}$ is the simplifying notation we use in this argument.

Note that we have replaced $\overline{\theta}_{n}^{V_1,\psi_1^0}(g)$ with
\[
\vartheta^{V_1,\psi_1^0}_{n}(g) = \sum_{\xi\in K^\times}\theta_{n}^{V_1,\psi_1^0}(h_1(\xi)g),
\]
where
\[
h_1(\xi) = \diag(1,\xi,1,\ldots,1,\xi^{-1},1).
\]

 We do this to simplify the Mackey theory we need to move forward. We may further decompose the above sum by considering the double cosets
\[
Q_{1,n+m-2}(K)\backslash\Sp_{2(n+m-1)}(K)/Q_{1,n+m-2}(K)
\]
which has 3 elements with the representatives $1, w_2,$ and $w_{0,2}$ where $$w_{0,2} = w_2w_3\cdots w_n\cdots w_3w_2$$ is the ``relative long element'' corresponding to the  parabolic subgroup $Q_{1,n+m-2}$. 

We claim that the only double coset to contribute to (\ref{eqn: step 1}) is the coset represented by $w_2$. In fact, to see that the case of $w_{0,2}$ does not contribute, we need only comment that the short root group $U_{\al}$, where 
\[
\al = \eta_{1,2} = \al_1+2\al_2+\cdots+ 2\al_{n-1}+\al_n
\]
is conjugated by $\tilde{w}_0$ to the root group $U_{\al_1}$. Thus, since this root group is contained in $R_l$ for all $l$, we see that as an inner integration we have
\begin{align*}
&\int_{U_\al(\A)R_l(K)\backslash R_l(\A)}\int_{[U_\al]}\vartheta^{V_1,\psi_1^0}_{n}(\tilde{w}_0x_\al(t)(v,g))\,dt\,\,dv\\
&=\int_{U_\al(\A)R_l(K)\backslash R_l(\A)}\sum_{\xi\in K^\times}\left(\int_{[\A]}\psi(\xi t)\,dt\right)\theta^{V_1,\psi_1^0}_{n}(h(\xi)\tilde{w}_0(v,g))\,\,dv.
\end{align*}
This vanishes since $\xi\neq0$, and thus
\[
\int_{[\A]}\psi(\xi t)\,dt=0.
\]
The contribution of $w=1$ also vanishes. In this case, the contribution equals
\begin{align*}
\sum_{\xi\in K^\times}\int_{[\Sp_{2m}]}\int_{[R_l]}\overline{\varphi(g)}\theta^{V_1,\psi_1^0}_{n}(h(\xi)(v,g))\,\,dv\,dg.
\end{align*}
By Lemma \ref{Lem: lemma 2}, we see that we may replace $\theta_{n}^{V_1,\psi_1^0} = \theta_{n}^{V_2,\psi_2^0}$. Factoring this integral along the unipotent radical $U_{(2,n+m-2)}$, and applying Theorem \ref{Thm: global constant term} with $a=2$, we have an inner integration contained in the lifting $\Theta_{m-2}(\tau)$. By part (\ref{1}) and the assumption that $\Theta_{n-1}(\pi)=0$, this vanishes.

Therefore, the only double coset which contributes to (\ref{eqn: step 1}) is the one represented by $w_2$. With this simplification, we may rewrite (\ref{eqn: step 1}) as
\begin{align*}%\label{eqn: step 2}
\sum_{\ga\in Q_{1,n+m-3}\backslash \Sp_{2(n+m-2)}(K)}\sum_{\ep\in K} \int_{[\Sp_{2m}]}\int_{[R_l]}\overline{\varphi(g)}\vartheta^{V_1,\psi_1^0}_{n}(w_2\ga x_{\al_2}(\ep)(v,g))\,dv\,dg.
\end{align*}
We now consider the set of double cosets $Q_{1,n+m-3}\backslash \Sp_{2(n+m-2)}(K)/Q_{1,n+m-3}$, which again has only 3 elements, represented by $1, w_3,$ and $ w_{0,3}$. The argument is the same as the previous case, and by induction we find that our integral reduces to
\begin{align*}%\label{eqn: unravel}
\int_{Q_{1,m-1}(K)\backslash\Sp_{2m}(\A)}\int_{[R_l]}\overline{\varphi(g)}\sum_{\ep_2,\cdots,\ep_n}\vartheta^{V_1,\psi_1^0}_{n}(w_2 x_{\al_2}(\ep_2)\cdots w_n x_{\al_n}(\ep_n)(v,g))\,dv\,dg.
\end{align*}
Setting $\tilde{w}_2 = w_2w_3\cdots w_n=\diag(w,I_{2m-2},w^\ast)$, where

\[w=\left(\begin{array}{ccc}
										1&&\\
										&&1\\
										&I_{n-1}&
										\end{array}\right),\]
										 and 
$r(\ep_2,\cdots, \ep_m) = \diag(u(\ep),I_{2m-2}, u(\ep)^*)$, where 

\[
u(\ep) = \left(\begin{array}{ccc}
					1&\cdots&0\\
					&&\ep_2\\
					&\ddots&\vdots\\
					&&\ep_n\\
					&&1
					\end{array}\right)\in \GL_{n+1}(F),
\]
we may rewrite this as 
\begin{align*}%\label{eqn: unravelled step 1}
\int_{Q_{1,m-1}(K)\backslash\Sp_{2m}(\A)}\int_{[R_l]}\overline{\varphi(g)}\sum_{\ep_2,\cdots,\ep_n}\vartheta^{V_1,\psi_1^0}_{n}(\tilde{w}_2 r(\ep_1,\ldots,\ep_n)(v,g))\,dv\,dg.
\end{align*}
We may immediately remove several of the above summands as follows: decompose our unipotent radical $R_l = L\cdot R_l^0$ where 
\[
L_1' =\la x_\al(t) : \al \in \{\ga_{1,l},\ldots, \ga_{1,n-1}\}\ra\cong \mathbb{G}_a^{n-l}.
\] 
Decomposing the integration over $R_l$ and conjugating $$\lam = x_{\ga_{1,l}}(t_{l})\cdots x_{\ga_{1,n-1}}(t_{n-1})\in L_\A$$ past $\tilde{w}_2r(\ep_1,\ldots,\ep_n)$, we change variables an obtain an inner integration of the form
\[
\int_{[\A^{n-l+1}]}\psi(\ep_{l+1}t_l+\cdots +\ep_nt_{n-1})\,dt_i,
\]
which vanishes unless $\ep_{l+1}=\cdots=\ep_n=0$. Therefore, we have reduced the integral to the case
\begin{align}\label{eqn: unravel step 2}%referenced
\int_{Q_{1,m-1}(K)\backslash\Sp_{2m}(\A)}\int_{[R^1_l]}\overline{\varphi(g)}\sum_{\ep_2,\cdots,\ep_l}\vartheta^{V_1,\psi_1^0}_{n}(\tilde{w}_2 r(\ep_2,\ldots,\ep_{l},0,\cdots,0)(v,g))\,dv\,dg,
\end{align}
where here $R_l^1$ is the subgroup of elements of $R$ such the projection to the root groups $U_\al$ with $\al\in\{\ga_{1,k},\eta_{1,k}: 1\leq k\leq n+m\}$ are trivial. The reason we have this integral is, as before, decomposing $R_l^0 = L_1''\cdot R_l^1$ where $L'$ is the rest of the top row, we can conjugate $L'$ past $\tilde{w}_2$ and absorb the integration over $L'$ into the integration over $V_1$ in $\vartheta^{V_1,\psi_1^0}_{n}(g)$.

If $l=1$ or $\ep_i=0$ for all $i$, then the above integral vanishes. To see this, note that we may decompose the integration over $Q_{1,m-1}(K)\backslash\Sp_{2m}(\A)$ to obtain as an inner integration over $[L_{1,m-1}]$. As noted in the proof of part (\ref{1}), $\vartheta^{V_1,\psi_1^0}_{n}$ is invariant under $\tilde{w}_2 L_{1,m-1}\tilde{w}_2^{-1}$, so that upon changing variables we find the constant term
\[
\int_{[L_{1,m-1}]}\varphi(lg)\,dl =0
\] 
as an inner integration. In the case $l>1$, we have thus reduced to the case $\ep_i\neq 0$. 

Consider, within the Levi subgroup $M_{(l,n-l)} \cong \GL_l\times\Sp_{2(n-l)}$ corresponding to our unipotent radical, the subgroup $\GL_{l-1}$ embedded as the (lower) mirabolic subgroup of the $\GL_l$ factor. Thus, if $\zeta\in \GL_{l-1}(K)$, then $\zeta$ embeds into $\Sp_{2n}$ as
\[
\underline{\zeta} = \diag(1,\zeta, I_{2(m+n-l)},\zeta^*,1).
\]
This subgroup acts by conjugation on the set $\{r(\ep_1,\ldots,\ep_l,0,\ldots,0) :\ep_i\neq0\}$ with one orbit. One can check that this $\GL_{l-1}$ commutes with embedded $\overline{\Sp}_{2m}(\A)$. As it is a subgroup of the embedded parabolic subgroup $$M_{(l,n-l)}\subset\Sp_{2n}\subset \Sp_{2(m+n)},$$ it clearly normalizes $R_l$. Using these two observations, we may rewrite (\ref{eqn: unravel step 2}) as
\begin{align*}%\label{eqn: step 3}
\sum_{\zeta\in H_1\backslash\GL_{l-1}(K)}\int_{Q_{1,m-1}(K)\backslash\Sp_{2m}(\A)}\int_{[R^1_l]}\overline{\varphi(g)}\vartheta^{V_1,\psi_1^0}_{n}(\tilde{w}_2 r_1(v,g)\underline{\zeta})\,dv\,dg.
\end{align*}
Here, $r_1 := r(1,0,\ldots, 0)= x_{\ga_{2,n}}(1)$ and $H_1$ is the stabilizer in $\GL_{l-1}$ of $r_1$. Note that we have used the fact that $\tilde{w}_2\underline{\zeta}^{-1}\tilde{w}^{-1}_2$ stabilizes the integral $\vartheta^{V_1,\psi_1^0}_{n}(g)$.

Thus, by replacing $\theta_{n}$ with a translate, we see that it suffices to show that
\[
\int_{Q_{1,m-1}^0(K)\backslash\Sp_{2m}(\A)}\int_{[R^1_l]}\overline{\varphi(g)}\theta^{V_1,\psi_1^0}_{n}(\tilde{w}_2 r_1(v,g))\,dv\,dg=0,
\]
where we have unraveled the definition of $\vartheta_n$ and let $Q_{1,m-1}'$ denote the same parabolic with the $\GL_1$-factor omitted.
%%%%%%%%%%%%%%%%%%%%%%%%%%%%%%%%%%%%%%%%%%%%%%%%%%%%%%%%%%%%%%%%%%%%%%%%%%%%%%%%%%%%%%%%%%%%%%%%%%%%%%%%%%%%%%%%%%%%%%%%%%%%%%%%%%%%%%%%%%%%%%%%%%%%%%%%%%%%%%%%%%%%%%%%%%%%%%
By Lemma \ref{Lem: lemma 2}, we may replace $\theta^{V_1,\psi_1^0}_{n} = \theta^{V_2,\psi_2^0}_{2n}$.  

If $l=2$, then we may conclude that the above integral vanishes. To see this, one can show that $$Z(H_3)V_{2,m+n-2}\subset\tilde{w}_2r(1)R_2^1(\tilde{w}_2r(1))^{-1}\subset V_{3,m+n-3}.$$  In fact, after conjugating and changing variables, our integral becomes
\begin{align*}%\label{eqn: l=2}
\int_{Q_{1,m-1}^0(K)\backslash\Sp_{2m}(\A)}\int_{[V_3^1]}\overline{\varphi(g)}\theta_{n}(\tilde{w}_2 r_1(v,g))\psi_3^1(v)\,dv\,dg.
\end{align*}
Here we have $V_3^1\subset V_{3,m+n-3}$  defined by
\[
V_3^1 = \{v\in V_{3,m+n-3} : v_{3,j} \neq 0 \implies j\leq m+1 \mbox{ or } j\geq 2(m+n)-n\},
\]
against the character $\psi_3^1(v) = \psi(v_{1,2})$.

Thus, we may expand along the complimentary subgroup $(V_3\setminus{V_3^1})\cong \mathbb{G}_a^{2m-2} \subset H_3$ (note that the center of $H_3$ is already being integrated against). The subgroup $\Sp_{2m-2}(K)\subset \overline{\Sp}_{2m}(\A)$ acts via conjugation with two orbits on the characters of this subgroup.

The trivial orbit does not contribute since, similarly to the proof of part (\ref{1}), we note that $\tilde{w}_2r_1V_{1,m-1}(\tilde{w}_2r_1)^{-1}\subset V_{3,n+m-3}$, and so conjugating and changing variables, we obtain
\[
\int_{[V_{1,m-1}]}\overline{\varphi(u'g)}du'
\]
as an inner integration. This vanishes by the cuspidality of $\pi$. For the nontrivial orbit, define the character $\psi_3^2: V_{3,m+n-3}\to \cc^\times$ to be 
\[
\psi_3^2(v) = \psi(v_{1,2}+v_{3,m+2}).
\]
Then, we see that our integral becomes
\begin{align*}%\label{eqn: even base case}
\sum_{\ga\in Q^0_{1,m-2}\backslash \Sp_{2m-2}(K)} \int_{Q_{1,m-1}^0(K)\backslash\Sp_{2m}(\A)}\overline{\varphi(g)}\theta^{V_3,\psi_3^2}_{n}(\ga\tilde{w}_2 r_1(1,g))\,dv\,dg\\
= \int_{Q_{2,m-2}^0(K)\backslash\Sp_{2m}(\A)}\overline{\varphi(g)}\theta^{V_3,\psi_3^2}_{n}(\tilde{w}_2 r_1(1,g))\,dv\,dg,
\end{align*}
where we used the fact that $\ga$ commutes with $\tilde{w}_2r_1$ to allow the unfolding.
 Let $\tilde{w}_4 = w_4\cdots w_{n+1}$ be the Weyl group element preserving $V_{3,m+n-3}$ such that $\tilde{w}_4\cdot\psi_3^2 = \psi_3^0$. Conjugating by this element, we see that our integral becomes
\begin{align*}%\label{eqn: even step 2}
\int_{Q_{2,m-2}^0(K)\backslash\Sp_{2m}(\A)}\overline{\varphi(g)}\theta^{V_3,\psi_3^0}_{n}(\tilde{w}_4\tilde{w}_2 r_1(1,g))\,dv\,dg.
\end{align*}
We claim that this integral vanishes. To see this, note that we may factor the unipotent radical of $Q_{2,m-2}$ $$V_{2, m-2} = N\cdot U_{(2,m-2)},$$
where $U_{(2,m-2)}$ is the unipotent radical of the maximal parabolic subgroup $P_{(2,m-2)}$ of  $\Sp_{2m}$. Conjugating the integration along $U_{(2,m-2)}$ across $\tilde{w}_4\tilde{w}_2 r_1$, we see that $\theta_{2n}^{V_3,\psi_3^0}$ is invariant under this subgroup. Thus, we obtain an inner integration of the form 
\[
\int_{[U_{(2,m-2)}]}\overline{\varphi(ug)}\,dg,
\]
which vanishes by cuspidality. This completes the case of $l=2$.

For $l\geq 3$, note that the root subgroup associated to the long root $\mu_2$ is contained in $R_l^1$. Doubling the integration along this root group, we may write our integral as 
\[
\int_{Q_{1,m-1}^0(K)\backslash\Sp_{2m}(\A)}\int_{[R^1_l]}\int_{[\A]}\overline{\varphi(g)}\theta^{V_2,\psi_2^0}_{n}(\tilde{w}_2 r_1x_{\mu_2}(t)(v,g))\,dt\,dv\,dg,
\]
and conjugating past $\tilde{w}_2r_1$, we obtain
\begin{align*}
\int_{Q_{1,m-1}^0(K)\backslash\Sp_{2m}(\A)}\int_{[R^1_l]}\int_{[Z(H_3)]}\overline{\varphi(g)}\theta^{V_2,\psi_2^0}_{n}(z\tilde{w}_2 r_1(v,g))dz\,dv\,dg.
\end{align*}
Therefore, we may expand $\int_{[Z(H_3)]}\theta^{V_2,\psi_2^0}(z\tilde{w}_2 r_1(v,g))dz$ along $[H_3/Z(H_3)]$. 
As before, we obtain two orbits under the action of the subgroup $\Sp_{2(m+n-3)}(K)$. Again, we see that cuspidality kills off the trivial orbit, and we are left to analyze the nontrivial orbit.

 This process is inductive, and the final case will depend on the relation between $m$ and $n$. In both cases, one is able to decompose the integration so that we have an inner integration along an appropriate unipotent radical of $\Sp_{2m}$. After conjugating across the accrued Weyl group elements and unipotent elements, the period along $\theta_{n}$ is invariant under the resulting subgroup. We obtain the constant term of $\varphi$ as an inner integration, and this vanishes by the cuspidality assumption. 

Specifically, if $m<n$, then in the finale we obtain the integral
\begin{align*}\label{eqn: m<n}
\int_{U_m(F)\backslash\Sp_{2m}(\A)}\int_{[R_l^t]}\overline{\varphi(g)}\theta^{V_{2m},\psi_{2m}^0}(\tilde{w}r(v,g))\,dv\,dg.
\end{align*}
Here $U_m$ is the unipotent radical of the Borel subgroup of $\Sp_{2m}$, $\tilde{w}\in W(\Sp(W_n))$ is an explicit Weyl group element, $r$ is an explicit unipotent element, and $R_l^t$ is a subgroup of $R_l$. To see that this integral vanishes, we decompose the integration to obtain an inner integration over $U_m(F)\backslash U_m(\A)$. One may then show that $\theta^{V_{2m},\psi_{2m}^0}$ is invariant under $(\tilde{w}r)U_m(\A)(\tilde{w}r)^{-1}$, so that after a change of variables we obtain
\[
\int_{[U_m]}\overline{\varphi(ug)}du
\]
as an inner integral. This vanishes by cuspidality. 

The case of $m>n$ is similar. This final analysis is similar to the considerations in \cite{BFG2}. We omit the final details, as the form of the argument above deals with all potential issues. This completes the proof of Theorem \ref{Thm: cuspidality}.
\end{proof}

\begin{Rem}
Consider for the moment the classical theta lifting between the metaplectic double cover $\Mp(W)$ and special orthogonal groups $\SO(2r+1)$. If $\pi$ is a genuine cuspidal automorphic representation of $\Mp(W)$ which lifts non-trivially to $\SO(2r+1)$, then the Rallis inner product formula may be used to show that the lift is square integrable. 

We suspect that a similar result holds in our present case, and that the analysis of the next section would then be able to identify such lifts with certain residual representations. Indeed, our identification of the Satake parameters of the local unramified lift in Section \ref{Section: local correspondence} enables us to predict the precise location of the poles of the corresponding Eisenstein series. Such an analysis will entail the study of a generalized doubling integral, which we hope to return to in future work.
\end{Rem}

\section{CAP Representations and Arthur Parameters}\label{Section: CAP}
In this section, we draw certain conclusions from the previous sections. We also relate our results to Arthur's conjectures.
\subsection{CAP representations} We say that two irreducible automorphic representations $\pi=\otimes'_\nu\pi_\nu$ and $\sigma=\otimes'_\nu\sigma_\nu$ are nearly equivalent if $\pi_\nu\cong\sigma_\nu \mbox{    for almost all places }\nu.$
Recall the definition of a CAP representation of a reductive group $\G$ over $K$. 
\begin{Def}
A cuspidal representation $\pi=\otimes'_\nu\pi_\nu$ of $\G(\A)$ is said to be {\bf CAP} (or, cuspidal associated to parabolic) if there exists
\begin{enumerate}
\item
a parabolic subgroup $P=MN$ of $\G$,
\item a cuspidal unitary representation $\tau$ of $M(\A)$,
\item an unramified character $\chi$ of $M(\A)$,
\end{enumerate}
such that $\pi$ is nearly equivalent to a irreducible constituent of $\Ind_P^G(\tau\otimes \chi)$. Equivalently, $\pi_\nu$ is isomorphic to the Langlands quotient of $\Ind_P^G(\tau_\nu\otimes\chi_\nu)$ for almost all $\nu$. In this case, we say that $\pi$ is CAP with respect to $(P,\tau,\chi)$.
\end{Def}
In the case of $\G$ quasi-split, Arthur's conjecture on square-integrable automorphic forms produces a precise determination of those triples $(P,\tau,\chi)$ for which CAP representations exist. 

For covering groups, there is an obvious generalization of the above definition, where the parabolic $P=MN$ is replaced with the full inverse image $\overline{P}=\overline{M}N$. This concept has been studied previously in the case of $\Mp_{2n}$, the two-fold metaplectic cover of $\Sp_{2n}$ (see, for example, \cite{Y}).

Recall that if $\pi$ is a cuspidal representation of $\overline{\Sp}(W_0)$, then we denote by $n(\pi)$ the index such that $\Theta_{n(\pi)}(\pi)\neq0$, while $\Theta_{n(\pi)-1}(\pi) =0$. By Theorem \ref{Thm: Non-vanishing} such an index exists and is unique.% by Proposition \ref{Prop: inductive}. 
Set $t(\pi) = n(\pi)-n.$ 

\begin{Thm}\label{Thm: CAP reps}
Suppose $\pi$ is a genuine cuspidal automorphic representation of $\overline{\Sp}(W_0)$, and suppose $\Theta(\pi):= \Theta_{n(\pi)}(\pi)$ is the first nontrivial lift. Then $$\Theta(\pi) = \hat{\bigoplus_{\lambda\in \Lambda}}\tau_\lambda$$ is semi-simple, and the irreducible summands all lie in the same near equivalence class. 
\begin{enumerate}
\item If $t(\pi) \geq 0$, then each $\tau_\lam$ is CAP with respect to the triple $(Q_{t(\pi),n},\pi, \overline{\chi}_{\Theta,2t(\pi)})$.
Here, we have $\overline{\chi}_{\Theta,2t(\pi)} = \overline{\chi}^0\chi_{\Theta,2t(\pi)}$
where  
\[
\chi_{\Theta,2t(\pi)}((M(t_1,\ldots,t_{t(\pi)}),m),\zeta) = \prod_{i=1}^{t(\pi)}|t_i|^{(2(t(\pi)-i)+1)/4}.
\]
\item If $t(\pi) <0$, then $\pi$ is  CAP with respect to $(Q_{-t(\pi),n(\pi)},\tau_\lam,\overline{\chi}_{\theta,-2t(\pi)})$ for any $\lam \in \Lam$.
\end{enumerate}
\end{Thm}

\begin{proof}
By Theorem \ref{Thm: cuspidality}, we know that $\Theta(\pi)\hra L^2_{cusp}(\Sp(V_{n(\pi)})_F\backslash \overline{\Sp}(V_{n(\pi)})_\A)$ is cuspidal. It is therefore semi-simple, and we may write
\[
\Theta(\pi) = \hat{\bigoplus_{\lambda\in \Lambda}}\tau_\lambda,
\]
where for each $\lambda\in \Lambda$, $\tau_\lambda$ is an irreducible cuspidal automorphic representation of $\overline{\Sp}(V_{n(\pi)})_\A$.
Fix an arbitrary $\lambda$. Using the Petersson inner product, we see then that
\[
\Hom_{\overline{\Sp}(V_{n(\pi)})_\A}(\tau_\lam^\vee\otimes\Theta(\pi),\cc)\neq 0,
\]
so that we obtain a nontrivial global trilinear form
\[
\Hom_{\overline{\Sp}(W_0)_\A\times_{\mu_4}\overline{\Sp}(V_n)_\A}(\tau_\lam^\vee\otimes\pi\otimes\Theta_{n(\pi)},\cc)\neq0.
\]
Since all three representations are restricted tensor products, we obtain at each place $\nu$ a non-zero local trilinear form. That is,
\[
\Hom_{\overline{\Sp}(W_0)_\A\times_{\mu_4}\overline{\Sp}(V_n)_\A}(\tau_{\lam,\nu}^\vee\otimes\pi_\nu\otimes\Theta_{n(\pi),\nu},\cc)\neq0.
\]
This implies, for any place $\nu$, the corresponding local hom-space is non-zero:
\[
\Hom_{\overline{\Sp}(W_0)_{K_\nu}\times_{\mu_4}\overline{\Sp}(V_n)_{K_\nu}}(\Theta_{n(\pi),\nu},\tau_{\lam,\nu}\otimes\pi^\vee_\nu)\neq0,
\]
Suppose $\nu$ is a finite place such that all three local representations are unramified. This holds for all but finitely many places. Thus, we may apply Theorem \ref{Thm: local correspondence} to conclude the appropriate relationship between the local factors.
\end{proof}
This constitutes the first construction of CAP representations on higher degree covering groups. In Section \ref{Section: Whittaker coefficients}, we provide evidence based on the dimension equation that $t(\pi) >0$ holds if $\pi$ is generic.

\subsection{Conjectural Shimura Lift}\label{subsection: Shimura}
Consider a finite place $\nu$ of $K$ such that the completion $F=K_\nu$ has odd residue characteristic. Recall that to the $4$-fold cover $\overline{\Sp}_{2r}$ one may associate the dual root datum

\[
(Y_{Q,4}, \{\al_{Q,4}^\vee\}, \Hom(Y_{Q,4},\zz), \{\al_{Q,4}\}),
\]
which corresponds to the complex dual group $\overline{\Sp}_{2r}^\vee =\Sp_{2r}(\cc)$. To this root datum, we may also attach the split linear algebraic group over $F$, $\mathbb{G}_{Q,4} =\SO_{2r+1}$. Gan-Gao \cite{GG} speculate that $\mathbb{G}_{Q,4}$ should be viewed as the principal endoscopic group of $\overline{\Sp}_{2r}$. Set $\mathbb{G}_{Q,4}(F) = \G_{Q,4}$.

Let $T_{Q,4}=Y_{Q,4}\otimes F^\times\subset \G_{Q,4}$ be the induced split maximal torus. Then the inclusion $Y_{Q,4}\hra Y$ induces an isogeny
\[
i: T_{Q,4} \to T_r.
\] Let $\chi:i(T_{Q,4})\to \cc^\times$ be a character of the subtorus $i(T_{Q,4})$ of $T_r$. For example, $\chi$ could be an unramified character arising from a character of the lattice $Y_{Q,4}$. Upon choosing a distinguished character $\overline{\chi}^0$, we previously saw how to obtain the principal series representation $$\Ind_{\overline{B_r}}^{\overline{\Sp}_{2r}(F)}(\overline{\chi}^0\chi).$$
 As noted in \cite[Section 15]{GG}, this choice induces a well-defined $W$-equivariant lifting 
\begin{equation}\label{eqn: unram lift}
\Ind_{\overline{B_r}}^{\overline{\Sp}_{2r}(F)}(\overline{\chi}^0\chi)\mapsto\Ind_{B_{Q,4}}^{G_{Q,4}}(\chi\circ i).
\end{equation}
This lifting provides a bijection between principal series representations of $\overline{\Sp}_{2r}(F)$ and principal series representations of $\G_{Q,4}$ induced from characters $\chi'$ such that $\ker(i)\subset \ker(\chi')$.

One may check that in the case of $\chi$ unramified, this lifting has the effect of squaring the Satake parameter. More explicitly, if 
\[
\chi(M(t_1,\ldots,t_r)) =\prod_{i=1}^r\chi_i(t_i),
\]
 and if we parametrize $T_{Q,4}\subset \G_{Q,4}$ by $D(s_1,\ldots,s_r) = \diag(s_1,\ldots,s_r,1,s_r^{-1},\ldots,s_1^{-1}),$ then
\[
\chi\circ i(D(s_1,\ldots,s_r)) = \prod_{i=1}^r\chi_i^2(s_i).
\]
A similar unramified correspondence exists for $\Mp_{2r}$ and $\SO_{2r+1}$, as follows from \cite{K}, Theorem 2.5. In the classical Shimura correspondence between $\Mp_2$ and $\PGL_2$, this is made explicit in \cite[Sec. 2.17]{Gan2}. With these results in mind, one is led to the (naive) conjecture of a generalized Shimura-type correspondence (dependent on our distinguished character $\overline{\chi}^0$):
\[
Shim_{\overline{\chi}^0}:\Rep_{\ep}(\overline{\Sp}_{2r}(F))\longrightarrow \Rep(\G_{Q,4}),
\]
of which this is the restriction to the principal series. One might also hope for an analogous global correspondence
\begin{align}\label{eqn: Shim}
Shim_{\overline{\chi}^0}: \mathcal{A}_{\ep}(\overline{\Sp}_{2r}(\A))\longrightarrow \mathcal{A}(\G_{Q,4}(\A))
\end{align}
 of (genuine) automorphic representations. For example, the classical dependence upon a choice of additive character $\psi$ is here encoded in the choice of distinguished character $\overline{\chi}^0$. Further evidence for such a conjecture may be seen in the metaplectic correspondences known for covers of general linear groups (see \cite{F} and \cite{FK}).

\subsection{Arthur's conjectures}\label{subsection: Arthur}
For a reductive linear algebraic group $\G$ over a number field $K$, Arthur has conjectured a decomposition 
\[
L^2_{disc}(\G(K)\backslash \G( \A)) = \hat{\bigoplus_\Psi}L^2_\Psi,
\]
where the Hilbert space direct sum runs over equivalence classes of global Arthur parameters $\Psi$. These are continuous maps
\[
\Psi : L_K \times \SL_2(\cc)\to \G^\vee, 
\]
satisfying certain axioms. Here $L_K$ is the conjectural Langlands group of $K$. For each $\Psi$, $L^2_\Psi$ is a direct sum of near equivalent representations. For a summary of Arthur's conjecture, see \cite{A}.

The primary motivation for considering the conjectural lifting (\ref{eqn: Shim}) is the hope of extending the notion of Arthur parameters from $\G_{Q,4}$ to $\overline{\Sp}_{2r}$.  This has only recently been accomplished by Gan-Ichino  \cite{GI} in the case of $\Mp_{2r}$. The proof uses the local Shimura correspondence between $\Mp_{2r}$ and $\SO_{2r+1}$ along with the global theta liftings from $\Mp_{2r}$ to $\SO_{2(r+k)+1}$ for $k>r$. This indicates that the conjectural correspondence (\ref{eqn: Shim}), even if it exists, is likely insufficient to prove a transfer of definition to $\overline{\Sp}_{2r}$. Likely further development of the trace formula in the context of BD-covering groups is necessary. Such investigations are underway in the work of Li.
 
  In any case, given an Arthur parameter $\Psi$ for $\SO_{2r+1}$, the bijection above enables us to define a near equivalence class $\mathcal{A}_{\Psi,\overline{\chi}^0}$ of genuine representations for $\overline{\Sp}_{2r}(\A)$ as follows.
To the global parameter $\Psi$, we have a family of local parameters
\[
\Psi_\nu : L_{K_\nu}\times \SL_2(\cc)\to \Sp_{2r}(\cc),
\]  
where if $\nu$ is a finite place, then $L_{K_\nu} = WD_{K_\nu}$ is the Weil-Deligne group of $K_\nu$. For almost all finite places, we obtain a Satake parameter
\[
s_{\Psi_\nu} = \Psi_\nu\left( Fr_\nu, \left(\begin{array}{cc}q_\nu^{1/2}&\\&q_\nu^{-1/2}\end{array}\right)\right),
\]
where $Fr_\nu$ is a Frobenius element and $q_\nu$ is the size of the residue field of $K_\nu$. This corresponds to an irreducible spherical representation $\pi_{\Psi_\nu}$ of $\G_{Q,4}(K_\nu)=\SO_{2r+1}(K_\nu)$. Using the unramified lifting (\ref{eqn: unram lift}), we obtain an unramified representation of $\overline{\Sp}_{2r}(K_\nu)$, $\pi_{\Psi_\nu,\overline{\chi}^0}$. Thus, $\Psi$ determines the near equivalence class 
\[
\mathcal{A}_{\Psi,\overline{\chi}^0}:=\left\{\pi=\otimes'_\nu \pi_\nu : \pi_\nu \cong \pi_{\Psi_\nu,\overline{\chi}^0} \mbox{ for almost all } \nu\right\}.
\]
We are led to the following:
  \begin{Conj} 
   For a choice of distinguished character $\overline{\chi}^0$, 
 \begin{align}\label{eqn: conj}
 \mathcal{A}_\ep(\overline{\Sp}_{2r}(\A)) = \hat{\bigoplus_\Psi}\mathcal{A}_{\Psi,\overline{\chi}^0}.
 \end{align}
 \end{Conj}
 This is analogous to Theorem 2.3 in \cite{GL}.

In the remainder of this section, we assume the existence of the decomposition (\ref{eqn: conj}) and use it to interpret our construction in terms of Arthur parameters of $\SO_{2r+1}(\A)$. Let $\pi$ be a genuine cuspidal representation of $\overline{\Sp}_{2r}(\A)$, and suppose $\Theta(\pi):= \Theta_{r+k}(\pi)$ is the first nontrivial lift. Suppose that $\pi \in \mathcal{A}_{\Psi,\overline{\chi}^0},$ and consider the parameter
\[
\Psi_\pi : L_K\times\SL_2(\cc) \to \Sp_{2r}(\cc).
\]

If $\{\chi_{\nu,i}^{\pm1}(\varpi_\nu)\}$ are the Satake parameters of $\pi$ at $\nu$ relative to $\overline{\chi}^0$, then Theorem \ref{Thm: local correspondence} tells us that the corresponding parameters for an irreducible summand $\tau_\pi$ of $\Theta(\pi)$ are given by 
\[
\{\chi_{\nu,i}^{\pm1}(\varpi_\nu)\} \cup\{ {q_\nu}^{(2(k-i)+1)/4} : 1\leq i \leq k\}.
\]
Passing through the unramified lifting (\ref{eqn: unram lift}), we see that, up to action by the Weyl group,
\[
s_{\Psi_{\pi,\nu}} = M(\chi^2_{\nu,1}(\varpi_\nu),\ldots, \chi^2_{\nu,r}(\varpi_\nu)) \in T_r(\cc).
\]
Thus, the parameter of the lift $\Psi_{\Theta(\pi)}$ must satisfy
\[
s_{\Psi_{\Theta(\pi),\nu}} = M({q_\nu}^{(2k-1)/2}, \ldots, {q_\nu}^{1/2},\chi^2_{\nu,1}(\varpi_\nu),\ldots, \chi^2_{\nu,r}(\varpi_\nu)) \in T_{r+k}(\cc),
\]
for almost all $\nu$.

Consider the unipotent orbit $\calo_{k} =((2k)1^{2r})$ of $\Sp_{2(r+k)}(\cc)$. This corresponds to a conjugacy class of morphisms $\phi_k:\SL_2(\cc) \to \Sp_{2(r+k)}(\cc)$. Then a simple computation based on the above requirements shows that the Arthur parameter $\Psi_{\Theta(\pi)}$ associated to the near equivalence class in which $\Theta(\pi)$ sits should be the composition

\begin{eqnarray*}
\begin{CD}
L_K\times\SL_2(\cc)@>Id\times \Delta>> L_K\times\SL_2(\cc)\times \SL_2(\cc)@>\Psi_\pi\times Id>>\Sp_{2r}(\cc)\times\SL_2(\cc)@>\iota\times \phi_k>>\Sp_{2(r+k)}(\cc),\\
\end{CD}
\end{eqnarray*}
where $\iota: \Sp_{2r}(\cc) \to \Sp_{2(r+k)}(\cc)$ has image in the centralizer of the image of $\phi_k$. In the notation of \cite{A1}, we have
\[
\Psi_{\Theta(\pi)}=\Psi_\pi\boxplus S_{2k},
\]
where $S_{2k}$ is the unique $2k$-dimensional irreducible representation of $\SL_2(\cc)$.
In particular, the Arthur parameter $\Psi_\Theta$ is non-tempered, in agreement with the representation $\Theta(\pi)$ being a CAP representation.
%%%%%%%%%%%%%%%%%%%%%%%%%%%%%%%%%%%%%%%%%%%%%%%%%%%%%%%%%%%%%%%%%%%%%%%%%%%%%%%%%%%%%%%%%%%%%%%%%%%%%%%%%%%%%%%%%%%%%%%%%%%%%%%%%%%%%%%%%%%%%%%%%%%%%%%%%%%%%%%%%%%%%%%%%%%%%%%%%%%%%%%%%%%%%%%%%%%%%%%%%%%%%%%%%

\section{Whittaker Coefficients and the Dimension Equation}\label{Section: Whittaker coefficients}
In this final section, we investigate the case that $\pi$ is assumed to be a generic cuspidal representation of $\overline{\Sp}(W_0)_\A\cong \overline{\Sp}_{2m}(\A)$; that is, we assume $\pi$ has a non-trivial Whittaker model. With this assumption, one expects that $n(\pi) = m+1$. Moreover, we anticipate that $\Theta_{m+1}(\pi)$ is a generic representation of $\overline{\Sp}(V_{m+1})_\A\cong\overline{\Sp}_{2m+2}(\A)$.

 The reason is the dimension equation of Ginzburg (see \cite{G3}), which we recall now.
Suppose $\pi$ is an automorphic representation of a reductive group or finite degree cover there of $\G$, and assume $\calo(\pi)$ is a single nilpotent orbit. Recall that the Gelfand-Kirillov dimension of $\pi$ is given by
\[
\dim(\pi) = \frac{1}{2}\dim_\cc(\calo(\pi)).
\]

The general philosophy of Ginzburg is that one may utilize the dimension equation to anticipate when a lifting such as the one under consideration is non-vanishing. In our case, the equation amounts to
\begin{align}\label{eqn: dimension equation}
\dim(\pi) +\dim(\Theta_{W_n}) = \dim_\cc\left(\Sp_{2m}(\cc)\right) + \dim(\Theta_n(\pi)).
\end{align}
Our computation of $\calo(\Theta_n)$ in Theorem \ref{Thm: intro Nilpotent orbit} tells us that $$\dim(\Theta_{W_n}) = \frac{(n+m+1)(n+m)}{2}.$$ If we assume $\pi$ is generic, so that $\dim(\pi) = m^2$, we find equality when $n=m+1$ and $\Theta_{m+1}(\pi)$ is generic.

For the remainder of this section, we compute the Whittaker coefficient of a vector in $\Theta_{m+1}(\pi)$ and showing that it may be expressed in terms of an integral pairing of the Whittaker coefficients of $\pi$ and the split coefficient (see Proposition \ref{Prop: Split coefficient}) of $\Theta_{m+1}$. Let $\psi_{n,\al}: U_n(K)\backslash U_n(\A)\to \cc^\times$ be the Gelfand-Graev character of the full unipotent radical of the Borel subgroup of $\Sp_{2n}$ corresponding to the square class $\al\in K^\times/(K^\times)^2$. Thus,
\[
\psi_{n,\al}(u) =\psi(u_{1,2}+u_{2,3}+\cdots+u_{n-1,n}+\al u_{n,n+1}).
\]
For an automorphic form $\varphi$ of $\overline{\Sp}_{2n}(\A)$, we set $$W_\al(\varphi)(g)=\int_{[U_n]}\varphi(ug)\psi_{n,\al}(u)du$$ to be the Whittaker coefficient of $\varphi$ with respect to this character.
\begin{Prop}\label{Prop: Whittaker}
Let $\pi$ be an irreducible cuspidal generic representation of $\overline{\Sp}_{2m}(\A)$, and let $\varphi\in \pi$. For a vector $\theta_{m+1}\in \Theta_{W_{m+1}}$, consider the pairing
\[
F_{\varphi,\theta}(h) = \displaystyle \int_{[\Sp_{2m}]}\overline{\varphi(g)}\theta_{m+1}((h,g))\,dg,
\]
where $h\in \overline{\Sp}_{2m+2}(\A)$. Then we have
\begin{align}\label{eqn: period identity}
W_\al(F_{\varphi,\theta})(1) = \int_{U_m(\A)\backslash\Sp_{2m}(\A)} W_\al(\overline{\varphi})(g)\theta_{m+1}^{U_{2m+1},\psi_{2m+1,\al}^0}(g)\,dg.
\end{align}
\end{Prop}
\begin{proof}
To prove this, we must compute the integral
\begin{align}\label{eqn: Whittaker}
\displaystyle \int_{[\Sp_{2m}]}\int_{[U_{m+1}]}\overline{\varphi(g)}\theta_{m+1}(u,g)\psi_{m+1,\al}(u)du\,dg.
\end{align}
The computation is similar to the one seen in the proof of Theorem \ref{Thm: cuspidality}. One expands $\theta_{m+1}$ iteratively along embedded Heisenberg groups, and reduces the integral to a single orbit under the action of an embedded symplectic group of smaller rank. We omit the details.
\end{proof}
%%%%%%%%%%%%%%%%%%%%%%%%%%%%%%%%%%%%%%%%%%%%%%%%%%%%%%%%%%%%%%%%%%%%%%%%%%%%%%%%%%%%%%%%%%%%%%%%%%%%%%%%%%%%%%%%%%%%%%%%%%%%%%%%%%%%%%%%%%%%%%%%%%%%%%%%%%%%%%%%%%%%%%%%%%%%%%%%%%%%%%%%%%%%%%%%%%%%%%%%%%%%%%%%%%%%%%%%%%%%%%%%%%%%%%%%%%%%%%%%%%%%%%%%%%%%%%%%%%%%%%%%%%%%%%%%%%%%%%%%%%%%%%%%%%%%%%%%%%%%%%%%%%%%%%%%%%%%%%%%%%%%%%%%%%%%%%%%%%%%%%
%%%%%%%%%%%%%%%%%%%%%%%%%%%%%%%%%%%%%%%%%%%%%%%%%%%%%%%%%%%%%%%%%%%%%%%%%%%%%%%%%%%%%%%%%%%%%%%%%%%%%%%%%%%%%%%%%%%%%%%%%%%%%%%%%%%%%%%%%%%%%%%%%%%%%%%%%%%%%%%%%%%%%%%%%%%%%%%%%%%%%%%%%%%%%%%%%%%%%%%%%%%%%%%%%%%%%%%%%%%%%%%%%%%%%%%%%%
From this proposition as well as Proposition \ref{Prop: Split coefficient}, we may immediately conclude the following.
\begin{Cor}
Fix $\al\in K^\times/(K^\times)^2$. If $\Theta_{m+1}(\pi)$ is generic with respect to the $\al$-coefficient,  then so is $\pi$.
\end{Cor}

%\newpage
\appendix
\section{Proof of the Filtration}\label{Appendix A}
In this appendix, we sketch the proof of Proposition \ref{Prop: filtration}. 
\begin{Lem}\label{Lem: SES}
Suppose $\pi$ is a smooth representation of $\Sp_{2r}(F)$ or a finite central extension thereof. If $U_{\al_0}$ denotes the root subgroup for the highest root $\al_0$, then there is a short exact sequence
\begin{align}\label{eqn: SES}
0 \longrightarrow \ind_{\overline{Q}_{2,r-2}'}^{\overline{Q}_{1,r-1}}\left(J_{U_{1,r-1},\psi_1}(\pi)\right) \longrightarrow J_{U_{\al_0}}(\pi)\longrightarrow J_{U_{1,r-1}}(\pi)\longrightarrow 0.
\end{align}
Here $Q_{2,r-2}'$ indicates the subgroup of $Q_{2,r-2}$ containing only $\GL_1^\Delta\times \Sp_{2(r-2)}\subset \GL_1^2\times \Sp_{2(r-2)}$.
\end{Lem}
\begin{proof}
This follows in a similar fashion to \cite[Prop. 5.12 (d)]{BZ2} in which the case of $\GL_n(F)$ is handled. The only difference is the need in the symplectic case to study the Jacquet module $J_{U_{\al_0}}(\pi)$ due to the fact that the unipotent radical of $Q_{1,r-1}$ is a Heisenberg group with center $U_{\al_0}$ rather than abelian. 

After quotienting out by the center of the unipotent radical, the analysis is similar to the general linear case. We omit the details.
\end{proof}
We now apply the Jacquet functor $J_{U_{1,n-1}}(\cdot)$ to the short exact sequence (\ref{eqn: SES}) in the case that $\pi = \Theta_{2(m+n)}$. Note that, under the our chosen embedding, $U_{1,n-1}$ contains the root group for the highest root of $\Sp_{2(n+m)}$. Using the exactnesses of the Jacquet functor, we have the short exact sequence
\begin{align}\label{eqn: theta SES}
0 \longrightarrow V(\Theta,\psi_1)\longrightarrow J_{U_{1,n-1}}(\Theta_{2(m+n)})\longrightarrow J_{U_{1,n+m-1}}(\Theta_{2(m+n)})\longrightarrow 0
\end{align}
where
\[
V(\Theta,\psi_1)=J_{U_{1,n-1}}\left(\ind_{\overline{Q}_{2,n+m-2}'}^{\overline{Q}_{1,n+m-1}}\left(J_{U_{1,n+m-1},\psi_1}(\Theta_{2(m+n)})\right)\right).
\]
 This completes the proof of Proposition \ref{Prop: filtration}. This lemma also shows us why attempting to study the analogous restriction question in the case of the metaplectic group $\Mp_{2r}$ with the classical theta representation $\omega_{2r,\psi}$ will not work: the twisted Jacquet modules 
\[
J_{U_{1,r-1},\psi_1}(\omega_{2r,\psi})
\]
which arise in the short exact sequence vanish due to the minimality of $\omega_{2r,\psi}$. This implies that the analogous local restriction question forces unique parameters for the characters $\chi$ and $\mu$. Of course, this is due to the isomorphism
\[
\omega_{2r,\psi}\bigg|_{\Mp_{2k}\times\Mp_{2r-2k}}\cong \omega_{2k,\psi}\otimes\omega_{2r-2k,\psi}.
\]

\section{Technical Global Lemmas}\label{Appendix B}
In this section, we gather a few technical lemmas which are used in Section \ref{Section: cuspidality}. Note that while everything we state here is global, the arguments in Section \ref{subsection: local vanishing} allow us to conclude the analogous local results at almost all places. Let $\Theta_{2r}$ be the theta representation on $\overline{\Sp}_{2r}(\A)$.
\begin{Lem}\label{Lem: lemma 1}
Let $\theta\in \Theta_{2r}$ and let $2\leq k\leq r$. Then the integrals
\[
\int_{[V_{k,r-k}]}\theta(vg)\psi_k(v)\,dv
\]
vanish for all $g\in\GA$. Here, recall that
\[
\psi_k(v) = \psi(v_{1,2}+v_{2,3}+\cdots + v_{k,k+1}).
\]
\end{Lem}
\begin{proof}
Note that the case of $k=r$ is clear as the above integral corresponds to a Whittaker coefficient of $\Theta_{2r}$ (corresponding to the nilpotent orbit $(2r)$), which vanishes by Theorem \ref{Thm: vanishing}.

We shall induct in $m=r-k$, the above giving the base case. Suppose now that the lemma is true for $m$, and consider $m+1=r-k$. 

Expanding along the root subgroup corresponding to the long root $\mu_{k+1}$, we see that the nontrivial Fourier coefficients correspond to the nilpotent orbit $((2k)1^{2(r-k)})$, which vanishes by Theorem \ref{Thm: vanishing}. Thus, we have that
\[
\int_{[V_{k,r-k}]}\theta(vg)\psi_k(v)\,dv = \int_{[U_{\mu_{k+1}}V_{k,r-k}]}\theta(vg)\psi_k(v)\,dv.
\]
We have integrated over the center of the Heisenberg group $H_{k+1}$, which enables us to expand the above integral along $[H_{k+1}/Z(H_{k+1})]\cong (F\backslash\A)^{2(r-k+1)}$. The subgroup $\Sp_{2(r-k)}(F)$ acts by conjugation on the dual of this quotient with two orbits. A representative of the nontrivial orbit is of the form
\[
\int_{[H_{k+1}/Z(H_{k+1})]}\int_{[U_{\mu_{k+1}}V_{k,r-k}]}\theta(vg)\psi_k(v)\psi(v_{k+1,k+2})\,dv = \int_{[V_{k+1,r-k-1}]}\theta(v'g)\psi_{k+1}(v')\,dv',
\]
which vanishes by induction. Therefore, setting $\psi_{k+1}^1:V_{k+1}\to\cc^\times$ to be the trivial extension of $\psi_k$ to $V_{k+1,r-k-1}$, we have that
\[
\int_{[V_k]}\theta(vg)\psi_k(v)\,dv =\int_{[V_{k+1,r-k-1}]}\theta(v'g)\psi_{k+1}^1(v')\,dv'.
\]
This last integral may be decomposed with inner integration the constant term along $U_{(k,r-k)}$. By Theorem \ref{Thm: global constant term}, this constant term is a theta function on a smaller symplectic group times an integral of the form
\[
\int_{[U_{\GL_{k+1}}]}\theta_{\GL_{k+1}}(uh)\psi_{gen}(u)du,
\]
where $U_{\GL_{k+1}}$ is the unipotent radical of the Borel subgroup of $\GL_{k+1}$, $\theta_{\GL_{k+1}}\in \Theta_{\GL_{k+1}}^{(2)}$, and
\[
\psi_{gen}(u) = \psi(u_{1,2}+\cdots + u_{k,k+1}).
\]
We see that this integral is a Whittaker coefficient of $\theta_{\GL_{k+1}}$, which vanishes as $k\geq 2$ (see \cite{KP}). Thus,
\[
\int_{[V_{k,r-k}]}\theta(vg)\psi_k(v)\,dv =0.\qedhere
\]

\end{proof}
For the next lemma, we simplify the notation by setting $V_{2k-1}=V_{2k-1,r-2k+1}$. 
\begin{Lem}\label{Lem: lemma 2}
Let $\theta\in \Theta_{2r}$ and let $1\leq k\leq r$. The function
\[
\theta^{V_{2k-1},\psi^0_{2k-1}}(g) = \int_{[V_{2k-1}]}\theta(vg)\psi^0_{2k-1}(v)\,dv
\]
is invariant under $H_{2k}$.
\end{Lem}
\begin{proof}
We expand the function $f_k(g)=\theta^{V_{2k-1},\psi^0_{2k-1}}(g) $ along the long root group $U_{\mu_{2k}}$:
\[
f_k(g) = \sum_{\xi\in F}\int_{[\A]}f_1(x_{\mu_{2k}}(t)g)\psi(\xi t)\,dt.
\]
For $\xi\neq 0$, we see that this integral may be decomposed into the constant term with respect to $U_{(2k-2,r-2k+2)}$ followed by a Fourier coefficient of type $(41^{2(r-2k)-2})$ on $\Theta_{2(r-2k+2)}$, which vanishes by Theorem \ref{Thm: vanishing}. Thus, we have that
\[
f_k(g) = \int_{[U_{\mu_{2k}}V_{2k-1}]}\theta(vg)\psi^0_{2k-1}(v)\,dv.
\]
We now expand along $[H_{2k}/Z(H_{2k})]\cong (F\backslash\A)^{2(r-2k)}$. The subgroup $\Sp_{2(r-2k)}(F)$ acts by conjugation on the dual of this quotient with two orbits. A representative of the nontrivial orbit is of the form
\[
\int_{[V_{2k}]}\theta(vg)\psi^0_{2k-1}(v)\psi(v_{2k,2k+1})\,dv, 
\]
which may also be decomposed as a constant term along $U_{(2k-2,r-2k+2)}$. Applying Theorem \ref{Thm: global constant term}, we obtain a constant times an integral of the form  
\[
\displaystyle\int_{[V_{2,(r-2k)-2}]}\theta_{2(r-2k)}(v'g')\psi_2(v')\,dv',
\]
 where $\theta_{2(r-2k)}\in \Theta_{2(r-2k)}$. This vanishes by the previous lemma.
Thus, we are left with the constant term and we obtain
\[
f_k(g) = \int_{[H_{2k}]}f(hg)dh.\qedhere
\]
\end{proof}

%%%%%%%%%%%%%%%%%%%%%%%%%%%%%%%%%%%%%%%%%%%%%%%%%%%%%%%%%%%%%%%%%%%%%%%%%%%%%%%%%%%%%%%%%%%%%%%%%%%%%%%%%%%%%%%%%%%%%%%%%%%%%%%%%%%%%%%%%%%%%%%%%%%%%%

%Citation: \cite{BG}.

%\subsection*{Acknowledgments}

%Acknowledgments

%%%%%%%%%%%%%%%%%%%%%%%%%%%%%%%%%%%%%%%%%%
%-------------This is the end of the paper
%%%%%%%%%%%%%%%%%%%%%%%%%%%%%%%%%%%%%%%%%%

%--------references
%\newpage
\bibliographystyle{alpha}

%--------------Need to modify the references in the file: mybib.bib see example above.
\bibliography{mybib}

\end{document}